\documentclass[11pt]{article}

\usepackage[margin=1in]{geometry}
\usepackage{setspace}
\usepackage{amsmath,amsfonts,amssymb,mathtools}

\usepackage{amsthm}

\usepackage{times}
\usepackage[varg]{txfonts}
\usepackage{bm}

\usepackage{algorithm}
\usepackage[noend]{algorithmic}

\usepackage{graphicx}
\usepackage{tikz}
\usetikzlibrary{arrows.meta,positioning,shapes.geometric,shapes.misc,fit,calc}
\usetikzlibrary{patterns}

\usepackage{booktabs}
\usepackage{multirow}
\usepackage{array}

\usepackage{enumitem}
\usepackage{xcolor}
\usepackage[font=small,labelfont=bf]{caption}
\usepackage{subcaption}
\usepackage{etoolbox}

\usepackage[hidelinks]{hyperref}
\usepackage{natbib}
\allowdisplaybreaks

\theoremstyle{plain}
\newtheorem{theorem}{Theorem}
\newtheorem{proposition}{Proposition}

\newtheorem{corollary}{Corollary}
\theoremstyle{definition}
\newtheorem{definition}{Definition}
\newtheorem{example}{Example}

\theoremstyle{remark}
\newtheorem{remark}{Remark}

\newcommand{\Halmos}{\hfill\ensuremath{\square}}
\newenvironment{proof}[1]
  {\par\medskip\noindent\textit{#1}\ \ignorespaces}
  {\par\medskip}

\newcommand{\SingleSpacedXI}{\singlespacing}

\newcommand{\OneAndAHalfSpacedXII}{\onehalfspacing}

\newcommand{\RUNAUTHOR}[1]{}
\newcommand{\RUNTITLE}[1]{}
\newcommand{\MANUSCRIPTNO}[1]{}

\newcommand{\ECSwitch}{%
  \clearpage
  \setcounter{section}{0}%
  \setcounter{table}{0}%
  \setcounter{figure}{0}%
  \setcounter{equation}{0}%
  \setcounter{theorem}{0}%
  \setcounter{proposition}{0}%
  \setcounter{lemma}{0}%
  \setcounter{corollary}{0}%
  \setcounter{definition}{0}%
  \setcounter{example}{0}%
  \setcounter{assumption}{0}%
  \setcounter{remark}{0}%
  \renewcommand{\thesection}{EC.\arabic{section}}%
  \renewcommand{\thesubsection}{EC.\arabic{section}.\arabic{subsection}}%
  \renewcommand{\thesubsubsection}{EC.\arabic{section}.\arabic{subsection}.\arabic{subsubsection}}%
  \renewcommand{\thetable}{EC.\arabic{table}}%
  \renewcommand{\thefigure}{EC.\arabic{figure}}%
  \renewcommand{\theequation}{EC.\arabic{equation}}%
  \renewcommand{\thetheorem}{EC.\arabic{theorem}}%
  \renewcommand{\theproposition}{EC.\arabic{proposition}}%
  \renewcommand{\thelemma}{EC.\arabic{lemma}}%
  \renewcommand{\thecorollary}{EC.\arabic{corollary}}%
  \renewcommand{\thedefinition}{EC.\arabic{definition}}%
  \renewcommand{\theexample}{EC.\arabic{example}}%
  \renewcommand{\theassumption}{EC.\arabic{assumption}}%
  \renewcommand{\theremark}{EC.\arabic{remark}}%
  \renewcommand{\theHsection}{EC.\arabic{section}}%
  \renewcommand{\theHsubsection}{EC.\arabic{section}.\arabic{subsection}}%
  \renewcommand{\theHsubsubsection}{EC.\arabic{section}.\arabic{subsection}.\arabic{subsubsection}}%
  \renewcommand{\theHparagraph}{EC.\arabic{section}.\arabic{subsection}.\arabic{subsubsection}.\arabic{paragraph}}%
  \renewcommand{\theHtable}{EC.\arabic{table}}%
  \renewcommand{\theHfigure}{EC.\arabic{figure}}%
  \renewcommand{\theHequation}{EC.\arabic{equation}}%
  \renewcommand{\theHtheorem}{EC.\arabic{theorem}}%
  \renewcommand{\theHproposition}{EC.\arabic{proposition}}%
  \renewcommand{\theHlemma}{EC.\arabic{lemma}}%
  \renewcommand{\theHcorollary}{EC.\arabic{corollary}}%
  \renewcommand{\theHdefinition}{EC.\arabic{definition}}%
  \renewcommand{\theHexample}{EC.\arabic{example}}%
  \renewcommand{\theHassumption}{EC.\arabic{assumption}}%
  \renewcommand{\theHremark}{EC.\arabic{remark}}%
}
\newcommand{\ECHead}[1]{%
  \begin{center}
    {\large\bfseries #1}
  \end{center}
  \medskip
}

\makeatletter
\newcommand{\blockhead}[1]{%
  \ifvmode\else\\\@empty\fi
  \noalign{%
    \penalty\postdisplaypenalty\vskip 6pt
    \vbox{\normalbaselines
      \ifdim\linewidth=\columnwidth\else
        \parshape\@ne \@totalleftmargin \linewidth
      \fi
      \noindent\ignorespaces\quad\emph{#1}\par}%
    \penalty\predisplaypenalty\vskip 2pt\relax
  }%
}
\makeatother

\usepackage[disable]{todonotes}
\newcommand{\EB}[1]{}
\newcommand{\RH}[1]{}
\newcommand{\HW}[1]{}

\begin{document}

\title{\Large\bfseries Cooperative Integer Programming Games:\\
Core Stability and Optimal Coalition Structures}

\author{
Hyunwoo Lee\thanks{Grado Department of Industrial and Systems Engineering, Virginia Tech. Email: \texttt{hyunwoolee@vt.edu}} \and
Robert Hildebrand\thanks{Grado Department of Industrial and Systems Engineering, Virginia Tech. Email: \texttt{rhil@vt.edu}} \and
{\.I.} Esra B{\"u}y{\"u}ktahtak{\i}n\thanks{Grado Department of Industrial and Systems Engineering, Virginia Tech. Email: \texttt{esratoy@vt.edu}}
}

\date{\today}

\maketitle

\begin{abstract}
We introduce \emph{cooperative integer programming games} (CIPGs), in which
agents pool budget constraints to accomplish indivisible tasks jointly and the
characteristic function maps every coalition to the optimal value of a pooled
integer program. Our goal is to identify an optimal coalition structure (OCS)
and a stable one (OSCS). We derive a stability inequality that keeps each formed
coalition in the Core with respect to itself, and present two mixed-integer OCS
formulations, aggregated and disaggregated, proving that the disaggregated
formulation is integer-equivalent yet yields a tighter LP relaxation.
Building on the stability inequality we develop lifted stability cuts, several
separation strategies inside a cutting-plane algorithm, an SCS-feasible primal
heuristic that constructs warm starts with guaranteed stability, and a
payoff-refinement step computing the Shapley value and the nucleolus of every
formed coalition. On benchmark cooperative knapsack games, the method
certifies optimality with up to 16 players and reaches MIP gaps below 1\% at
30 players while evaluating 766 of the roughly $10^9$ coalition values.
\end{abstract}


\bigskip

\section{Introduction} \label{s:intro}

In many multi-agent systems, each player faces an individual resource-allocation
problem naturally modeled as an integer program, say a knapsack or a
capacity-constrained selection problem. Acting alone, a player may be unable to
activate certain indivisible items or projects because their budget is
insufficient; by pooling budgets, a coalition can satisfy these integer
requirements jointly and share the resulting value continuously. Municipalities
funding indivisible infrastructure projects are a canonical example, and similar
settings arise in joint procurement, collaborative R\&D, and shared logistics.
In each, what a group achieves depends on who joins it and not on what the
players outside it do. That assumption is what makes a characteristic function
the right object.

A \emph{cooperative game} in characteristic function form is a pair $(N,V)$
consisting of a finite player set $N$ and a function
$V : 2^N \to \mathbb{R}$ with $V(\emptyset) = 0$, where $V(c)$ is the value a
coalition $c$ can secure on its own \citep{peleg2007introduction}. The classical
questions are which coalitions form and how they divide what they earn, and the
central stability notion is the \emph{Core} of $c$, the set of payoff vectors
\begin{equation}
\Bigl\{\, v \in \mathbb{R}^{|c|} \;:\;
   \sum_{i \in c} v_i = V(c), \quad
   \sum_{i \in s} v_i \ge V(s) \ \ \forall\, s \subseteq c \,\Bigr\},
\label{eq:core_def}
\end{equation}
that is, payoffs that are \emph{efficient} and \emph{stable} against
subcoalition deviation.

The classical theory takes $V$ as given. We introduce a variant in which it must
be computed. A \emph{cooperative integer programming game} (CIPG) is a
cooperative game whose values are optimal values of integer programs, so each
player's feasible plans live in a mixed-integer set and evaluating even one
$V(c)$ can require solving an NP-hard problem. A second difficulty is
independent of the first: $V$ has $2^{|N|}-1$ entries, and one stability
requirement refers to many at once.

Writing down $V(c)$ is itself a modeling decision. The program a player solves
alone is given, but there is no canonical way to extend it to a coalition.
Members' decisions may interact, their constraints may combine, and
collaborating may open up plans that no member could have chosen alone. Each of
these has to be settled and written into the constraints, and every answer
produces a different $V$. We show how much rides on the choice: superadditivity
and nonemptiness of the Core hold or fail with it, and so does the question of
whether the grand coalition is simply optimal.

We propose a resource allocation setting in which those choices are made
explicit. Certain resources combine when players collaborate, certain decisions
are then taken jointly by the coalition, and the rest stay local to each player.
Computing in this regime is our focus.

A \emph{coalition structure} (CS) over $c$ is a partition of $c$. Following
\citet[Def.~3.30]{koczy2018partition} we work with payoff-configurations, pairs
$(v, CS(N))$ efficient within each block and individually rational, and add Core
stability on each formed coalition \citep{aumann1974cooperative}:
$CS(N) = \{c_1,\dots,c_k\}$ with payoff vector $v$ is an SCS if, for each
$c_\ell$, the restriction of $v$ to $c_\ell$ lies in the Core of the game
restricted to $c_\ell$. This is weaker than the coalition structure core
\citep[Def.~3.31]{koczy2018partition}, which requires
$\sum_{i \in s} v_i \ge V(s)$ for \emph{every} $s \subseteq N$;
Section~\ref{app:complexity} shows that the weaker notion is what makes our
question well posed. Because
\eqref{eq:core_def} is enforced \emph{within} each coalition, an SCS guards
against fission: no subcoalition has an incentive to break away. It does not
guard against fusion, so separate coalitions may still benefit from merging
\citep{yang2025weak}. We call any SCS
that maximizes total payoff $\sum_{i \in N} v_i$ an OSCS, and our algorithms
optimize over SCS-feasible solutions directly.

We investigate three questions: which structural properties of $V$ can be
established in CIPGs; how to model and compute an \emph{optimal coalition
structure} (OCS); and how to characterize SCSs and \emph{optimize} over them to
obtain an OSCS.

\subsection{Literature Review} \label{s:literature}

Coalition-structure generation asks for a partition of $N$ maximizing social
welfare when the values $V(c)$ are already available, listed explicitly or
produced by a simple formula
\citep{sandholm1999coalition,rahwan2009anytime,rahwan2015coalition}. A parallel
literature equips games with coalition structures with their own stability
notions
\citep{aumann1974cooperative,koczy2018partition,apt2006stable,apt2009generic};
we adopt the coalition-structure stability of that line in
Section~\ref{ss:stability_CSS}.

Once the values are not given, the question becomes where they come from. The
closest classical answer is the \emph{linear production game} of
\citet{owen1975core}, in which players pool resource endowments and each
coalition is worth the optimal value of a linear program;
\citet{granot1986generalized} extends this to a generalized linear production
model that unifies a broad class of cooperative games. In that setting the Core
of the grand coalition is never empty: LP duality reads a Core allocation
directly off the dual prices on the pooled resources, so the grand coalition is
stable and optimal by construction.

CIPGs are the integer counterpart, and that guarantee fails: no dual
construction survives integrality. Optimal dual prices certify nothing once coalition-level decisions are integer,
so the duality argument for Core nonemptiness has no analogue; evaluating a
single $V(c)$ becomes NP-hard; and our experiments contain a full-pooling
instance whose grand coalition admits no Core allocation
(Section~\ref{s:numerical}). The scope differs as well. Linear production games
allocate within the grand coalition, whereas we optimize over coalition
\emph{structures} and enforce Core stability inside each formed coalition.

A second answer defines $V(c)$ through a fixed combinatorial problem on a graph
or network, which lets many Core-related questions be settled with specialized
combinatorial structure
\citep{borm2001operations,deng1999algorithmic,chalkiadakis2011computational}.
The \emph{integer minimization games} of \citet{caprara2010cost} subsume
these, with $c(S)$ the value of an integer program parameterized by the
coalition's incidence vector; subadditivity is automatic there, while a CIPG's
coalition policy can break its mirror, making the partition itself the
decision (Remark~\ref{rem:im_games} in the e-companion).

A third answer, and the one closest to practice, reads $V(c)$ off an operational
optimization model. \citet{gothe1996nucleolus} define $V(c)$ as the optimal
value of a vehicle routing problem and compute Core and nucleolus quantities by
constraint generation, evaluating $V(s)$ only when a violated inequality demands
it; this is the separation-as-optimization idea we develop in
Section~\ref{sec:separation}. Such routing games need not be balanced, so their
Core can be empty \citep{chardaire2001core}, which is itself an argument for
optimizing over partitions instead of assuming the grand coalition.
\citet{ozener2013allocating} allocate inventory-routing cost by Core-based
dual methods with IP-defined values.

Closest to our OCS problem, \citet{guajardo2015coalition} propose OR models for
choosing the coalition structure itself in collaborative logistics, assessing
stability separately for each candidate structure. We instead embed the
coalition-level IPs and the Core constraints in a single mixed-integer model
with exact single-level MIP separation, which makes the partition a decision
variable and enforces stability inside every formed coalition instead of
verifying it afterwards
\citep[surveys:][]{guajardo2016review,gansterer2018collaborative}.

Optimization-based characteristic functions arise wherever indivisibilities
and player-specific integer decisions are essential, from water resources to
transmission cost allocation, waste management, and shared energy storage
\citep{wang2003water,stamtsis2004use,eryganov2020application,wang2024cooperative}.

For the knapsack setting in particular, the axiomatic literature studies a
\emph{single} knapsack held by the grand coalition
\citep{arribillaga2022cooperative,Dror1990,DarmannKlamler2014}, whereas the
Cooperative Knapsack Game of Section~\ref{s:comp_results} gives every player
their own knapsack and every coalition a pooled one.

On the non-cooperative side, integer programming games model strategic
interaction in which each player solves a mixed-integer program and payoffs
couple through bilinear or shared-constraint structures
\citep{carvalho2023integer,dragotto2023zero,lee2024algorithms}. These models
form no coalitions.

Across these literatures the coalition values are given, LP-generated,
graph-generated, read off an operational model, or priced for a fixed grand
coalition. In a CIPG they are integer-programming values generated inside the
search over coalition structures, which fuses three problems usually treated separately: computing $V(c)$
is NP-hard, choosing the partition is combinatorial, and certifying stability
needs values for exponentially many subcoalitions.
Sections~\ref{s:cipg} and~\ref{sec:computation} address the three together.

\subsection{Contributions and Outline}

We introduce \emph{cooperative integer programming games} (CIPGs), the integer
counterpart of the linear production games of \citet{owen1975core}. Integrality
removes the dual-price construction that guarantees a nonempty Core in the
linear case, and evaluating a single coalition value is already NP-hard.

\begin{itemize}

\item \textbf{Structure of the characteristic function.} We show that the
coalition policy, not the CIPG framework, decides whether the grand coalition is
optimal. Full pooling makes $V$ superadditive
(Proposition~\ref{prop:superadd_natural}); restricted bounds can break it and
make the partition a decision worth making.

\item \textbf{Two OCS formulations.} We give a mixed-integer formulation of the
optimal coalition structure problem with an aggregated and a disaggregated
linearization, and compare them theoretically and computationally. The
disaggregated model leaves the integer problem unchanged but yields a
relaxation that is never weaker and is strictly tighter on every instance we
solve (Proposition~\ref{prop:disagg_LP}).

\item \textbf{Stability.} Gated stability inequalities enforce Core constraints
inside each formed coalition. Theorem~\ref{thm:CSS_projection} shows that the
resulting mixed-integer set projects exactly onto the stable payoffs, so
maximizing over it returns an OSCS. A lifted variant dominates it
pointwise (Proposition~\ref{prop:lifted_dominance}).

\item \textbf{Computation and experiments.} We separate the inequalities without
dualization, through a single-level MIP, and warm-start the solver with a primal
heuristic that returns provably stable partitions. The heuristic is the single
most effective computational component on the reported benchmark grid: it
accelerates the solves and makes the disaggregated formulation usable at scale.
Section~\ref{s:comp_results} instantiates the framework as a Cooperative
Knapsack Game and reports the study.

\end{itemize}

The remainder of the paper is organized as follows.
Section~\ref{s:cipg} formalizes CIPGs and asks when the coalition structure is a
decision worth making. Section~\ref{s:formulation} builds the mixed-integer model:
the OCS formulation with its two linearizations, the stability inequalities and
their lifted strengthening, and the exactness result that ties them together.
Section~\ref{sec:computation} solves that model, covering separation oracles, the
iterative and lazy-constraint architectures, the SCS-feasible warm start, and the
per-coalition payoff refinement, summarized in
Algorithm~\ref{alg:ocss_lazy}.
Section~\ref{s:comp_results} reports the computational study on cooperative
knapsack games, including the comparison with an exhaustive baseline, and
Section~\ref{s:conclusion} concludes. Proofs omitted from the main text are in
the e-companion (Section~\ref{app:proofs}).

\section{Cooperative Integer Programming Games}\label{s:cipg}

Two objects specify a CIPG: the problem each player solves alone, and the policy
that turns a set of players into a single problem they solve together.

\begin{definition}[Cooperative integer programming game]
\label{def:cipg}
A \emph{cooperative integer programming game} consists of a player set $N$, a
vector of decision variables $x^i$ for each player $i \in N$, and, for every
coalition $c \subseteq N$, an objective $f^c$ and a mixed-integer set
$\mathcal{X}^c$ in the members' variables $x^c := (x^i)_{i \in c}$ together with
auxiliary variables $y^c$ that the coalition controls jointly, giving
\begin{equation}
  V(c) \;=\; \max\ \bigl\{\, f^c(x^c, y^c) \;:\; (x^c, y^c) \in \mathcal{X}^c
  \,\bigr\}, \qquad V(\emptyset) = 0.
\label{eq:player_general}
\end{equation}
For a singleton this is player $i$'s own problem. We assume every program is
feasible and bounded, so that $V$ is finite and $(N,V)$ is a cooperative game.
\end{definition}

A coalition structure is itself a decision, so the model must switch a
player's decisions on and off with membership, and keeping each member's
$x^i$ identifiable is what makes that possible. The auxiliary
variables $y^c$ belong to the coalition and are not inherited when membership
changes.

We focus on linear objectives $f^c$. More general objectives whose
hypograph admits an exact mixed-integer representation can be incorporated
through auxiliary variables, which $\mathcal{X}^c$ is free to carry.

The definition becomes computationally interesting when the family
$\{\mathcal{X}^c\}$ is described by a \emph{coalition policy} rather than
listed. Any one number is the optimum of some trivial integer program, so a
listed family would admit every cooperative game with rational values, and an
algorithm could do nothing but fill in $2^{|N|}$ entries one at a time. A policy
instead builds $\mathcal{X}^c$ from the players in $c$, which is what lets
$V(s)$ be reached for a subset never examined before. It also carries the
modeling decision described in Section~\ref{s:intro}: which constraints combine,
which decisions the coalition takes jointly, what stays local to a player, and
which plans only collaboration makes available. Different policies give
different games on the same players.

One commitment is built into the definition. Since $\mathcal{X}^c$ depends on
$c$ alone, a coalition acts in isolation: what it achieves is unaffected by how
the remaining players arrange themselves, and no resource is contested between
two coalitions of the same partition. That places the game in characteristic
function form rather than partition function form, and it is why a partition
$CS(N)$ is worth exactly $\sum_{c \in CS(N)} V(c)$, the quantity
Section~\ref{ss:OCS_formulation} maximizes.

\subsection{A resource allocation game}
\label{ss:player_coalition_IP}

For the rest of the paper we fix one coalition policy. Each player has a single
budget constraint, which combines across a coalition, together with private
constraints that never combine. Section~\ref{ss:general_milp} of the
e-companion returns to the general case.

For each player $i \in N$ and item $j \in R$, let
$x^i_j \in \mathbb{Z}_+$ be the amount of item $j$ chosen by player $i$, with
profit coefficient $p^i_j \ge 0$ and common weight $a_j > 0$. Player $i$ has budget
$b^i$, a capacity $u^i_j$ limiting how much of item $j$ it may take, and
possibly additional private constraints $B^i x^i \le d^i$. When a coalition
forms its members pool their budgets $\{b^i\}_{i \in c}$ and retain their
private constraints, and a \emph{pooling parameter} $\alpha \in (0,1]$ caps the
coalition's total use of each common item at $\alpha$ times the sum of its
members' own capacities. Since $\alpha$ fixes $V$, it is part of the instance,
not of any formulation of it; Section~\ref{ss:V_structural} examines what
different values do.

Not every player has access to every item. For each item $j \in R$ let
$I_j \subseteq N$ be the players for whom $j$ is available. Given a coalition
$c \subseteq N$,
\begin{equation}
\begin{aligned}
  I^c_j &:= I_j \cap c,
    &\qquad R^c &:= \{\, j \in R : I^c_j \neq \emptyset \,\}, \\[2pt]
  R^c_{\mathrm{ind}} &:= \{\, j \in R^c : |I^c_j| = 1 \,\},
    &\qquad R^c_{\mathrm{com}} &:= \{\, j \in R^c : |I^c_j| \ge 2 \,\},
\end{aligned}
\label{eq:item_sets}
\end{equation}
so $R^c_{\mathrm{ind}}$ collects the items available to exactly one member of
$c$, which we call \emph{individual-specific}, and $R^c_{\mathrm{com}}$ those
available to two or more, which we call \emph{common}.

This designation is coalition-dependent, as Figure~\ref{fig:availability}
illustrates: an item available only to players $1$ and $3$ is
individual-specific in coalition $\{1,2\}$ and common in $\{1,3\}$. The sets
\eqref{eq:item_sets} are therefore recomputed for every coalition under
consideration, in the OCS formulations, in the separation oracles, and in the
experiments alike.

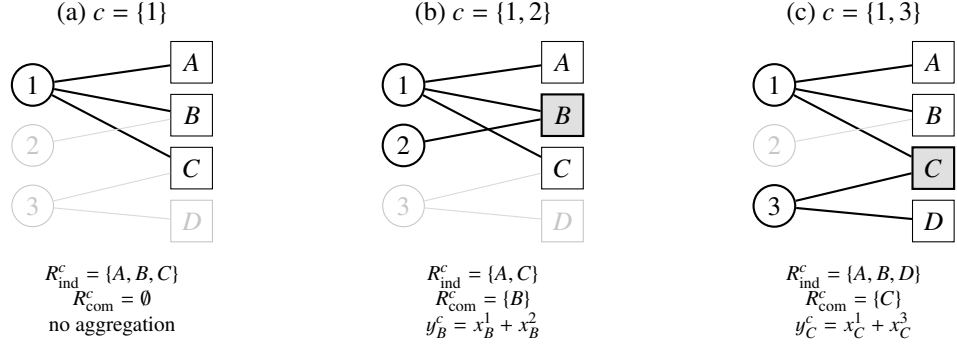
\begin{figure}[t]
\centering
\begin{tikzpicture}[
  ply/.style  ={circle,draw,thick,minimum size=5.5mm,inner sep=0pt,font=\small},
  plyo/.style ={circle,draw,gray!45,text=gray!45,minimum size=5.5mm,inner sep=0pt,font=\small},
  itm/.style  ={rectangle,draw,minimum size=5.5mm,inner sep=1pt,font=\small},
  com/.style  ={rectangle,draw,thick,fill=black!12,minimum size=5.5mm,inner sep=1pt,font=\small},
  itmo/.style ={rectangle,draw,gray!45,text=gray!45,minimum size=5.5mm,inner sep=1pt,font=\small},
  ee/.style   ={thick},
  eo/.style   ={gray!30},
  ttl/.style  ={font=\small},
  ann/.style  ={font=\scriptsize,align=center,anchor=north}]

\begin{scope}[xshift=0cm]
  \node[ttl] at (1.05,2.75) {(a) $c=\{1\}$};
  \node[ply] (a1) at (0,1.8) {1};  \node[plyo](a2) at (0,1.0) {2};  \node[plyo](a3) at (0,0.2) {3};
  \node[itm] (aA) at (2.1,2.1) {$A$}; \node[itm] (aB) at (2.1,1.4) {$B$};
  \node[itm] (aC) at (2.1,0.7) {$C$}; \node[itmo](aD) at (2.1,0.0) {$D$};
  \draw[ee](a1)--(aA); \draw[ee](a1)--(aB); \draw[ee](a1)--(aC);
  \draw[eo](a2)--(aB); \draw[eo](a3)--(aC); \draw[eo](a3)--(aD);
  \node[ann] at (1.05,-0.45) {$R^c_{\mathrm{ind}}=\{A,B,C\}$\\
                              $R^c_{\mathrm{com}}=\emptyset$\\
                              no aggregation};
\end{scope}

\begin{scope}[xshift=4.9cm]
  \node[ttl] at (1.05,2.75) {(b) $c=\{1,2\}$};
  \node[ply] (b1) at (0,1.8) {1};  \node[ply] (b2) at (0,1.0) {2};  \node[plyo](b3) at (0,0.2) {3};
  \node[itm] (bA) at (2.1,2.1) {$A$}; \node[com] (bB) at (2.1,1.4) {$B$};
  \node[itm] (bC) at (2.1,0.7) {$C$}; \node[itmo](bD) at (2.1,0.0) {$D$};
  \draw[ee](b1)--(bA); \draw[ee](b1)--(bB); \draw[ee](b1)--(bC);
  \draw[ee](b2)--(bB); \draw[eo](b3)--(bC); \draw[eo](b3)--(bD);
  \node[ann] at (1.05,-0.45) {$R^c_{\mathrm{ind}}=\{A,C\}$\\
                              $R^c_{\mathrm{com}}=\{B\}$\\
                              $y^c_B=x^1_B+x^2_B$};
\end{scope}

\begin{scope}[xshift=9.8cm]
  \node[ttl] at (1.05,2.75) {(c) $c=\{1,3\}$};
  \node[ply] (c1) at (0,1.8) {1};  \node[plyo](c2) at (0,1.0) {2};  \node[ply] (c3) at (0,0.2) {3};
  \node[itm] (cA) at (2.1,2.1) {$A$}; \node[itm] (cB) at (2.1,1.4) {$B$};
  \node[com] (cC) at (2.1,0.7) {$C$}; \node[itm] (cD) at (2.1,0.0) {$D$};
  \draw[ee](c1)--(cA); \draw[ee](c1)--(cB); \draw[ee](c1)--(cC);
  \draw[eo](c2)--(cB); \draw[ee](c3)--(cC); \draw[ee](c3)--(cD);
  \node[ann] at (1.05,-0.45) {$R^c_{\mathrm{ind}}=\{A,B,D\}$\\
                              $R^c_{\mathrm{com}}=\{C\}$\\
                              $y^c_C=x^1_C+x^3_C$};
\end{scope}
\end{tikzpicture}
\caption{Availability and the coalition-dependent split of the item set.
Circles are players, squares are items, and an edge means availability. Members
of $c$ are drawn in black, and shaded squares are the common items, reachable by
two or more members. Item $B$ is common in (b) but individual-specific in (c),
and item $C$ the other way round: the split depends on the coalition, not on the
item. Note also that $D$ is out of reach until player $3$ joins.}
\label{fig:availability}
\end{figure}

For common items we introduce coalition-level auxiliary variables
\begin{equation}
  y^c_j \;=\; \sum_{i \in I^c_j} x^i_j, \qquad j \in R^c_{\mathrm{com}},
\label{eq:link_y}
\end{equation}
which aggregate total usage of item $j$ in coalition $c$. Integrality is placed
on $y^c_j$ and dropped from the player-level $x^i_j$ for
$j \in R^c_{\mathrm{com}}$, which reflects the two-stage nature of a coalition
decision: the coalition first commits to how many units of each common item to
activate, an integer choice recorded by $y^c_j$, and the resulting total is then
divided among the contributing members, a continuous one recorded by
$x^i_j \ge 0$ with $\sum_i x^i_j = y^c_j$. Items in $R^c_{\mathrm{ind}}$ keep
their integrality, since one player decides them.

Dropping integrality on those $x^i_j$ does not change $V(c)$. A common item's
$x^i_j$ appears only in the objective, in \eqref{eq:coal_IP_general_link} and in
its own bound, so for fixed integer $y^c_j$ the division among members is an
interval problem whose optimum is attained at an integer point. We assume
throughout that \eqref{eq:coal_IP_general_other} does not tie a member's
common items to its other decisions, and that the zero plan is feasible for
every program, so $V(c) \ge 0$ for every coalition.

The coalition value $V(c)$ is the optimal value of
\begin{subequations}\label{eq:coal_IP_general}
\begin{align}
V(c) \;=\;
&\max \sum_{j \in R^c_{\mathrm{com}}} \sum_{i \in I^c_j} p^i_j x^i_j
      + \sum_{j \in R^c_{\mathrm{ind}}} \sum_{i \in I^c_j} p^i_j x^i_j
    \label{eq:coal_IP_general_obj} \\[0.2em]
\text{s.t.}\quad
&\sum_{j \in R^c_{\mathrm{com}}} a_j y^c_j
 + \sum_{j \in R^c_{\mathrm{ind}}} \sum_{i \in I^c_j} a_j x^i_j
 \;\le\; \sum_{i \in c} b^i,
    \label{eq:coal_IP_general_budget} \\
&y^c_j = \sum_{i \in I^c_j} x^i_j,
    \quad \forall j \in R^c_{\mathrm{com}},
    \label{eq:coal_IP_general_link} \\
&B^i x^i \le d^i,
    \quad \forall i \in c,
    \label{eq:coal_IP_general_other} \\
&0 \le x^i_j \le u^i_j,
    \quad \forall i \in c,\ j \in R^c,
    \label{eq:coal_IP_general_x_ub} \\
&0 \le y^c_j \le \alpha \sum_{i \in I^c_j} u^i_j,
    \quad \forall j \in R^c_{\mathrm{com}},
    \label{eq:coal_IP_general_y_ub} \\
&y^c_j \in \mathbb{Z}_+, \quad \forall j \in R^c_{\mathrm{com}}, \qquad
 x^i_j \in \mathbb{Z}_+, \quad \forall i \in c,\ j \in R^c_{\mathrm{ind}}.
    \label{eq:coal_IP_general_domains}
\end{align}
\end{subequations}

Taking $c=\{i\}$ recovers the individual problem, since
$|I_j \cap \{i\}| \le 1$ for every $j$ forces
$R^{\{i\}}_{\mathrm{com}} = \emptyset$. We write its value $V(\{i\})$ and use it
throughout as the singleton value.

\subsection[Structural properties of the characteristic function V]{Structural properties of the characteristic function $V$}
\label{ss:V_structural}

The coalition structure is a decision only if some partition can beat the grand
coalition, and whether one ever does is settled here rather than by the CIPG
framework. Under full pooling it never does, so the OCS problem is trivial
(the OSCS problem is not); under restricted pooling it can, and that is the
regime the rest of the paper addresses. Two properties of $V$ carry the argument, a third bears on
how payoffs are divided.

Call $V$ \emph{superadditive} if $V(c \cup d) \ge V(c) + V(d)$ for all disjoint
coalitions $c,d$, \emph{subadditive} if the reverse inequality holds, and
\emph{convex} if $V(c) + V(d) \le V(c \cup d) + V(c \cap d)$ for all $c,d$. Repeated merging
shows that the grand coalition is an OCS in a superadditive game, and
symmetrically that the singleton partition is an OCS in a subadditive one.
Convexity is strictly stronger than superadditivity and matters later rather
than here: on a convex game the Shapley value lies in the Core
(Section~\ref{app:refinement_details} of the e-companion), so convexity would make
Section~\ref{ss:coalition_refinement} free.

\medskip
\noindent\textbf{(i) Full pooling, $\alpha = 1$.}
A coalition may then use the entire sum of its members' capacities on each common
item, and \eqref{eq:coal_IP_general_y_ub} is implied by
\eqref{eq:coal_IP_general_link} and \eqref{eq:coal_IP_general_x_ub}, so it never
binds on its own.

\begin{proposition}[Superadditivity]
\label{prop:superadd_natural}
Assume full pooling, $\alpha = 1$. Then $V$ is superadditive.
\end{proposition}

\begin{proof}{Proof}
Let $c,d \subseteq N$ be disjoint and take optimal solutions $(x^{c},y^{c})$ for
$V(c)$ and $(x^{d},y^{d})$ for $V(d)$. Define a candidate for $c \cup d$ by
$\hat x^i_j = x^{c,i}_j$ if $i \in c$, $\hat x^i_j = x^{d,i}_j$ if $i \in d$,
$\hat x^i_j = 0$ otherwise, and
$\hat y^{\,c \cup d}_j := \sum_{i \in I^{c \cup d}_j} \hat x^i_j$ for
$j \in R^{c \cup d}_{\mathrm{com}}$. Since $I^{c \cup d}_j$ is the disjoint
union of $I^c_j$ and $I^d_j$, this total splits into the two coalitions'
totals, each integral ($y^c_j$ if $j$ is common in $c$, a single integer
$x^{c,i}_j$ if individual-specific, $0$ if $j \notin R^c$; likewise for $d$),
so $\hat y^{\,c \cup d}_j$ is integral, and $0 \le \hat x^i_j \le u^i_j$
gives $0 \le \hat y^{\,c \cup d}_j \le \sum_{i \in I^{c \cup d}_j} u^i_j$.
The budget constraint for $c \cup d$
holds by summing the budgets of $c$ and $d$, and every private constraint
$B^i x^i \le d^i$ is unaffected since each $x^i$ is unchanged. The objective
value of $(\hat x,\hat y)$ is $V(c) + V(d)$, so $V(c \cup d) \ge V(c) + V(d)$.
\Halmos
\end{proof}

Under full pooling Proposition~\ref{prop:superadd_natural} identifies the OCS
\emph{in advance}, so the
full-pooling runs of Section~\ref{s:comp_results} measure how quickly each
formulation \emph{certifies} an optimum rather than how quickly it finds one.
It also explains why the lifted stability cut never falls back to the basic cut
in that regime: superadditivity forces the synergy condition
$V(s) \ge \sum_{i \in s} V(\{i\})$ for every subset
(Section~\ref{ss:lifted_cuts}). Finally, it is the reason the restricted-bound
setting of part (ii) exists at all: without it no partition could strictly
outperform the grand coalition.

Superadditivity does not bring convexity with it. Capacity saturation makes
marginal contributions shrink as coalitions grow: in a three-player instance
with a single item too expensive for any one player, the second member of a
coalition is worth the full item and the third is worth nothing, which breaks
convexity while leaving superadditivity intact
(Example~\ref{ex:nonconvex} in the e-companion).

CIPGs therefore need not be convex, and the lack of convexity removes the
Shapley--Core guarantee, consistent with the observed violations on the larger
pooled coalitions of Section~\ref{s:numerical}.

\medskip
\noindent\textbf{(ii) Restricted pooling, $\alpha < 1$.}
A coalition then cannot fully exploit that sum. No later result depends on the
linear form: the stability inequalities, the separation and the heuristic need
only that the cap depends on which members are present, so a different policy,
in the extreme a binary activation $y^c_j \in \{0,1\}$ for a shared resource,
would change $V$ but not the machinery. Such a bound arises whenever the
limit attaches to the applicant rather than to the activity: public funding
programs cap awards per recipient, and spectrum aggregation caps or airport
slot rules bind a merged licensee once rather than twice, so cooperation is
worth having for other reasons but costs the coalition part of what its
members were separately entitled to.

Superadditivity can now fail. With a binary common item and two players, each alone
sets $x^i_1 = 1$ and obtains value $1$, while the coalition $\{1,2\}$ faces a
single shared unit $y^{\{1,2\}}_1 \in \{0,1\}$ and is worth only $1 < 1+1$. In
this regime the OCS is no longer guaranteed to be the grand coalition; it
becomes instance-dependent and nontrivial partitions can beat $\{N\}$,
as assumed from Section~\ref{ss:OCS_formulation} onward.

\section{A Mixed-Integer Formulation}\label{s:formulation}

We now build a model whose solutions are exactly the stable coalition
structures. It comes in two layers. Section~\ref{ss:OCS_formulation} selects a
partition and solves the coalition programs of every formed group, which gives
an optimal coalition structure but enforces no stability;
Sections~\ref{ss:stability_CSS} and~\ref{ss:lifted_cuts} add a family of
inequalities that cut off the partitions in which some subcoalition would
rather deviate. Together the two layers describe the SCS payoffs exactly
(Theorem~\ref{thm:CSS_projection}), so an OSCS is obtained by maximizing over
them.

\subsection{OCS formulations: aggregated and disaggregated}
\label{ss:OCS_formulation}

We now formulate a mixed-integer program that selects a coalition structure and
solves the coalition-level integer programs for each formed group. We reuse the
item set $R$ and the availability sets $I_j$ from
Section~\ref{ss:player_coalition_IP}. Because the split of \eqref{eq:item_sets}
into common and individual items depends on which coalition forms, we do not fix
global sets $R_{\mathrm{com}}$ and $R_{\mathrm{ind}}$; the classification is
determined endogenously by indicator variables.

Coalitions are represented by \emph{group indices}. Since at most $|N|$ groups
can form, we take $G := N$ as the index set and say that the players with
$z_{ig} = 1$ form coalition $c(g)$. Table~\ref{tab:OCS_notation} in the
e-companion collects the models' notation; notation local to separation and
the heuristic is introduced where stated
(Sections~\ref{sec:separation}, \ref{sec:heuristics}).

The variables are group-indexed copies of those in \eqref{eq:coal_IP_general}:
$x^{i,g}_j \ge 0$ and $y^g_j \in \mathbb{Z}_+$ replace $x^i_j$ and $y^c_j$ for the
coalition realized in group $g$, with $z_{ig} \in \{0,1\}$ recording the
assignment, $w^g_j \in \{0,1\}$ marking item $j$ as common in group $g$,
$\pi_g \ge 0$ the profit group $g$ realizes, and $v_i \ge 0$ the payoff to player
$i$. We write
$\mathrm{AGG}_\alpha$ for the \emph{aggregated} formulation of the
$\mathrm{OCS}_\alpha$ problem: the subscript matters, since changing $\alpha$
changes $V$ and so changes the problem, not just its formulation. Its constraints
follow the order of \eqref{eq:coal_IP_general}, so the two can be read side by
side.
\begingroup
\allowdisplaybreaks
\begin{subequations}\label{eq:OCS}
\begin{align}
\blockhead{objective, the partition, and profit accounting:}
    \max \quad
    & \sum_{i \in N} v_i
      \label{eq:OCS_obj} \\
    \text{s.t.} \quad
    & \sum_{g \in G} z_{ig} = 1,
        && \forall i \in N,
        \label{eq:OCS_partition} \\
    & \pi_g = \sum_{j \in R} \sum_{i \in I_j} p^i_j x^{i,g}_j
             = \sum_{i \in N} v_i z_{ig},
        && \forall g \in G,
        \label{eq:OCS_efficiency} \\
\blockhead{coalition program for each group, exactly
  \eqref{eq:coal_IP_general_budget}--\eqref{eq:coal_IP_general_y_ub}:}
    & \sum_{j \in R} a_j y^g_j
      \;\le\; \sum_{i \in N} b^i z_{ig},
        && \forall g \in G,
        \label{eq:OCS_cap} \\
    & y^g_j = \sum_{i \in I_j} x^{i,g}_j,
        && \forall g \in G,\ \forall j \in R,
        \label{eq:OCS_link} \\
    & B^i x^{i,g} \le d^i z_{ig},
        && \forall g \in G,\ \forall i \in N,
        \label{eq:OCS_other} \\
    & 0 \le x^{i,g}_j \le u^i_j z_{ig},
        && \forall g \in G,\ \forall j \in R,\ \forall i \in I_j,
        \label{eq:OCS_x_ub} \\
    & y^g_j \;\le\; \sum_{i \in I_j} u^i_j z_{ig},
        && \forall g \in G,\ \forall j \in R,
        \label{eq:OCS_y_ub} \\
\blockhead{constraints that endogenous partitioning requires, with
  domains as in \eqref{eq:coal_IP_general_domains}:}
    & z_{ig} = 0,
        && \forall g \in G,\ \forall i \in N \text{ with } i < g,
        \label{eq:OCS_sym_leader_range} \\
    & \sum_{i \in N} z_{ig} \le |N|\, z_{gg},
        && \forall g \in G,
        \label{eq:OCS_sym_leader_nonempty} \\
    & v_i \ge 0,\quad
      y^g_j \in \mathbb{Z}_+,\quad
      z_{ig} \in \{0,1\},
        && \forall i \in N,\ \forall g \in G,\ \forall j \in R,
        \label{eq:OCS_domains} \\
\blockhead{needed only when $\alpha < 1$; omitted under full pooling:}
    & 2\,w^g_j \;\le\; \sum_{i \in I_j} z_{ig} \;\le\; (|I_j|-1)\,w^g_j + 1,
        && \forall g \in G,\ \forall j \in R,
        \label{eq:OCS_w} \\
    & y^g_j \;\le\; \alpha \sum_{i \in I_j} u^i_j z_{ig} + M_j(1 - w^g_j),
        && \forall g \in G,\ \forall j \in R,
        \label{eq:alpha_restriction} \\
    & w^g_j \in \{0,1\},
        && \forall g \in G,\ \forall j \in R.
        \label{eq:OCS_w_dom}
\end{align}
\end{subequations}
\endgroup

The objective \eqref{eq:OCS_obj} maximizes total payoff. Constraint
\eqref{eq:OCS_partition} puts each player in exactly one group, so the nonempty
groups induce a partition of $N$. The chain \eqref{eq:OCS_efficiency} says that
what a group produces is what it pays out: summing it over $g$ and using
\eqref{eq:OCS_partition} gives $\sum_{g} \pi_g = \sum_{i} v_i$, so one could
equivalently maximize total profit $\sum_{g \in G} \pi_g$. Since $\pi_g$ is the
objective value of a feasible, not necessarily optimal, production plan, it
satisfies $\pi_g \le V(c(g))$ for the realized coalition
$c(g) = \{i \in N : z_{ig} = 1\}$, with equality at any optimal solution.
Constraints \eqref{eq:OCS_cap} to \eqref{eq:alpha_restriction}, with the domains
in \eqref{eq:OCS_domains}, embed for each group a coalition-level integer program
consistent with \eqref{eq:coal_IP_general}, and in the same order: budgets of
assigned members are pooled in \eqref{eq:OCS_cap}, items aggregate through
\eqref{eq:OCS_link}, private rows are preserved by \eqref{eq:OCS_other}, usage is
bounded by \eqref{eq:OCS_x_ub}, and \eqref{eq:alpha_restriction} carries the
coalition policy, matching \eqref{eq:coal_IP_general_y_ub} one for one. The
indicators $w^g_j$ it refers to are fixed by \eqref{eq:OCS_w} below.

The $\alpha$-restriction \eqref{eq:alpha_restriction} carries the coalition
policy into the formulation. It is released by a big-$M$ term keyed to $w^g_j$
because it should bind only on common items: an item reaching a single member is
already capped by \eqref{eq:OCS_x_ub}. At $\alpha = 1$ it is redundant, since
summing \eqref{eq:OCS_x_ub} over $i \in I_j$ with \eqref{eq:OCS_link} gives
$y^g_j \le \sum_{i \in I_j} u^i_j z_{ig}$ already, so one formulation covers both
regimes. The final block of \eqref{eq:OCS} collects everything this needs, so under full
pooling it is simply omitted: our implementation does so, removing
$|G| \cdot |R|$ binaries, which is $1500$ at $n = 30$. What remains is then
exactly \eqref{eq:coal_IP_general} lifted to group indices, with
\eqref{eq:OCS_y_ub} in the role of \eqref{eq:coal_IP_general_y_ub}. That row is
implied by \eqref{eq:OCS_link} and \eqref{eq:OCS_x_ub} and is stated only to make
the correspondence visible.
The constant $M_j$ has to be large enough to make
\eqref{eq:alpha_restriction} redundant when $w^g_j = 0$. In that case
the upper bound in \eqref{eq:OCS_w} leaves at most one member of $I_j$ in the
group, so $y^g_j$
cannot exceed that member's own bound $u^i_j$ by \eqref{eq:OCS_link} and
\eqref{eq:OCS_x_ub}, and $M_j = (1-\alpha)\max_{i \in I_j} u^i_j$ suffices. Any
larger value is also valid but only weakens the relaxation.

The two-sided bound \eqref{eq:OCS_w} makes $w^g_j = 1$ exactly when two or more players
of $I_j$ are assigned to group $g$: the upper bound forces $w^g_j = 1$ once a
second player joins, and the lower bound forbids it before that, which is
the endogenous version of the classification \eqref{eq:item_sets}. Their purpose
is to let the restricted bounds of Section~\ref{ss:V_structural} apply to common
items only, which is where they mean anything: an item reaching a single member
is already capped at $u^i_j$ by \eqref{eq:OCS_x_ub}, and there is no sharing to
restrict.

The parameter interpolates between the two regimes of
Section~\ref{ss:V_structural}. At $\alpha = 1$ the bound coincides with the
full-pooling bound, so Proposition~\ref{prop:superadd_natural}
applies, $V$ is superadditive, and the grand coalition is an OCS. Smaller
$\alpha$ is the restricted regime in which superadditivity fails and the
partition becomes a decision worth making. Section~\ref{s:comp_results} reports
both, at $\alpha = 1$ and $\alpha = 0.5$.

Constraints \eqref{eq:OCS_sym_leader_range} and
\eqref{eq:OCS_sym_leader_nonempty} implement a group-leader symmetry-breaking
scheme: group $g$ contains only players with indices $i \ge g$, and player $g$
belongs to group $g$ whenever the group is nonempty, so each nonempty group is
led by its smallest-indexed member and every coalition structure has a canonical
representation.

Any optimal solution to \eqref{eq:OCS} therefore yields an OCS.

Constraint \eqref{eq:OCS_efficiency} involves the product $v_i z_{ig}$ and is
bilinear. We linearize it with standard big-$M$ techniques: for each pair
$(i,g)$ introduce $\phi_{ig} = v_i z_{ig}$ with
$\phi_{ig} \le v_i$, $\phi_{ig} \le M z_{ig}$,
$\phi_{ig} \ge v_i - M(1-z_{ig})$, and $\phi_{ig} \ge 0$, where $M$ bounds any
payoff. Since a player's payoff cannot exceed the value of a coalition it
belongs to, a computation-free bound is
$M = \sum_{i \in N} \sum_{j \in R:\, i \in I_j} p^i_j u^i_j$. The right-hand
side of \eqref{eq:OCS_efficiency} becomes $\sum_{i \in N} \phi_{ig}$. We call
this the \emph{aggregated formulation}, or AGG, with LP relaxation
$F_{\mathrm{agg}}$.

\subsubsection{Disaggregated formulation.}
\label{ss:disagg}

The aggregated formulation introduces $3|N| \cdot |G|$ McCormick constraints and
relies on a big-$M$ subtraction in $\phi_{ig} \ge v_i - M(1-z_{ig})$, which
weakens the LP relaxation. A \emph{disaggregated} alternative, or DIS, removes
every McCormick constraint. Disaggregating a variable controlled by a
multiple-choice constraint and linking the copies is a standard reformulation
technique \citep{frangioni2006perspective,gunluk2010perspective}; what is
specific here is that it renders the McCormick system redundant outright, and
the size of the effect measured in Section~\ref{s:numerical}. Replace the single payoff variable $v_i$ by
\emph{per-group} variables $v_{ig} \ge 0$, the payoff player $i$ receives from
group $g$. These are tied to the total, efficiency becomes purely linear, and
the on/off structure is carried by a single variable upper bound:
\begin{equation}\label{eq:disagg}
  \underbrace{v_i = \sum\nolimits_{g \in G} v_{ig}}_{\text{linking}},\ \forall i \in N;
  \quad
  \underbrace{\pi_g = \sum\nolimits_{i \in N} v_{ig}}_{\text{efficiency}},\ \forall g \in G;
  \quad
  \underbrace{v_{ig} \le M z_{ig}}_{\text{upper bound}},\ \forall i \in N,\, g \in G.
\end{equation}
with the same computation-free bound $M$, which now appears only in
the variable upper bound and never in a subtraction. The per-player quantity
$U_i = \sum_{j:\, i \in I_j} p^i_j u^i_j$ would not do: a Core allocation may
reward a player beyond its own profit potential, for instance one whose
contribution is purely budget (Example~\ref{ex:budget_only} in the
e-companion). The disaggregated model replaces the second equality of
\eqref{eq:OCS_efficiency}, and the McCormick linearization, by \eqref{eq:disagg}, keeping every other constraint of
\eqref{eq:OCS}, including the $\alpha$-restriction. We call it
$\mathrm{DIS}_\alpha$, write $F_{\mathrm{disagg}}$ for its LP relaxation and, setting
$\phi_{ig} := v_{ig}$, compare the two in the coordinates $(v,\phi,z,x,y,\pi)$.

\begin{proposition}[Same integer problem, tighter relaxation]
\label{prop:disagg_LP}
The aggregated and disaggregated formulations have the same mixed-integer
feasible points. Their relaxations satisfy
$F_{\mathrm{disagg}} \subseteq F_{\mathrm{agg}}$, and the inclusion can be
strict. Moreover, adding the linking constraint $v_i = \sum_g \phi_{ig}$ to
$F_{\mathrm{agg}}$ recovers $F_{\mathrm{disagg}}$ exactly.
\end{proposition}

A byproduct of the proof (Section~\ref{app:proofs} of the e-companion) is that
under disaggregation every McCormick constraint is redundant, a net reduction of
$2|N| \cdot |G| - |N|$ constraints. Strictness rests on a single fractional
point, but the practical effect is larger: on all 30 instances of
Section~\ref{s:numerical} the disaggregated root bound is strictly tighter, by
margins between $45$ and $98\%$.

\subsection{Stability inequality for SCS}
\label{ss:stability_CSS}

The OCS formulation \eqref{eq:OCS} maximizes total coalition value over all
partitions of $N$ but does not enforce stability: a formed coalition may admit a
blocking subcoalition. We now introduce a family of \emph{stability
inequalities} restricting \eqref{eq:OCS} to \emph{stable coalition structures}
(SCS), in which every formed coalition is in the Core with respect to itself.

Each group index $g \in G$ corresponds to a potential coalition
$c(g) := \{ i \in N : z_{ig} = 1 \}$. For each formed coalition, Core stability
requires $\sum_{i\in s} v_i \ge V(s)$ for all $s \subseteq c(g)$. We encode
these inequalities only when the players of $s$ are together in some coalition,
and we do so without coalition-indicator variables by gating on the assignment
variables. For each subset $s \subseteq N$ and each group $g \in G$ we impose
the \emph{stability inequality}
\begin{equation}
\sum_{i\in s} v_i
\;\ge\;
V(s) - M_s \Bigl( |s| - \sum_{i\in s} z_{ig} \Bigr),
\label{eq:stability_bigM}
\end{equation}
where $M_s \ge \max\{V(s),\, 0\}$. The gate term $|s| - \sum_{i\in s} z_{ig}$
vanishes exactly when all of $s$ is assigned to group $g$, in which case
\eqref{eq:stability_bigM} is the Core inequality for $s$; otherwise the
right-hand side drops to at most zero and the inequality is inactive. Validity
must hold for \emph{every} pair $(g,s)$, including pairs with
$s \not\subseteq c(g)$, because a cut persists in the model once generated.

Write $\mathcal{F}$ for the mixed-integer feasible set of
$\mathrm{AGG}_\alpha$.

\begin{proposition}[Validity and exact separation]
\label{prop:stability_valid}
Let $s \subseteq N$ and $g \in G$, and suppose $M_s \ge \max\{V(s), 0\}$ and
$V(\{i\}) \ge 0$ for all $i \in s$; both hold in our model with $M_s = V(s)$,
since the empty solution is feasible and profits are nonnegative.
\begin{itemize}
\item[(i)] \emph{(Validity)} Every mixed-integer point of $\mathcal{F}$ that
encodes an SCS satisfies \eqref{eq:stability_bigM} for the pair $(g,s)$.
\item[(ii)] \emph{(Exact separation)} A mixed-integer point of $\mathcal{F}$ with
$s \subseteq c(g)$ and $\sum_{i \in s} v_i < V(s)$ violates
\eqref{eq:stability_bigM}.
\end{itemize}
\end{proposition}

Exact separation is NP-complete in general (Section~\ref{app:complexity}).
Since the separation oracle computes $V(s)$ exactly, we set $M_s = V(s)$, the
tightest choice in this big-$M$ family when coalition values are nonnegative;
Section~\ref{ss:lifted_cuts} strengthens it further. We call
\eqref{eq:stability_bigM} the \emph{gated stability inequality}, the bracket
being the gate, closed exactly when all of $s$ lands in group $g$. Validity for
every pair lets us impose it for all of them, so define
\begin{equation}
\mathcal{F}^{\mathrm{stab}}
 :=
\bigl\{
  (v,x,y,z,\pi) \in \mathcal{F} :
  \text{\eqref{eq:stability_bigM} holds for all } g \in G
  \text{ and all } s \subseteq N
\bigr\}.
\label{eq:Fstab}
\end{equation}
Let $\mathcal{S} \subseteq \mathbb{R}^N$ be the set of payoff vectors that admit
some stable coalition structure: $v \in \mathcal{S}$ when there is a partition
$CS(N)$ of $N$ such that, for every $c \in CS(N)$, the restriction of $v$ to $c$
lies in the Core of the game restricted to $c$. Note that $\mathcal{S}$ is
defined without reference to \eqref{eq:OCS}.

\begin{theorem}[Exact extended formulation]
\label{thm:CSS_projection}
Under the assumptions of Proposition~\ref{prop:stability_valid}, the
mixed-integer system $(v,x,y,z,\pi) \in \mathcal{F}^{\mathrm{stab}}$ is an exact
extended formulation for SCS payoffs:
\begin{equation}
\operatorname{proj}_v(\mathcal{F}^{\mathrm{stab}}) \;=\; \mathcal{S}.
\label{eq:proj}
\end{equation}
In particular, maximizing $\sum_{i\in N} v_i$ over
$\mathcal{F}^{\mathrm{stab}}$ yields an \emph{optimal stable coalition structure}
(OSCS).
\end{theorem}

\begin{proof}{Proof}
($\subseteq$) Let $(v,x,y,z,\pi) \in \mathcal{F}^{\mathrm{stab}}$ be
mixed-integer feasible. The assignment variables define a partition
$CS(N) = \{c(g) : \sum_i z_{ig} > 0\}$. Constraints \eqref{eq:OCS_cap} to
\eqref{eq:OCS_other} and the domains \eqref{eq:OCS_domains} ensure feasibility of
the coalition-level integer programs, with \eqref{eq:alpha_restriction} supplying
the coalition bound \eqref{eq:coal_IP_general_y_ub}, so
\eqref{eq:OCS_efficiency} gives
$\sum_{i\in c(g)} v_i = \pi_g \le V(c(g))$ whenever $c(g)$ is nonempty. By
Proposition~\ref{prop:stability_valid}(ii) the family \eqref{eq:stability_bigM}
implies $\sum_{i\in s} v_i \ge V(s)$ for all $s \subseteq c(g)$. Taking
$s = c(g)$ supplies the reverse inequality, so each formed coalition is
efficient as well as internally stable, and $v \in \mathcal{S}$.

($\supseteq$) Conversely, let $v \in \mathcal{S}$, with coalition structure
$CS(N)$ and coalition decisions $x,y,\pi$ that are efficient and internally
stable. Encode the structure canonically by assigning each coalition the group
index of its smallest member, which respects
\eqref{eq:OCS_sym_leader_range} and \eqref{eq:OCS_sym_leader_nonempty}. The
resulting point satisfies all OCS constraints, including the nonnegativity in
\eqref{eq:OCS_domains}: internal stability applied to the singleton $\{i\}$
gives $v_i \ge V(\{i\}) \ge 0$, the last inequality because the empty plan is
feasible for player $i$. By internal stability with
Proposition~\ref{prop:stability_valid}(i) it satisfies
\eqref{eq:stability_bigM} for every pair, so
$v \in \operatorname{proj}_v(\mathcal{F}^{\mathrm{stab}})$. \Halmos
\end{proof}

Note that $\mathcal{S}$ need not be a single polyhedron in $v$-space; it is in
general a union of Core polytopes corresponding to different partitions of $N$.
The system $\mathcal{F}^{\mathrm{stab}}$ is therefore an extended formulation
for SCS payoffs rather than a direct polyhedral description in payoff space.

\begin{proposition}[Existence of SCS]
\label{prop:ex_css}
Assume $V(c)$ is finite for every $c \subseteq N$. Then an SCS exists: the
singleton partition $CS(N)=\{\{i\}: i\in N\}$ with payoffs
$v_i = V(\{i\})$.
\end{proposition}

\begin{proof}{Proof}
For each singleton coalition $\{i\}$, efficiency requires $v_i = V(\{i\})$ and
internal stability requires $v_i \ge V(\{i\})$ for the only nonempty
subcoalition. These are compatible, so the Core of the restricted game
$(\{i\},V_{\{i\}})$ is nonempty. Collecting these singletons yields an SCS.
\Halmos
\end{proof}

By Proposition~\ref{prop:ex_css} the feasible region
$\mathcal{F}^{\mathrm{stab}}$ is nonempty, so by
Theorem~\ref{thm:CSS_projection} solving \eqref{eq:OCS} augmented with
\eqref{eq:stability_bigM} optimizes over SCS-feasible solutions and returns
an OSCS.

\subsection{Lifted stability cuts}
\label{ss:lifted_cuts}

The basic inequality \eqref{eq:stability_bigM} can be strengthened using the
singleton values. For $s \subseteq c(g)$ write
$\Delta(s) := V(s) - \sum_{i \in s} V(\{i\})$ for the \emph{synergy} of $s$, the
value gained from cooperation beyond what its members achieve alone. When
$\Delta(s) \ge 0$ we replace \eqref{eq:stability_bigM} by the \emph{lifted
stability cut}
\begin{equation}
  \sum_{i \in s} v_i
  \;\ge\;
  \Delta(s) \Bigl(\sum_{i \in s} z_{ig} - |s| + 1\Bigr)
  \;+\;
  \sum_{i \in s} V(\{i\})\, z_{ig}.
  \label{eq:lifted_cut}
\end{equation}

\begin{proposition}[Validity of the lifted cut]
\label{prop:lifted_valid}
Let $s \subseteq N$ with $\Delta(s) \ge 0$ and $V(\{i\}) \ge 0$ for all
$i \in s$. Then \eqref{eq:lifted_cut} is valid for every mixed-integer point of
$\mathcal{F}^{\mathrm{stab}}$.
\end{proposition}

\begin{proposition}[Pointwise dominance]
\label{prop:lifted_dominance}
Under the hypotheses of Proposition~\ref{prop:lifted_valid} and with
$M_s = V(s)$, the right-hand side of \eqref{eq:lifted_cut} is at least that of
\eqref{eq:stability_bigM} at every point of the box $z_{ig} \in [0,1]$,
$i \in s$. Consequently \eqref{eq:lifted_cut} implies
\eqref{eq:stability_bigM} as a constraint on $(v,z)$.
\end{proposition}

\begin{remark}[Dominance and limits of lifted cuts]
\label{rem:lifted_validity}
When all players of $s$ belong to group $g$, the lifted cut reduces to the Core
inequality $\sum_{i\in s} v_i \ge V(s)$, matching \eqref{eq:stability_bigM}. Elsewhere
in the box, Proposition~\ref{prop:lifted_dominance} shows
\eqref{eq:lifted_cut} dominates \eqref{eq:stability_bigM} pointwise, so the
lifted family is at least as strong.

This dominance should not be read as a tightening of the root LP relaxation.
Inequalities of the payoff-stability form $\sum_{i\in s} v_i \ge f_s(z)$ bound
the payoff variables from below while the objective of \eqref{eq:OCS} maximizes
total payoff, which is why one should not expect them to move the root bound;
we confirm computationally in Section~\ref{s:numerical} that they do not. Any
benefit of \eqref{eq:lifted_cut} over \eqref{eq:stability_bigM} is realized
inside branch and bound, where the dominated inequality can be discarded.
Validity requires $\Delta(s) \ge 0$ and $V(\{i\}) \ge 0$; when either fails, our
implementation falls back to \eqref{eq:stability_bigM}. By
Proposition~\ref{prop:superadd_natural} the fallback never triggers under
full pooling.
\end{remark}

Nothing in Sections~\ref{ss:stability_CSS} and \ref{ss:lifted_cuts} used the
knapsack form of \eqref{eq:coal_IP_general}: those inequalities reference the
instance only through the values $V(s)$ and $V(\{i\})$. The separation of
Section~\ref{sec:separation} does use the rows, and requires the subset
indicators to enter only through right-hand sides and variable bounds, with the
player variables bounded. Section~\ref{ss:general_milp} of the e-companion
states the general MILP setting and checks each component.

\section{Computing an OSCS}\label{sec:computation}

The formulation of Section~\ref{s:formulation} has exponentially many stability
inequalities, one for every pair of a group and a subset of players. They cannot all be stated up front,
so they must be \emph{separated}: generated on demand at candidate solutions.
Our methods combine choices along three axes, taken in turn below: the
separation oracle (Section~\ref{sec:separation}), when the model is re-solved
(Section~\ref{sec:lazy}), and the warm start (Section~\ref{sec:heuristics}).
Table~\ref{tab:method_overview} lists the four resulting configurations, each
run under both the AGG and DIS formulations of
Section~\ref{ss:OCS_formulation}.

\bgroup
\renewcommand{\arraystretch}{0.95}\SingleSpacedXI
\begin{table}[tbp]
\caption{The four OSCS configurations tested in Section~\ref{s:comp_results}.}
\label{tab:method_overview}
\centering
\begin{tabular}{@{}llll@{}}
\toprule
Configuration & Cut management & Separation for $|c| > \tau$ & SCS-feasible warm start \\
\midrule
cutplane & re-solve the model & enumeration (all violated) & --- \\
lazy & branch and cut (lazy) & enumeration (first violated) & --- \\
lazy\_MIP & branch and cut (lazy) & separation MIP (most violated) & --- \\
lazy\_MIP\_ws & branch and cut (lazy) & separation MIP (most violated) & Algorithm~\ref{alg:heuristic} \\
\bottomrule
\end{tabular}
\par\vspace{4pt}
\parbox{0.92\linewidth}{\footnotesize \emph{Note.} For coalitions with
$|c| \le \tau$, every configuration enumerates all subsets and adds all
violated cuts; the column ``Separation for $|c| > \tau$'' is where the
configurations differ. Each configuration is tested under both the AGG and
DIS formulations.}
\end{table}
\egroup
\subsection{Separation strategies}\label{sec:separation}

Given an incumbent $(\bar v, \bar z)$ with realized coalitions
$c(g) = \{i \in N : \bar z_{ig} = 1\}$, the separation oracle must find a subset
$s \subseteq c(g)$ with $\sum_{i \in s} \bar v_i < V(s)$, or prove that none
exists. We use two oracles, combined through a cardinality threshold $\tau$.

\paragraph{Enumeration.}
For $|c(g)| \le \tau$ we enumerate all $2^{|c(g)|}-2$ proper nonempty subsets
and test each. This is exponential in the coalition size but exact. We use
$\tau = 9$ throughout, and the two considerations pull against each other: a
larger threshold covers more of the low-cardinality cuts that stay active
through the search, for reasons taken up at the end of this section, while each increment of $\tau$
doubles the number of subsets to scan. At $\tau = 9$ the $2^9 - 2 = 510$ proper subsets
of a coalition are cheap to check against cached coalition values, and past that
point the cost outruns the extra coverage.

\paragraph{Optimization-based separation.}
For a group $g$, separation seeks the most violated subcoalition,
\begin{equation}
\max_{s \subseteq c(g)} \Bigl\{\, V(s) - \sum_{i \in s} \bar v_i \,\Bigr\}.
\label{eq:sep_problem}
\end{equation}
Written through the value function this looks like one coalition program per
subset. It is not: choosing the subset and solving its coalition program can be
done in a single model. Introduce binaries $\sigma_i \in \{0,1\}$ for
$i \in c(g)$, with $\sigma_i = 1$ exactly when $i \in s$, and let them drive the
right-hand sides and variable bounds of \eqref{eq:coal_IP_general}. Because
$\sigma$ never touches the objective of the coalition program, the selection and
the production decisions can be optimized together.

Both oracles enumerate below the threshold and differ only above it. The
configuration that keeps enumerating, in order of increasing cardinality, and
stops at the first violated subset we call \emph{lazy}; the one that solves
\eqref{eq:sep_single_level} for the most violated subset we call
\emph{lazy\_MIP}.
Each adds a single cut per group, so lazy trades more rounds for cheaper
callbacks.

Fix $c(g)$ and let $R^{c(g)}_{\mathrm{ind}}$ and $R^{c(g)}_{\mathrm{com}}$ be
the sets of \eqref{eq:item_sets} evaluated at $c = c(g)$, and let
$R_i := \{ j \in R : i \in I_j \}$ be the items available to player $i$. The
decision variables are the subset indicators $\sigma_i$, the player-level usage
$x^i_j$, and the coalition-level usage $y_j$ of common items. Mirroring
\eqref{eq:coal_IP_general}, integrality falls on $x^i_j$ for individual items
and on $y_j$ for common ones, while $x^i_j$ for common items stays continuous:
\begin{subequations}\label{eq:sep_single_level}
\begin{align}
\max \quad
& \sum_{i \in c(g)} \sum_{j \in R_i} p^i_j x^i_j
  \;-\;
  \sum_{i \in c(g)} \bar v_i \sigma_i,
\label{eq:sep_obj}\\
\text{s.t.} \quad
& \sum_{j \in R_{\mathrm{ind}}^{c(g)}}
    \sum_{i \in I_j \cap c(g)} a_j x^i_j
  \;+\;
  \sum_{j \in R_{\mathrm{com}}^{c(g)}} a_j y_j
  \;\le\;
  \sum_{i \in c(g)} b^i \sigma_i,
  \label{eq:sep_cap}\\
& y_j
  \;=\;
  \sum_{i \in I_j \cap c(g)} x^i_j,
  \quad \forall j \in R_{\mathrm{com}}^{c(g)},
  \label{eq:sep_link}\\
& B^i x^i \;\le\; d^i \sigma_i,
  \quad \forall i \in c(g),
  \label{eq:sep_private}\\
& 0 \;\le\; x^i_j \;\le\; u^i_j \sigma_i,
  \quad \forall i \in c(g),\ \forall j \in R_i,
  \label{eq:sep_x_ub}\\
& y_j \;\le\; \sum_{i \in I_j \cap c(g)} u^i_j \sigma_i,
  \quad \forall j \in R_{\mathrm{com}}^{c(g)},
  \label{eq:sep_y_ub}\\
& x^i_j \in \mathbb{Z}_+, \ \forall i \in c(g),\ j \in R_{\mathrm{ind}}^{c(g)} \cap R_i; \,\,
  y_j \in \mathbb{Z}_+, \ \forall j \in R_{\mathrm{com}}^{c(g)}; \,\,
  \sigma_i \in \{0,1\},\ \forall i \in c(g),
  \label{eq:sep_int}\\
\blockhead{needed only when $\alpha < 1$; omitted under full pooling:}
& 2\kappa_j \;\le\; \sum_{i \in I_j \cap c(g)} \sigma_i
  \;\le\; \bigl(|I_j \cap c(g)|-1\bigr)\,\kappa_j + 1,
  \quad \forall j \in R_{\mathrm{com}}^{c(g)},
  \label{eq:sep_kappa}\\
& y_j \;\le\; \alpha \!\!\sum_{i \in I_j \cap c(g)}\!\! u^i_j \sigma_i
      + M_j (1-\kappa_j),
  \quad \forall j \in R_{\mathrm{com}}^{c(g)},
  \label{eq:sep_alpha}\\
& \kappa_j \in \{0,1\},
  \quad \forall j \in R_{\mathrm{com}}^{c(g)}.
  \label{eq:sep_kappa_dom}
\end{align}
\end{subequations}
The separation MIP must evaluate $V(s)$ under the \emph{same} coalition policy as
the formulation, and the common and individual designation depends on the
selected subset itself: an item $j \in R_{\mathrm{com}}^{c(g)}$ is common in $s$
only when two or more of its users are selected. That is what the indicator
$\kappa_j$ records in \eqref{eq:sep_kappa} and \eqref{eq:sep_alpha} then
enforces; $\kappa_j$ is the separation counterpart of $w^g_j$, and the two pairs
of constraints mirror \eqref{eq:OCS_w} and \eqref{eq:alpha_restriction}. Under full pooling \eqref{eq:sep_alpha} is implied by \eqref{eq:sep_y_ub}, so
nothing reads $\kappa_j$ and both \eqref{eq:sep_kappa} and \eqref{eq:sep_alpha}
are dropped, as in the formulation.

We keep $M_j = \sum_{i \in I_j \cap c(g)} u^i_j$ here and the analogous constant
in \eqref{eq:alpha_restriction}, although $(1-\alpha)\max_{i} u^i_j$ is also
valid and considerably smaller. Both choices leave the mixed-integer feasible set
unchanged, so every optimum we report is unaffected; only the relaxation differs.
The reported times are therefore conservative, since the tighter constant can
only strengthen the bound.

\begin{proposition}[Single-level separation]
\label{prop:sep_equiv}
Fix a group $g$ and an incumbent $(\bar v, \bar z)$, and assume the player
variables are bounded above by $u$. Then \eqref{eq:sep_single_level} and the separation problem
\eqref{eq:sep_problem} have the same optimal value, and every optimal $\sigma^*$ of
\eqref{eq:sep_single_level} yields a maximizer
$s^* = \{i : \sigma^*_i = 1\}$ of \eqref{eq:sep_problem}. Hence the optimal
value is nonpositive exactly when no subset of $c(g)$ violates stability at
$\bar v$.
\end{proposition}

\begin{proof}{Proof}
Fix $\sigma \in \{0,1\}^{c(g)}$ and let $s = \{i : \sigma_i = 1\}$. Since $u$ is
finite, \eqref{eq:sep_x_ub} forces $x^i = 0$ for every $i \notin s$, and
\eqref{eq:sep_private} then reads $B^i x^i \le 0$, which $x^i = 0$ satisfies;
for $i \in s$ both revert to the original bounds. Constraint
\eqref{eq:sep_cap} pools exactly the budgets of $s$, and
\eqref{eq:sep_kappa} and \eqref{eq:sep_alpha} make the common and individual
designation the one induced by $s$. The remaining constraints are those of
\eqref{eq:coal_IP_general} for coalition $s$, so maximizing
\eqref{eq:sep_obj} over $(x,y)$ at this fixed $\sigma$ gives
$V(s) - \sum_{i \in s} \bar v_i$. Maximizing over $\sigma$ as well therefore
returns $\max_{s \subseteq c(g)} \{V(s) - \sum_{i \in s} \bar v_i\}$, which is
\eqref{eq:sep_problem}; the maximizers correspond. The last claim follows
because $s = \emptyset$ is feasible with objective zero. \Halmos
\end{proof}

The second part is what justifies accepting an incumbent: a nonpositive optimum
is a \emph{certificate} of internal stability, not a failure to find a cut. Its
hypotheses are also exactly the conditions under which the framework extends
beyond the resource allocation game (Section~\ref{ss:general_milp} of the
e-companion). One may additionally impose $\sum_{i \in c(g)} \sigma_i \ge 1$, since
$s = \emptyset$ is feasible with objective zero and yields no cut. The full
subset $s = c(g)$ must \emph{not} be excluded. Its inequality is the efficiency
requirement $\sum_{i \in c(g)} v_i \ge V(c(g))$, and
\eqref{eq:OCS_efficiency} alone does not enforce it at a mixed-integer feasible
point: it gives only $\sum_{i \in c(g)} v_i = \pi_g \le V(c(g))$, since the
group's production need not be optimal there. Separating $s = c(g)$ is what makes
every accepted incumbent efficient, hence SCS-feasible. Our implementation omits
it: the objective maximizes total payoff, which \eqref{eq:OCS_efficiency} ties to
production, so efficiency holds at any optimum; time-limit incumbents are
verified directly instead.%

\paragraph{Strength of the most violated cut.}
The cut built from the maximizer $s^*$ is usually not the strongest cut for the
rest of the search. A stability cut for $s$ is non-vacuous only when its gate is
closed, that is when all of $s$ is assigned to one group, and the number of
coalition structures in which that happens falls rapidly with $|s|$. Since the
violation tends to grow with cardinality, the most violated subset is often
large, and its cut rarely binds again later in the tree, whereas
low-cardinality cuts such as $v_i \ge V(\{i\})$ are weaker at the incumbent but
activate in far more configurations. This inverts classical cutting-plane
practice, where a most violated cover inequality, strengthened by lifting, is
valid for the whole polytope and can bind anywhere in the tree
\citep{crowder1983solving,gu1998lifted}. Hence the hybrid: enumeration below
$\tau$, where low-cardinality cuts keep paying, and the separation MIP above it,
where only the certificate is available.

\begin{algorithm}[htbp]
\caption{\textsc{Separate}$(\bar v, \bar z)$, run by the solver at an
integer-feasible point}
\label{alg:separate}
\begin{algorithmic}[1]
\small
\REQUIRE incumbent $(\bar v, \bar z)$; enumeration threshold $\tau$ (we use
  $\tau = 9$)
\ENSURE lazy stability cuts for every violated subcoalition found
\FOR{each nonempty realized coalition $c(g)$}
  \IF{$|c(g)| \le \tau$}
    \STATE for every violated proper nonempty $s \subsetneq c(g)$, add
      \eqref{eq:lifted_cut} as a lazy cut, or \eqref{eq:stability_bigM} where
      the lifted cut is invalid \COMMENT{up to $2^\tau - 2$ cuts}
  \ELSE
    \STATE \label{line:sepmip} solve \eqref{eq:sep_single_level}, built from
      $c(g)$ and $\bar v$; if its value is positive, add that same cut at its
      maximizer $s^*$ \COMMENT{one cut}
  \ENDIF
\ENDFOR
\end{algorithmic}
\end{algorithm}

\subsection{Iterative re-solving versus lazy constraints}
\label{sec:lazy}

The baseline is the classical iterative cutting-plane loop, which we call
\emph{cutplane}: solve \eqref{eq:OCS} to optimality, enumerate the
subcoalitions of every formed group, add all violated stability cuts, and
re-solve, stopping when no violation remains, at which point the incumbent is a
provably stable optimum. Its drawback is equally classical: every round
discards the branch-and-bound tree, and separation happens only at optimal
solutions.

Our remaining configurations embed separation in the branch-and-bound tree using
lazy constraints. Whenever an integer-feasible solution $(\bar v,\bar z)$ is
found, at a \texttt{MIPSOL} callback, we extract the implied coalition structure
$CS(N)=\{c(g): \bar z_{ig}=1\}$, run Algorithm~\ref{alg:separate} for each $c(g)$, and
add any violated stability inequalities through Gurobi's lazy-constraint
callback \texttt{cbLazy} \citep{gurobi}. For each violated subset $s$ we add the
lifted cut \eqref{eq:lifted_cut} when the synergy condition $\Delta(s) \ge 0$
holds and the basic cut \eqref{eq:stability_bigM} otherwise; since the lifted
cut dominates the basic one wherever it is valid, this is our default
throughout. Every accepted incumbent then satisfies the separated
inequalities, and the search
optimizes over stable coalition structures directly.

\subsection{An SCS-feasible primal heuristic}\label{sec:heuristics}

The warm start we supply to lazy\_MIP is \emph{provably} SCS-feasible: every
coalition it returns carries a verified nonempty Core before the partition is
assembled. Algorithm~\ref{alg:heuristic} states the procedure, with phase
budgets $T_1$ and $T_2$ reported in Section~\ref{s:comp_results}.

Two features are worth pointing out. First, the pool $\mathcal{P}$ collects
\emph{every} coalition that passes a Core check, not only the merges that are
executed, so one pass leaves Phase~2 with far more material than the greedy
merge path alone. Second, the Core check splits at the same threshold $\tau$
used for separation. For $|c| \le \tau$ it enumerates every proper nonempty
subcoalition and tests \eqref{eq:core_def} directly. For $|c| > \tau$ it
generates rows instead: solve the Core LP over the subcoalitions found so far,
call \eqref{eq:sep_single_level} at the resulting allocation, and add the most
violated subcoalition, stopping when separation finds no violation, which
certifies a nonempty Core, or when the LP goes infeasible, which certifies an
empty one. If the budget expires before either, the check returns
\emph{inconclusive} and the merging loop treats it as a failure, so a coalition
enters $\mathcal{P}$ only with a proof.

Phase~2 selects from the pool by solving the set partitioning problem
\begin{equation}
\max \ \sum_{c \in \mathcal{P}} V(c)\, \chi_c
\quad \text{s.t.} \quad
\sum_{c \in \mathcal{P}:\, i \in c} \chi_c = 1 \ \ \forall i \in N,
\qquad \chi_c \in \{0,1\},
\label{eq:setpart}
\end{equation}
which is always feasible because $\mathcal{P}$ contains the singletons.

\begin{algorithm}[htbp]
\caption{SCS-feasible warm start}
\label{alg:heuristic}
\begin{algorithmic}[1]
\small
\REQUIRE player set $N$, threshold $\tau$, phase budgets $T_1$ and $T_2$
\ENSURE coalition structure $CS^*$ and payoffs $\tilde v$ encoding an SCS
\STATE evaluate $V(\{i\})$ for every $i \in N$
\STATE $CS \gets \{\{i\} : i \in N\}$; \quad
       $\mathcal{P} \gets \{\{i\} : i \in N\}$; \quad
       $(CS^*,\tilde v) \gets \bigl(CS, (V(\{i\}))_{i \in N}\bigr)$
       \COMMENT{incumbent, already an SCS}
\vspace{2pt}\hrule\vspace{2pt}
\STATE \textbf{Phase 1: Core-verified merging}, within budget $T_1$
\WHILE{$|CS| \ge 2$ and $T_1$ is not exhausted}
  \STATE $\mathcal{C} \gets \{(c_1,c_2) \subseteq CS :
         V(c_1 \cup c_2) - V(c_1) - V(c_2) > 0\}$;
         \textbf{if} $\mathcal{C} = \emptyset$ \textbf{then break}
  \FOR{$(c_1,c_2) \in \mathcal{C}$ in order of decreasing gain}
    \STATE \textbf{if} \textsc{CoreNonempty}$(c_1 \cup c_2)$ returns
           \textbf{true then} $\mathcal{P} \gets \mathcal{P} \cup \{c_1 \cup c_2\}$
           \COMMENT{kept whether or not the merge is taken}
  \ENDFOR
  \STATE \textbf{if} no candidate passed \textbf{then break}
  \STATE $(c_1^*,c_2^*) \gets$ passing pair of largest gain; \quad
         $CS \gets (CS \setminus \{c_1^*,c_2^*\}) \cup \{c_1^* \cup c_2^*\}$
\ENDWHILE
\vspace{2pt}\hrule\vspace{2pt}
\STATE \textbf{Phase 2: set partitioning}, within budget $T_2$
\STATE $CS' \gets$ best solution of \eqref{eq:setpart} over $\mathcal{P}$ found
       within $T_2$, seeded with the singleton partition
\STATE \textbf{if} a Core allocation $\tilde v'$ is obtained for every
       $c \in CS'$ \textbf{then} $(CS^*,\tilde v) \gets (CS',\tilde v')$
       \COMMENT{incumbent replaced only when complete}
\RETURN $(CS^*, \tilde v)$
\end{algorithmic}
\end{algorithm}

\begin{proposition}[Warm start: feasibility and cost]
\label{prop:heuristic_feasible}
Algorithm~\ref{alg:heuristic} evaluates the $|N|$ singleton values during
initialization. From the end of initialization onward it maintains a pair
$(CS^*,\tilde v)$ encoding an SCS, and returns that pair whenever it stops, for
any budgets $T_1, T_2 \ge 0$. Beyond initialization it performs at most
$\binom{|N|+1}{3}$ Core checks, and each check on a coalition of size at most
$\tau$ evaluates at most $2^\tau - 2$ coalition values.
\end{proposition}

The initialization cost cannot be avoided: Core membership constrains the
singleton subcoalitions of every formed coalition, so no output can be certified
without $V(\{i\})$ for each $i$. Past that point the heuristic is \emph{anytime},
and the budgets control how good the partition is rather than whether one is
returned. The cost bound makes $O(|N|^3)$ Core checks; a check on a coalition
of size at most $\tau$ evaluates at most $2^\tau - 2$ coalition values, and
larger checks use time-budgeted row generation. This stands against the
$2^{|N|}-1$ values the exhaustive baseline of Section~\ref{ss:res_cost} needs
before it can begin. It
is loose in practice: at $|N| = 30$ it permits $4{,}495$ checks and the runs use
210 at $\alpha = 0.5$ and 774 at $\alpha = 1.0$.

\paragraph{Limits of the guarantee.}
If no positive-gain pair
passes a Core check in the first round the heuristic returns the singleton
partition, worth $\sum_i V(\{i\})$; we claim no approximation ratio, and none is
available without further assumptions on $V$. The guarantee is conservative by
construction, since an inconclusive Core check counts as a failure. It is also
only as sharp as the LP tolerance: the near-degenerate Cores of
Table~\ref{tab:refinement_summary} sit close to the boundary at which that
tolerance decides the answer.

\subsection{Per-coalition solution concepts and payoff refinement}
\label{ss:coalition_refinement}

The system $(v,x,y,z,\pi) \in \mathcal{F}^{\mathrm{stab}}$ fixes the coalition
structure but not the payoffs within a coalition: by construction of SCS the
restricted game $(c,V_c)$ with $V_c(s) := V(s)$ has a nonempty Core, and any of
its Core allocations may replace the one the solver returns without changing the
structure or the total welfare. We therefore report, for every formed coalition,
the Shapley value \citep{shapley1953value}; the nucleolus
\citep{schmeidler1969nucleolus}, which lies in the Core whenever it is nonempty
and hence for every formed coalition, computed by the standard
sequential-LP scheme \citep{maschler1979geometric}; and the least-core value
$\varepsilon(c)$, whose sign certifies Core nonemptiness and whose magnitude
measures the slack of the tightest Core constraint. Formed coalitions are small
enough that this is cheap. Definitions are in Section~\ref{app:details}, the
computation in Section~\ref{app:addl}, and the results in
Section~\ref{s:numerical}.

\subsection{Algorithm summary}
\label{sec:algsummary}

Algorithm~\ref{alg:ocss_lazy} assembles the pieces into lazy\_MIP\_ws. Its outer
loop is ordinary branch and cut, and we state it only so that the order of
operations is in one place; the content is in
Algorithms~\ref{alg:separate} and \ref{alg:heuristic}. The formulation is an input,
not a fixed choice: nothing in the separation or the warm start depends on how
the payoffs are linearized, and Section~\ref{ss:res_formulation} measures the
difference between AGG and DIS. One step is easy to get wrong. Line~2 must relabel $CS^*$ canonically, because the symmetry-breaking
constraints
\eqref{eq:OCS_sym_leader_range}--\eqref{eq:OCS_sym_leader_nonempty} admit exactly
one group labeling of each partition, so a warm start that labels coalitions any
other way is infeasible and the solver discards it, losing the whole benefit of
Algorithm~\ref{alg:heuristic}. The remaining configurations of
Table~\ref{tab:method_overview} drop the warm start, giving lazy\_MIP, or replace
the MIP separation in line~\ref{line:sepmip} of Algorithm~\ref{alg:separate} with
first-violation enumeration.

\begin{algorithm}[htbp]
\caption{OSCS via lazy constraints with an SCS-feasible warm start
(lazy\_MIP\_ws)}
\label{alg:ocss_lazy}
\begin{algorithmic}[1]
\small
\REQUIRE CIPG instance with pooling parameter $\alpha$; an
  $\mathrm{OCS}_\alpha$ formulation, $\mathrm{AGG}_\alpha$ or
  $\mathrm{DIS}_\alpha$; threshold $\tau$; time limit $T$; phase budgets $T_1$
  and $T_2$
\ENSURE an OSCS, or the best SCS-feasible incumbent found
\STATE $(CS^*, \tilde v) \gets$ Algorithm~\ref{alg:heuristic} with budgets
  $T_1, T_2$
\STATE encode $CS^*$ as $\tilde z$ canonically, giving each coalition the group
  index of its smallest member, as the symmetry-breaking scheme of
  Section~\ref{ss:OCS_formulation} requires
\STATE warm start the formulation at $(\tilde v, \tilde z)$ and solve it with the
  time remaining, separating stability cuts with Algorithm~\ref{alg:separate} at
  every integer-feasible incumbent (\texttt{MIPSOL})
\RETURN best incumbent, an OSCS if optimality is certified
\end{algorithmic}
\end{algorithm}

\section{Experimental design and results}
\label{s:comp_results}\label{s:numerical}

We instantiate the framework as a \emph{Cooperative Knapsack Game} (CKG): each
player solves an individual knapsack of the form \eqref{eq:coal_IP_general} with
$c = \{i\}$, and coalitions pool their budgets under the $\alpha$-restriction
\eqref{eq:alpha_restriction}. Instances use $|R| = 50$ items, availability
probability $q = 0.5$, profits $p^i_j \sim U\{10,100\}$, weights
$a_j \sim U\{5,10\}$, capacities $b^i = \rho \sum_{j \in R_i} a_j$ with
$\rho = 0.3$ where $R_i = \{j : i \in I_j\}$, and $u^i_j = 1$, so each player
solves a $0$--$1$ knapsack. There are no private constraints, so
$B^i x^i \le d^i$ is vacuous. The individual coalition programs are
inexpensive here (Remark~\ref{rem:not-the-knapsack}), so the benchmark
isolates coalition generation and stability certification, not single-knapsack
difficulty. We report both regimes, $\alpha \in \{0.5,1.0\}$, one
instance per pair. Every configuration receives the grand coalition as a
solver MIP start, which lazy\_MIP\_ws replaces with the heuristic's partition; we checked coalition-level
efficiency for every reported incumbent afterwards by re-solving $V(c)$.

All runs used an Intel Core i9-13900F (24 cores) with 64\,GB of RAM, in Python
with Gurobi 13.0.0 \citep{gurobi} on up to 16 threads, and the enumeration
threshold $\tau = 9$ throughout. For lazy\_MIP\_ws we give
Algorithm~\ref{alg:heuristic} the budgets $T_1 = 2T/5$ and $T_2 = T/5$, leaving
the solver at least $2T/5$ and more when the heuristic finishes early; the
split worked well across our instances, and
Proposition~\ref{prop:heuristic_feasible}'s guarantee holds for any choice. Every OSCS
configuration uses the lifted cuts
\eqref{eq:lifted_cut} by default and falls back to \eqref{eq:stability_bigM}
when the synergy condition fails. The full grid ran in one session, and the
lazy\_MIP\_ws configurations were re-run in a second session after a coalition
value cache was shared between the heuristic and the solver, so timings are
comparable across methods and formulations.

Experiment~1 covers $n = 2, 4, \ldots, 20$ with $T = 1800$ seconds and tests
five methods under both formulations: the OCS formulation \eqref{eq:OCS} without
stability constraints, as a baseline, and the four OSCS configurations of
Table~\ref{tab:method_overview}. Experiment~2 covers $n = 22, 24, \ldots, 30$
with $T = 3600$ seconds and drops cutplane, for reasons given in
Section~\ref{ss:res_tractable}. Throughout, ``TL'' marks a timeout and the
adjacent Gap column gives the gap remaining at that point.

\bgroup
\renewcommand{\arraystretch}{0.95}\SingleSpacedXI
\setlength{\tabcolsep}{4pt}
\begin{table}[tbp]
\caption{OCS formulation comparison, aggregated (AGG) vs.\ disaggregated (DIS): integer performance and root LP bounds.}
\label{tab:OCS_formulation}
\centering
\small
\begin{tabular}{rcrrrrrrcrrrr}
\toprule
& & \multicolumn{3}{c}{\textbf{OCS (AGG)}} & \multicolumn{3}{c}{\textbf{OCS (DIS)}} & & \multicolumn{4}{c}{\textbf{Root LP bounds}} \\
\cmidrule(lr){3-5} \cmidrule(lr){6-8} \cmidrule(lr){10-13}
$n$ & $\alpha$ & Obj & Time & Gap & Obj & Time & Gap & Better & $\mathrm{LP}_{\mathrm{AGG}}$ & $\mathrm{LP}_{\mathrm{DIS}}$ & Tight.\ (\%) & IP gap (\%) \\
\midrule
2 & 0.5 & 1340 & 0.04 & --- & 1340 & 0.01 & --- & DIS & 2,531 & 1,376.5 & 45.6 & 2.72 \\
2 & 1.0 & 1365 & 0.01 & --- & 1365 & 0.00 & --- & DIS &  &  &  & 0.84 \\
4 & 0.5 & 3037 & 0.3 & --- & 3037 & 0.1 & --- & DIS & 13,795 & 3,114.0 & 77.4 & 2.54 \\
4 & 1.0 & 3105 & 0.09 & --- & 3105 & 0.06 & --- & DIS &  &  &  & 0.29 \\
6 & 0.5 & 3622 & 0.3 & --- & 3622 & 0.4 & --- & AGG & 29,319 & 3,699.9 & 87.4 & 2.15 \\
6 & 1.0 & 3696 & 0.1 & --- & 3696 & 0.09 & --- & DIS &  &  &  & 0.10 \\
8 & 0.5 & 5487 & 1.5 & --- & 5487 & 1.4 & --- & DIS & 61,197 & 5,614.7 & 90.8 & 2.33 \\
8 & 1.0 & 5609 & 0.4 & --- & 5609 & 0.2 & --- & DIS &  &  &  & 0.10 \\
10 & 0.5 & 7483 & 11 & --- & 7483 & 4.1 & --- & DIS & 110,032 & 7,567.9 & 93.1 & 1.13 \\
10 & 1.0 & 7564 & 1.0 & --- & 7564 & 0.6 & --- & DIS &  &  &  & 0.05 \\
12 & 0.5 & 8244 & 14 & --- & 8244 & 12 & --- & DIS & 150,727 & 8,331.4 & 94.5 & 1.06 \\
12 & 1.0 & 8325 & 1.1 & --- & 8325 & 3.4 & --- & AGG &  &  &  & 0.08 \\
14 & 0.5 & 8855 & 57 & --- & 8855 & 43 & --- & DIS & 193,826 & 8,954.8 & 95.4 & 1.13 \\
14 & 1.0 & 8953 & 1.2 & --- & 8953 & 0.3 & --- & DIS &  &  &  & 0.02 \\
16 & 0.5 & 10250 & 94 & --- & 10250 & 113 & --- & AGG & 266,031 & 10,344.9 & 96.1 & 0.93 \\
16 & 1.0 & 10341 & 55 & --- & 10341 & TL & 0.02 & AGG &  &  &  & 0.04 \\
18 & 0.5 & 11434 & TL & 0.23 & 11425 & TL & 0.38 & AGG & 346,662 & 11,563.0 & 96.7 & 1.13 \\
18 & 1.0 & 11560 & TL & 0.01 & 11560 & TL & 0.02 & AGG &  &  &  & 0.03 \\
20 & 0.5 & 12906 & TL & 0.17 & 12896 & TL & 0.30 & AGG & 441,402 & 12,981.3 & 97.1 & 0.58 \\
20 & 1.0 & 12978 & TL & 0.02 & 12978 & TL & 0.02 & --- &  &  &  & 0.03 \\
\midrule
22 & 0.5 & 15135 & TL & 0.79 & 15135 & TL & 0.73 & DIS & 567,079 & 15,342.2 & 97.3 & 1.37 \\
22 & 1.0 & 15341 & 3.3 & --- & 15341 & 0.9 & --- & DIS &  &  &  & 0.01 \\
24 & 0.5 & 16374 & TL & 0.32 & 16353 & TL & 0.31 & DIS & 672,802 & 16,445.6 & 97.6 & 0.44 \\
24 & 1.0 & 16439 & TL & 0.04 & 16439 & TL & 0.04 & --- &  &  &  & 0.04 \\
26 & 0.5 & 17906 & TL & 0.67 & 17896 & TL & 0.67 & --- & 779,708 & 18,036.9 & 97.7 & 0.73 \\
26 & 1.0 & 18033 & TL & 0.02 & 18033 & TL & 0.02 & --- &  &  &  & 0.02 \\
28 & 0.5 & 18768 & TL & 0.73 & 18760 & TL & 0.64 & DIS & 922,279 & 18,945.4 & 97.9 & 0.94 \\
28 & 1.0 & 18941 & TL & 0.02 & 18941 & TL & 0.02 & --- &  &  &  & 0.02 \\
30 & 0.5 & 20708 & TL & 0.79 & 20716 & TL & 0.75 & DIS & 1,077,559 & 20,876.5 & 98.1 & 0.77 \\
30 & 1.0 & 20876 & 6.8 & --- & 20876 & 1.6 & --- & DIS &  &  &  & 0.00 \\
\bottomrule
\end{tabular}
\par\vspace{4pt}
\parbox{0.92\linewidth}{\footnotesize \emph{Note.} Obj, wall-clock time
(s), and MIP gap (\%) at timeout (TL); ``---'' in a Gap column means solved
to optimality. ``Tight.''\ $= (\mathrm{LP}_{\mathrm{AGG}} -
\mathrm{LP}_{\mathrm{DIS}})/\mathrm{LP}_{\mathrm{AGG}}$; ``IP gap'' is the
relative gap of $\mathrm{LP}_{\mathrm{DIS}}$ to the best known integer OCS
value of the row; root bounds are reported once per $n$. The ``Better''
convention and remaining details are given in the text. Time limit:
$1800$\,s ($n \le 20$); $3600$\,s ($n \ge 22$).}
\end{table}
\egroup

\subsection{Does disaggregation pay?}
\label{ss:res_formulation}

For the OCS problem, yes, and by a wide margin at the root.
Table~\ref{tab:OCS_formulation} reports both on all 30 instances. Its ``Better''
column ranks the formulation with the smaller time when both certify, the
certifying formulation when only one does, and the smaller reported gap when
both time out, with ``---'' when the two coincide. Timeout rows report best
incumbents rather than proven optima.

\paragraph{Root bounds.}
The disaggregated bound is tighter on every instance, and the margin grows from
$45.6\%$ at $n = 2$ to $98.1\%$ at $n = 30$. The mechanism is the payoff
double-counting of the proof of Proposition~\ref{prop:disagg_LP}: the only
upper bound on a payoff in $F_{\mathrm{agg}}$ is
$v_i \le \phi_{ig} + M(1 - z_{ig})$, so a fractional assignment collects a
payoff in several groups at once. Payoffs decouple from production: the
observed bound grows like $(n-1)M$, quadratic in $n$ since $M$ grows linearly,
with a ratio of $9.79$ between $n = 30$ and $n = 10$. The linking constraint in \eqref{eq:disagg}
removes exactly this mechanism: total payoff equals total fractional production,
which the pooled resources bound, so $\mathrm{LP}_{\mathrm{DIS}}$ sits just above
the integer value. In absolute terms it lies within $0.84\%$ of the best known integer value at
$\alpha = 1.0$ and within $2.72\%$ at $\alpha = 0.5$, where the LP does not see
the $\alpha$-restriction, while the aggregated bound exceeds it by more than a
factor of $50$ at $n = 30$. The two $\alpha$ values have nearly identical
root bounds: the
indicators $w^g_j$ relax and the restriction goes inactive, so the table reports
one per $n$ (the disaggregated bounds differ by less than $0.06\%$ between the
two $\alpha$ values; the table reports $\alpha = 1.0$).

The margin measures the formulations as stated; a global bound on the objective
closes most of it but is counterproductive inside branch and bound, as
Section~\ref{app:objbound} explains.

Adding all basic and lifted stability cuts for subsets of size up to two a
priori leaves the root LP value unchanged on every instance, as argued in
Remark~\ref{rem:lifted_validity}.

\paragraph{Solve times.}
The runtime margins are far more modest than the bound margins. DIS is faster in
14 of the 17 instances where both formulations certify. At $\alpha = 1.0$ the
$\alpha$-restriction reduces to the full-pooling bound, so $V$ is superadditive by
Proposition~\ref{prop:superadd_natural} and the grand coalition is known to be
optimal in advance; those rows therefore measure certification speed rather than
search. DIS usually certifies faster, taking under two seconds at $n = 22$ and
$n = 30$ where AGG takes several times as long. The exception is $n = 16$, where
AGG certifies in 55 seconds and DIS times out. At $n \in \{18, 20, 24, 26, 28\}$
neither formulation certifies within the limit, both stalling at gaps between
$0.01$ and $0.04\%$: the last fraction of a percent is hard for both. Where both
certify they agree.

\paragraph{The reversal for OSCS.}
Once stability constraints enter, the ranking inverts: within lazy\_MIP, AGG is
faster in 9 of the 15 cells where both certify (Table~\ref{tab:SOCS_formulation}
in the e-companion). More seriously, without a warm start DIS fails to produce
good incumbents at scale on this grid: on seven of the eight tested
configurations with $n \ge 24$, and already at $n = 20$, $\alpha = 0.5$, it
returns near-trivial partitions with objectives of 303 to 375 against expected
values of 12{,}900 to 20{,}900, leaving gaps above $4{,}000\%$. Despite its
tighter relaxation and smaller constraint system
(Proposition~\ref{prop:disagg_LP}), DIS generates poor incumbents on these
large OSCS instances, an observed solver-search behavior whose cause our
evidence does not establish. With the warm start the failure never occurs,
but AGG is the safer default on the reported grid.

\bgroup
\renewcommand{\arraystretch}{0.95}\SingleSpacedXI
\setlength{\tabcolsep}{1.5pt}
\begin{table}[tbp]
\caption{OSCS separation strategy comparison (AGG formulation).}
\label{tab:SOCS_methods}
\centering
\scriptsize
\begin{tabular}{rcrrrrrrrrrrrrrrrrrrrr}
\toprule
& & \multicolumn{4}{c}{\textbf{cutplane}} & \multicolumn{4}{c}{\textbf{lazy}} & \multicolumn{4}{c}{\textbf{lazy\_MIP}} & \multicolumn{6}{c}{\textbf{lazy\_MIP\_ws}} & & \\
\cmidrule(lr){3-6} \cmidrule(lr){7-10} \cmidrule(lr){11-14} \cmidrule(lr){15-20}
$n$ & $\alpha$ & Obj & Time & Gap & Cuts & Obj & Time & Gap & Cuts & Obj & Time & Gap & Cuts & Obj & Time & Gap & Cuts & Heur & HGap & $|CS|$ & avg.\ $|c|$ \\
\midrule
2 & 0.5 & 1340 & \textbf{0.01} & --- & 0 & 1340 & \textbf{0.01} & --- & 1 & 1340 & \textbf{0.01} & --- & 1 & 1340 & 0.09 & --- & 0 & 0.0 & $<$0.01 & 2 & 1.0 \\
2 & 1.0 & 1365 & 0.02 & --- & 1 & 1365 & 0.02 & --- & 2 & 1365 & \textbf{0.01} & --- & 2 & 1365 & \textbf{0.01} & --- & 0 & 0.0 & $<$0.01 & 1 & 2.0 \\
4 & 0.5 & 3037 & 0.8 & --- & 2 & 3037 & \textbf{0.2} & --- & 7 & 3037 & \textbf{0.2} & --- & 7 & 3037 & \textbf{0.2} & --- & 0 & 0.0 & $<$0.01 & 3 & 1.3 \\
4 & 1.0 & 3105 & 0.5 & --- & 14 & 3105 & \textbf{0.1} & --- & 14 & 3105 & \textbf{0.1} & --- & 14 & 3105 & \textbf{0.1} & --- & 0 & 0.0 & $<$0.01 & 1 & 4.0 \\
6 & 0.5 & 3622 & 1.0 & --- & 4 & 3622 & \textbf{0.5} & --- & 35 & 3622 & \textbf{0.5} & --- & 35 & 3622 & 0.6 & --- & 0 & 0.0 & $<$0.01 & 4 & 1.5 \\
6 & 1.0 & 3696 & 0.8 & --- & 54 & 3696 & 0.4 & --- & 54 & 3696 & 0.4 & --- & 54 & 3696 & \textbf{0.3} & --- & 0 & 0.1 & $<$0.01 & 1 & 6.0 \\
8 & 0.5 & 5487 & 3.2 & --- & 4 & 5487 & 2.2 & --- & 158 & 5487 & 2.3 & --- & 158 & 5487 & \textbf{2.0} & --- & 0 & 0.1 & $<$0.01 & 4 & 2.0 \\
8 & 1.0 & 5606 & 5.4 & --- & 243 & 5606 & \textbf{1.6} & --- & 225 & 5606 & 1.8 & --- & 225 & 5606 & 2.1 & --- & 250 & 0.6 & 0.11 & 2 & 4.0 \\
10 & 0.5 & 7483 & 13 & --- & 4 & 7483 & 12 & --- & 136 & 7483 & 12 & --- & 137 & 7483 & \textbf{4.0} & --- & 0 & 0.2 & $<$0.01 & 6 & 1.7 \\
10 & 1.0 & 7564 & 15 & --- & 917 & 7564 & 9.2 & --- & 257 & 7564 & 9.4 & --- & 257 & 7564 & \textbf{3.1} & --- & 0 & 0.6 & $<$0.01 & 2 & 5.0 \\
12 & 0.5 & 8244 & 29 & --- & 8 & 8244 & 18 & --- & 90 & 8244 & 19 & --- & 91 & 8244 & \textbf{17} & --- & 0 & 0.2 & $<$0.01 & 6 & 2.0 \\
12 & 1.0 & 8325 & 756 & --- & 3004 & 8325 & 77 & --- & 1347 & 8325 & 75 & --- & 1348 & 8325 & \textbf{13} & --- & 0 & 2.1 & $<$0.01 & 2 & 6.0 \\
14 & 0.5 & 8855 & 188 & --- & 23 & 8855 & 150 & --- & 615 & 8855 & 147 & --- & 616 & 8855 & \textbf{72} & --- & 44 & 0.2 & 0.08 & 6 & 2.3 \\
14 & 1.0 & 8953 & 352 & --- & 12488 & 8953 & 140 & --- & 2076 & 8953 & \textbf{133} & --- & 2316 & 8953 & 172 & --- & 3272 & 3.6 & 0.15 & 2 & 7.0 \\
16 & 0.5 & 10250 & 337 & --- & 15 & 10250 & 167 & --- & 353 & 10250 & 162 & --- & 354 & 10250 & \textbf{74} & --- & 10 & 0.5 & $<$0.01 & 9 & 1.8 \\
16 & 1.0 & 10341$^\ddagger$ & TL & n/a & 36765 & 10341 & TL & 0.02 & 3094 & 10341 & TL & 0.02 & 4533 & 10341 & \textbf{1800} & --- & 2712 & 24.1 & $<$0.01 & 2 & 8.0 \\
18 & 0.5 & 11434$^\ddagger$ & TL & n/a & 0 & 11425 & TL & \textbf{0.38} & 187 & 11425 & TL & 0.39 & 188 & 11427 & TL & 0.40 & 19 & 0.7 & 0.07 & 10 & 1.8 \\
18 & 1.0 & 11560$^\ddagger$ & TL & n/a & 0 & 11557 & TL & 0.05 & 6378 & 11559 & TL & \textbf{0.03} & 4809 & 11560 & TL & \textbf{0.03} & 3030 & 28.4 & 0.10 & 3 & 6.0 \\
20 & 0.5 & 12906$^\ddagger$ & TL & n/a & 0 & 12905 & TL & 0.27 & 114 & 12905 & TL & 0.28 & 115 & 12908 & TL & \textbf{0.12} & 41 & 1.4 & 0.09 & 10 & 2.0 \\
20 & 1.0 & 12978$^\ddagger$ & TL & n/a & 0 & 12975 & TL & 0.05 & 6920 & 12970 & TL & 0.09 & 8692 & 12978 & TL & \textbf{0.02} & 6600 & 17.4 & 0.12 & 3 & 6.7 \\
\midrule
22 & 0.5 &  &  &  &  & 15125 & TL & 0.91 & 62 & 15125 & TL & 0.91 & 63 & 15151 & TL & \textbf{0.72} & 23 & 1.8 & $<$0.01 & 11 & 2.0 \\
22 & 1.0 &  &  &  &  & 15322 & TL & 0.13 & 9654 & 15327 & TL & 0.10 & 8698 & 15340 & TL & \textbf{0.01} & 7821 & 98.5 & $<$0.01 & 3 & 7.3 \\
24 & 0.5 &  &  &  &  & 16293 & TL & 0.76 & 67 & 16278 & TL & 0.83 & 70 & 16315 & TL & \textbf{0.64} & 32 & 2.2 & 0.07 & 12 & 2.0 \\
24 & 1.0 &  &  &  &  & 16435 & TL & 0.06 & 7793 & 16437 & TL & \textbf{0.05} & 5125 & 16431 & TL & 0.09 & 6561 & 10.3 & 0.05 & 5 & 4.8 \\
26 & 0.5 &  &  &  &  & 17819 & TL & 1.1 & 94 & 17816 & TL & 1.1 & 98 & 17862 & TL & \textbf{0.87} & 28 & 3.0 & $<$0.01 & 13 & 2.0 \\
26 & 1.0 &  &  &  &  & 17996 & TL & 0.23 & 6840 & 18007 & TL & 0.17 & 13090 & 18021 & TL & \textbf{0.09} & 7327 & 106.4 & 0.04 & 5 & 5.2 \\
28 & 0.5 &  &  &  &  & 18735 & TL & 0.89 & 76 & 18735 & TL & 0.89 & 77 & 18776 & TL & \textbf{0.60} & 30 & 3.6 & $<$0.01 & 14 & 2.0 \\
28 & 1.0 &  &  &  &  & 18908 & TL & 0.20 & 14429 & 18916 & TL & 0.16 & 13588 & 18930 & TL & \textbf{0.08} & 4179 & 65.5 & $<$0.01 & 6 & 4.7 \\
30 & 0.5 &  &  &  &  & 20653 & TL & 1.1 & 363 & 20653 & TL & 1.1 & 364 & 20676 & TL & \textbf{0.94} & 28 & 4.5 & $<$0.01 & 17 & 1.8 \\
30 & 1.0 &  &  &  &  & 20732 & TL & 0.70 & 18228 & 20854 & TL & 0.11 & 14181 & 20866 & TL & \textbf{0.05} & 5393 & 60.1 & $<$0.01 & 6 & 5.0 \\
\bottomrule
\end{tabular}
\par\vspace{4pt}
\parbox{0.92\linewidth}{\footnotesize \emph{Note.} \textbf{Bold} marks the best entry in each row: the fastest certified time where any configuration certifies, and otherwise the smallest remaining gap, since at a timeout every configuration reports the same elapsed time. Ties are all marked.  Obj is each
configuration's own best incumbent; ``---'' means solved to optimality;
``n/a'' means no single MIP bound is available; blank cells were not run
($n \ge 22$); $^\ddagger$ marks incumbents that are not stability-verified
(see text). Heur is the heuristic runtime (s) and HGap its gap (\%) to
the best known SCS-feasible solution; $|CS|$ and avg.\ $|c|$ describe the
lazy\_MIP\_ws incumbent. Time limits as in Table~\ref{tab:OCS_formulation}.}
\end{table}
\egroup

\subsection{Solution strategies for OSCS}
\label{ss:res_tractable}

Table~\ref{tab:SOCS_methods} compares the four OSCS configurations under AGG;
the e-companion gives DIS.

\paragraph{Separation strategies.}
All four methods return identical optima wherever they certify, which is the
basic correctness check on the separation approaches; at timeout the incumbents
differ and each column reports its own. The choice between enumeration and MIP
separation matters only when coalitions grow past the threshold. At
$\alpha = 0.5$ the formed structures are fragmented, mostly pairs and triples,
so $\tau = 9$ absorbs nearly all separation work and lazy and lazy\_MIP are
indistinguishable. At $\alpha = 1.0$ coalitions are larger and the separation
MIP earns its cost: at $n = 30$ it closes the gap to $0.11\%$ against $0.70\%$
for first-violation enumeration. As anticipated in Section~\ref{sec:separation}, the separation MIP's value
lies in its certificate rather than its cut.

\paragraph{The warm start.}
The heuristic is the single most effective component; per-instance
diagnostics are in Table~\ref{tab:heur_diag} of the e-companion. It is fastest or
tied-fastest in 11 of the 15 cells that solve to optimality, and at $n = 16$,
$\alpha = 1.0$ it is the only configuration that certifies optimality at all.
From $n = 18$ onward nothing certifies, but it holds the smallest gap in 12 of
the 14 timeout cells. At $n = 16$, $\alpha = 0.5$ it finishes in 74 seconds against 167, 162 and 337
for the three cold-started configurations. It is cheap where it needs to be,
taking at most $4.5$ seconds at $\alpha = 0.5$ and up to $106$ seconds at
$\alpha = 1.0$ where the Core checks run on large pooled coalitions, and its
objective is within $0.15\%$ of the best known stable solution. Where it
does not pay, the cold solve was already short (the only material case is
$n = 14$, $\alpha = 1.0$; Table~\ref{tab:SOCS_methods}).

\paragraph{The iterative baseline.}
The zero cut counts for cutplane at $n \in \{18, 20\}$ record that its first
re-solve never completed inside the time limit; at $n = 16$, $\alpha = 1.0$ many
rounds completed, but the loop never reached a violation-free incumbent. In all
five cases the reported incumbent comes from an interrupted loop, is not verified
against every stability constraint, and carries no MIP bound, which is why those
entries are marked $^\ddagger$. This is the cost of discarding
the branch-and-bound tree every round, and it is why Experiment~2 drops cutplane.

\paragraph{Lifted versus basic cuts.}
Running lazy\_MIP with lifted against basic cuts on all Experiment~1 instances,
the two certify identical optima on the 15 solved cells, with median runtime
ratio $0.99$ and no systematic winner on the five double timeouts
(Table~\ref{tab:ablation_lifted}). As Remark~\ref{rem:lifted_validity}
predicts, neither family moves the root bound,
so any advantage is path-dependent. The fallback separates the two regimes
structurally, never triggering at $\alpha = 1.0$, where
Proposition~\ref{prop:superadd_natural} forces $\Delta(s) \ge 0$, and firing on
$62.8\%$ of added cuts at $\alpha = 0.5$.

\subsection{Effect of the coalition policy}
\label{ss:res_solutions}

Section~\ref{ss:V_structural} argued that the coalition policy, not the CIPG
framework, decides whether the partition is a decision worth making. The solutions show
this, and more: the policy also decides whether fairness and stability
collide, because it controls how large coalitions get.

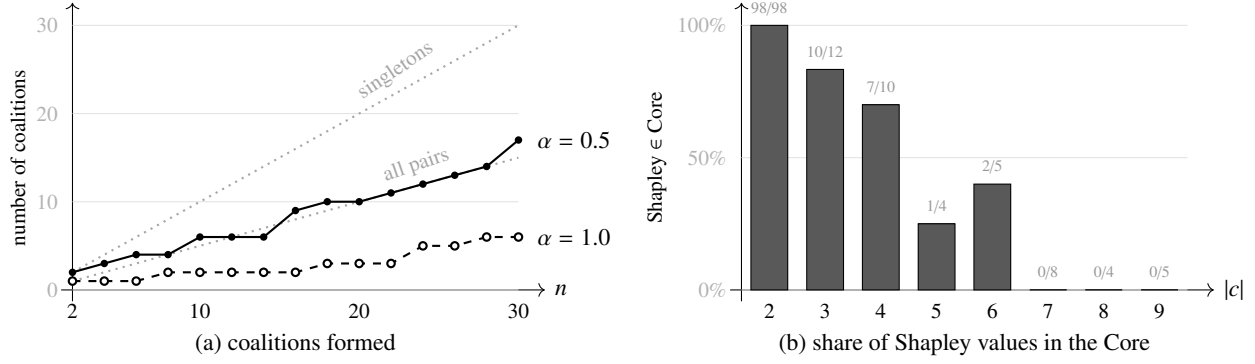
\begin{figure}[tbp]
\centering
\begin{tikzpicture}[font=\footnotesize]
\begin{scope}
\draw[->] (-0.15,0) -- (6.25,0) node[right] {$n$};
\draw[->] (0,-0.15) -- (0,3.80);
\draw[gray!22] (0,0.000) -- (5.90,0.000);
\node[left,gray!60] at (-0.06,0.000) {\scriptsize 0};
\draw[gray!22] (0,1.167) -- (5.90,1.167);
\node[left,gray!60] at (-0.06,1.167) {\scriptsize 10};
\draw[gray!22] (0,2.333) -- (5.90,2.333);
\node[left,gray!60] at (-0.06,2.333) {\scriptsize 20};
\draw[gray!22] (0,3.500) -- (5.90,3.500);
\node[left,gray!60] at (-0.06,3.500) {\scriptsize 30};
\node[below] at (0.000,-0.05) {\scriptsize 2};
\node[below] at (1.686,-0.05) {\scriptsize 10};
\node[below] at (3.793,-0.05) {\scriptsize 20};
\node[below] at (5.900,-0.05) {\scriptsize 30};
\node[rotate=90] at (-0.72,1.75) {\scriptsize number of coalitions};
\draw[gray!70,dotted,thick] (0.000,0.233) -- (5.900,3.500);
\node[gray!70,rotate=34,anchor=south] at (4.425,2.683) {\scriptsize singletons};
\draw[gray!70,dotted,thick] (0.000,0.117) -- (5.900,1.750);
\node[gray!70,rotate=19,anchor=south] at (4.636,1.400) {\scriptsize all pairs};
\draw[thick] plot coordinates {(0.000,0.233) (0.421,0.350) (0.843,0.467) (1.264,0.467) (1.686,0.700) (2.107,0.700) (2.529,0.700) (2.950,1.050) (3.371,1.167) (3.793,1.167) (4.214,1.283) (4.636,1.400) (5.057,1.517) (5.479,1.633) (5.900,1.983)};
\fill (0.000,0.233) circle (1.4pt);
\fill (0.421,0.350) circle (1.4pt);
\fill (0.843,0.467) circle (1.4pt);
\fill (1.264,0.467) circle (1.4pt);
\fill (1.686,0.700) circle (1.4pt);
\fill (2.107,0.700) circle (1.4pt);
\fill (2.529,0.700) circle (1.4pt);
\fill (2.950,1.050) circle (1.4pt);
\fill (3.371,1.167) circle (1.4pt);
\fill (3.793,1.167) circle (1.4pt);
\fill (4.214,1.283) circle (1.4pt);
\fill (4.636,1.400) circle (1.4pt);
\fill (5.057,1.517) circle (1.4pt);
\fill (5.479,1.633) circle (1.4pt);
\fill (5.900,1.983) circle (1.4pt);
\draw[thick,dashed] plot coordinates {(0.000,0.117) (0.421,0.117) (0.843,0.117) (1.264,0.233) (1.686,0.233) (2.107,0.233) (2.529,0.233) (2.950,0.233) (3.371,0.350) (3.793,0.350) (4.214,0.350) (4.636,0.583) (5.057,0.583) (5.479,0.700) (5.900,0.700)};
\draw[fill=white,thick] (0.000,0.117) circle (1.4pt);
\draw[fill=white,thick] (0.421,0.117) circle (1.4pt);
\draw[fill=white,thick] (0.843,0.117) circle (1.4pt);
\draw[fill=white,thick] (1.264,0.233) circle (1.4pt);
\draw[fill=white,thick] (1.686,0.233) circle (1.4pt);
\draw[fill=white,thick] (2.107,0.233) circle (1.4pt);
\draw[fill=white,thick] (2.529,0.233) circle (1.4pt);
\draw[fill=white,thick] (2.950,0.233) circle (1.4pt);
\draw[fill=white,thick] (3.371,0.350) circle (1.4pt);
\draw[fill=white,thick] (3.793,0.350) circle (1.4pt);
\draw[fill=white,thick] (4.214,0.350) circle (1.4pt);
\draw[fill=white,thick] (4.636,0.583) circle (1.4pt);
\draw[fill=white,thick] (5.057,0.583) circle (1.4pt);
\draw[fill=white,thick] (5.479,0.700) circle (1.4pt);
\draw[fill=white,thick] (5.900,0.700) circle (1.4pt);
\node[anchor=west] at (6.02,1.98) {$\alpha=0.5$};
\node[anchor=west] at (6.02,0.70) {$\alpha=1.0$};
\node at (2.95,-0.72) {(a) coalitions formed};
\end{scope}
\begin{scope}[xshift=8.85cm]
\draw[->] (-0.15,0) -- (6.25,0) node[right] {$|c|$};
\draw[->] (0,-0.15) -- (0,3.80);
\draw[gray!22] (0,0.000) -- (5.90,0.000);
\node[left,gray!60] at (-0.06,0.000) {\scriptsize 0\%};
\draw[gray!22] (0,1.750) -- (5.90,1.750);
\node[left,gray!60] at (-0.06,1.750) {\scriptsize 50\%};
\draw[gray!22] (0,3.500) -- (5.90,3.500);
\node[left,gray!60] at (-0.06,3.500) {\scriptsize 100\%};
\fill[black!65] (0.125,0) rectangle (0.612,3.500);
\draw (0.125,0) rectangle (0.612,3.500);
\node[below] at (0.369,-0.05) {\scriptsize 2};
\node[above,gray!80] at (0.369,3.500) {\tiny 98/98};
\fill[black!65] (0.863,0) rectangle (1.350,2.917);
\draw (0.863,0) rectangle (1.350,2.917);
\node[below] at (1.106,-0.05) {\scriptsize 3};
\node[above,gray!80] at (1.106,2.917) {\tiny 10/12};
\fill[black!65] (1.600,0) rectangle (2.087,2.450);
\draw (1.600,0) rectangle (2.087,2.450);
\node[below] at (1.844,-0.05) {\scriptsize 4};
\node[above,gray!80] at (1.844,2.450) {\tiny 7/10};
\fill[black!65] (2.338,0) rectangle (2.825,0.875);
\draw (2.338,0) rectangle (2.825,0.875);
\node[below] at (2.581,-0.05) {\scriptsize 5};
\node[above,gray!80] at (2.581,0.875) {\tiny 1/4};
\fill[black!65] (3.075,0) rectangle (3.562,1.400);
\draw (3.075,0) rectangle (3.562,1.400);
\node[below] at (3.319,-0.05) {\scriptsize 6};
\node[above,gray!80] at (3.319,1.400) {\tiny 2/5};
\fill[black!65] (3.813,0) rectangle (4.300,0.000);
\draw (3.813,0) rectangle (4.300,0.000);
\node[below] at (4.056,-0.05) {\scriptsize 7};
\node[above,gray!80] at (4.056,0.000) {\tiny 0/8};
\fill[black!65] (4.550,0) rectangle (5.037,0.000);
\draw (4.550,0) rectangle (5.037,0.000);
\node[below] at (4.794,-0.05) {\scriptsize 8};
\node[above,gray!80] at (4.794,0.000) {\tiny 0/4};
\fill[black!65] (5.288,0) rectangle (5.775,0.000);
\draw (5.288,0) rectangle (5.775,0.000);
\node[below] at (5.531,-0.05) {\scriptsize 9};
\node[above,gray!80] at (5.531,0.000) {\tiny 0/5};
\node[rotate=90] at (-1.12,1.75) {\scriptsize Shapley $\in$ Core};
\node at (2.95,-0.72) {(b) share of Shapley values in the Core};
\end{scope}
\end{tikzpicture}
\caption{The coalition policy sets the size of the formed coalitions, and size
decides whether a fair allocation is also stable. (a) Coalitions in the
lazy\_MIP\_ws incumbent (Table~\ref{tab:SOCS_methods}), against the two
trivial answers.
Restricted pooling gives slightly more coalitions than a perfect pairing, and
full pooling stays far above the grand coalition from $n = 8$ on; beyond the
smallest instances, neither regime returns a trivial structure. (b) Share of the $146$ formed coalitions of
each size whose Shapley value lies in the Core, with the count above each bar (Table~\ref{tab:refinement_summary}).}
\label{fig:solutions}
\end{figure}

\paragraph{Structures.}
Figure~\ref{fig:solutions}(a) places the computed structures against the two a
reader might expect, all singletons and the grand coalition, and neither regime
returns either. Under restricted pooling the
count runs just above the all-pairs line, so the answers are close to a matching
with a few players left alone; what is being decided is which players to pair,
not how many coalitions to form. Under full pooling the count sits far from the
grand coalition from $n = 8$ onward. Superadditivity makes the grand
coalition OCS-optimal there, so each nontrivial OSCS partition either ties its
welfare or is forced by an empty Core; the only certified loss is at $n = 8$. Per-instance detail is in
Table~\ref{tab:coalition_structures}.

Stability is free at $\alpha = 0.5$: OCS and OSCS coincide on every instance
solved to optimality, that is for $n \le 16$. Beyond $n = 16$ neither run is certified, so the incumbents are not
comparable, and in three cells the OSCS incumbent is the larger. The agreement
has a structural reason: a pair is worth forming exactly when
$V(\{i,j\}) \ge V(\{i\}) + V(\{j\})$, and in that case a stable two-way split
always exists, so every OCS pair is stabilizable. At $\alpha = 1.0$ the splits are certified
through $n = 16$, and only once does one cost anything: at $n = 8$ the OSCS gives
up three units, $5{,}609 \to 5{,}606$. That near-zero cost is a property of
these instances, not of CIPGs. With three players, $V(\{i\}) = 0$, every pair
worth $10$ and $V(N) = 12$, the pair inequalities force a total payoff of at
least $15 > 12$, so the grand coalition's Core is empty and stability costs
nearly $17\%$.

\paragraph{Payoffs.}
The same regimes decide whether a fair allocation is also stable. Refining every
OSCS solution as in Section~\ref{ss:coalition_refinement} treats 171
coalitions, 146 of size at least two (Table~\ref{tab:refinement_summary}). The nucleolus is Core-feasible in all 146, as the SCS guarantee
requires. The Shapley value is not: it lies in the Core in 118 of the 146, and membership
declines with size, reaching zero for the observed coalitions of
sizes $7$--$9$ (Figure~\ref{fig:solutions}(b); the dip at $|c| = 5$ rests on
four coalitions). The contrast across regimes is sharp, 100 of 102 at
$\alpha = 0.5$ against 18 of 44 at $\alpha = 1.0$, the empirical
counterpart of Example~\ref{ex:nonconvex}, since the Shapley value is
guaranteed to lie in the Core only for convex games. Every violation is under $0.232\%$ of $V(c)$, so the Shapley value is
nearly stable even when it fails, but a manager who wants fairness and
stability under full pooling should take the nucleolus.

\subsection{Cost of the framework}
\label{ss:res_cost}

\paragraph{Against exhaustive enumeration.}
As an implementation-independent check we ran a baseline that computes all
$2^n - 1$ coalition values, finds the OCS by dynamic programming over
partitions, and certifies the OSCS with Core-nonemptiness LPs. Wherever it
completes, its optimal values and partitions agree exactly with
Tables~\ref{tab:OCS_formulation} and~\ref{tab:SOCS_methods}
(Table~\ref{tab:bruteforce} in the e-companion).

Through $n \approx 14$ to $16$ the two routes are comparable. The difference is
entirely in the scaling: value enumeration alone is $\Theta(2^n)$, and the OSCS
search stage reaches more than 25 years at $n = 30$. Our
pipeline evaluates 766 of the $2^{30} - 1 \approx 1.07 \times 10^9$ coalition
values at $n = 30$, $\alpha = 0.5$, and returns certified sub-1\% solutions
inside the hour. At small $n$ under full pooling the Core checks touch every
subset anyway, so the saving is asymptotic.

Brute force is also all or nothing. Interrupted, it
yields neither a usable incumbent nor a bound, whereas every incumbent the lazy
framework accepts carries a certified gap, and every lazy\_MIP\_ws incumbent
was verified SCS-feasible, so the pipeline degrades gracefully with $n$.

\subsection{Implementation recommendations}
\label{ss:res_recommendations}

Six choices arise. Use the \emph{aggregated} model
\eqref{eq:OCS} for an OSCS and the \emph{disaggregated} model \eqref{eq:disagg}
for an OCS; without a warm start DIS produced poor incumbents at scale on our
instances. Add
stability as \emph{lazy constraints} in one branch-and-bound tree, not by
re-solving. Separate by \emph{enumeration} up to $\tau = 9$ and by the
\emph{separation MIP} \eqref{eq:sep_single_level} above it, which alone returns a
certificate. Prefer the \emph{lifted cut} \eqref{eq:lifted_cut}, falling back to
\eqref{eq:stability_bigM} when $\Delta(s) < 0$: no faster, but dominant and free.
\emph{Warm start} with Algorithm~\ref{alg:heuristic}, the most effective
choice on every large instance. Allocate with the \emph{nucleolus}.

\paragraph{Limitations.}
We claim two things only within limits. First, OSCS optimality is certified
only to $n = 16$, and at $\alpha = 1.0$ only by the warm-started
configuration, right at the limit (the others reach $n = 14$). Beyond that,
quality rests on the
reported gaps, below $1\%$ for lazy\_MIP\_ws under AGG and $1.1\%$ for the other
AGG configurations, and closing them at scale calls for decomposition or
branch-and-price. Second, we use one instance per $(n,\alpha)$ pair to span a
wide range of sizes and both regimes; a randomized study across seeds and
$(|R|, q, \rho)$ would strengthen the conclusions.

\section{Conclusion} \label{s:conclusion}

What makes a cooperative integer programming game hard is the coalition
policy laid over it, not the integrality of the coalition programs. Under full pooling the grand
coalition is optimal in advance and the computational problem is
certification. Under a restricted bound the partition becomes a decision
worth making, and difficulties follow: superadditivity can be lost,
stability needs separation, and on our instances DIS needed a warm start to
produce good incumbents.

Two directions stand out. Since payoff-space cuts leave the root bound
unchanged, tightening it needs valid inequalities in the assignment and item
variables, such as cover inequalities for~\eqref{eq:OCS_cap}. SCS guards
against fission, not fusion; merge-proofness atop Core stability is open
\citep{yang2025weak}.


%

%
%



\bigskip
\noindent\textbf{Code and data availability.} All code, instance generators, and
result files used in this paper are available at
\url{https://github.com/HyunwooLee0429/CIPG/}.

\bigskip
\noindent\textbf{Acknowledgments.} We thank Jian Yang for his constructive
feedback and for a rigorous review of the related literature during the
preparation of the conference version of this work.

\SingleSpacedXI
\bibliography{references} 

\begin{thebibliography}{39}
\providecommand{\natexlab}[1]{#1}
\providecommand{\url}[1]{\texttt{#1}}
\expandafter\ifx\csname urlstyle\endcsname\relax
  \providecommand{\doi}[1]{doi: #1}\else
  \providecommand{\doi}{doi: \begingroup \urlstyle{rm}\Url}\fi

\bibitem[Apt and Radzik(2006)]{apt2006stable}
Krzysztof~R Apt and Tadeusz Radzik.
\newblock Stable partitions in coalitional games.
\newblock \emph{Preprint,}, 2006.
\newblock \url{https://arxiv.org/abs/cs/0605132}.

\bibitem[Apt and Witzel(2009)]{apt2009generic}
Krzysztof~R Apt and Andreas Witzel.
\newblock A generic approach to coalition formation.
\newblock \emph{Internat. Game Theory Rev.}, 11\penalty0 (03):\penalty0
  347--367, 2009.

\bibitem[Arribillaga and Berganti{\~n}os(2022)]{arribillaga2022cooperative}
R~Pablo Arribillaga and Gustavo Berganti{\~n}os.
\newblock Cooperative and axiomatic approaches to the knapsack allocation
  problem.
\newblock \emph{Ann. Oper. Res.}, 318\penalty0 (2):\penalty0 805--830, 2022.

\bibitem[Aumann and Dr{\`e}ze(1974)]{aumann1974cooperative}
Robert~J Aumann and Jacques~H Dr{\`e}ze.
\newblock Cooperative games with coalition structures.
\newblock \emph{Internat. J. Game Theory}, 3\penalty0 (4):\penalty0 217--237,
  1974.

\bibitem[Borm et~al.(2001)Borm, Hamers, and Hendrickx]{borm2001operations}
Peter Borm, Herbert Hamers, and Ruud Hendrickx.
\newblock Operations research games: A survey.
\newblock \emph{Top}, 9\penalty0 (2):\penalty0 139--199, 2001.

\bibitem[Caprara and Letchford(2010)]{caprara2010cost}
Alberto Caprara and Adam~N Letchford.
\newblock New techniques for cost sharing in combinatorial optimization games.
\newblock \emph{Math. Programming}, 124\penalty0 (1--2):\penalty0 93--118,
  2010.

\bibitem[Carvalho et~al.(2023)Carvalho, Dragotto, Lodi, and
  Sankaranarayanan]{carvalho2023integer}
Margarida Carvalho, Gabriele Dragotto, Andrea Lodi, and Sriram
  Sankaranarayanan.
\newblock Integer programming games: A gentle computational overview.
\newblock In \emph{Tutorials in Oper. Res.}, pages 31--51. INFORMS, 2023.

\bibitem[Chalkiadakis et~al.(2011)Chalkiadakis, Elkind, and
  Wooldridge]{chalkiadakis2011computational}
Georgios Chalkiadakis, Edith Elkind, and Michael Wooldridge.
\newblock \emph{Computational aspects of cooperative game theory}.
\newblock Morgan \& Claypool Publishers, 2011.

\bibitem[Chardaire(2001)]{chardaire2001core}
Pierre Chardaire.
\newblock The core and nucleolus of games: A note on a paper by
  {G}{\"o}the-{L}undgren et al.
\newblock \emph{Math. Programming}, 90\penalty0 (1):\penalty0 147--151, 2001.

\bibitem[Crowder et~al.(1983)Crowder, Johnson, and Padberg]{crowder1983solving}
Harlan Crowder, Ellis~L Johnson, and Manfred Padberg.
\newblock Solving large-scale zero-one linear programming problems.
\newblock \emph{Oper. Res.}, 31\penalty0 (5):\penalty0 803--834, 1983.

\bibitem[Darmann and Klamler(2014)]{DarmannKlamler2014}
Andreas Darmann and Christian Klamler.
\newblock Knapsack cost sharing.
\newblock \emph{Rev. Econom. Design}, 18\penalty0 (3):\penalty0 219--241, 2014.

\bibitem[Deng et~al.(1999)Deng, Ibaraki, and Nagamochi]{deng1999algorithmic}
Xiaotie Deng, Toshihide Ibaraki, and Hiroshi Nagamochi.
\newblock Algorithmic aspects of the core of combinatorial optimization games.
\newblock \emph{Math. Oper. Res.}, 24\penalty0 (3):\penalty0 751--766, 1999.

\bibitem[Dragotto and Scatamacchia(2023)]{dragotto2023zero}
Gabriele Dragotto and Rosario Scatamacchia.
\newblock The zero regrets algorithm: Optimizing over pure {N}ash equilibria
  via integer programming.
\newblock \emph{INFORMS J. Comput.}, 35\penalty0 (5):\penalty0 1143--1160,
  2023.

\bibitem[Dror(1990)]{Dror1990}
Moshe Dror.
\newblock Cost allocation: The traveling salesman, binpacking, and the
  knapsack.
\newblock \emph{Appl. Math. Comput.}, 35\penalty0 (2):\penalty0 191--207, 1990.

\bibitem[Eryganov et~al.(2020)Eryganov, {\v{S}}ompl{\'a}k, Nevrl{\`y},
  Smejkalov{\'a}, Hrabec, and Haugen]{eryganov2020application}
Ivan Eryganov, Radovan {\v{S}}ompl{\'a}k, Vlastim{\'\i}r Nevrl{\`y}, Veronika
  Smejkalov{\'a}, Du{\v{s}}an Hrabec, and Kjetil~K Haugen.
\newblock Application of cooperative game theory in waste management.
\newblock \emph{Chemical Engrg. Trans.}, 81:\penalty0 877--882, 2020.

\bibitem[Frangioni and Gentile(2006)]{frangioni2006perspective}
Antonio Frangioni and Claudio Gentile.
\newblock Perspective cuts for a class of convex 0--1 mixed integer programs.
\newblock \emph{Math. Programming}, 106\penalty0 (2):\penalty0 225--236, 2006.

\bibitem[Gansterer and Hartl(2018)]{gansterer2018collaborative}
Margaretha Gansterer and Richard~F Hartl.
\newblock Collaborative vehicle routing: A survey.
\newblock \emph{Eur. J. Oper. Res.}, 268\penalty0 (1):\penalty0 1--12, 2018.

\bibitem[G{\"o}the-Lundgren et~al.(1996)G{\"o}the-Lundgren, J{\"o}rnsten, and
  V{\"a}rbrand]{gothe1996nucleolus}
Maud G{\"o}the-Lundgren, Kurt J{\"o}rnsten, and Peter V{\"a}rbrand.
\newblock On the nucleolus of the basic vehicle routing game.
\newblock \emph{Math. Programming}, 72\penalty0 (1):\penalty0 83--100, 1996.

\bibitem[Granot(1986)]{granot1986generalized}
Daniel Granot.
\newblock A generalized linear production model: A unifying model.
\newblock \emph{Math. Programming}, 34\penalty0 (2):\penalty0 212--222, 1986.

\bibitem[Gu et~al.(1998)Gu, Nemhauser, and Savelsbergh]{gu1998lifted}
Zonghao Gu, George~L Nemhauser, and Martin~WP Savelsbergh.
\newblock Lifted cover inequalities for 0-1 integer programs: Computation.
\newblock \emph{INFORMS J. Comput.}, 10\penalty0 (4):\penalty0 427--437, 1998.

\bibitem[Guajardo and R{\"o}nnqvist(2015)]{guajardo2015coalition}
Mario Guajardo and Mikael R{\"o}nnqvist.
\newblock Operations research models for coalition structure in collaborative
  logistics.
\newblock \emph{Eur. J. Oper. Res.}, 240\penalty0 (1):\penalty0 147--159, 2015.

\bibitem[Guajardo and R{\"o}nnqvist(2016)]{guajardo2016review}
Mario Guajardo and Mikael R{\"o}nnqvist.
\newblock A review on cost allocation methods in collaborative transportation.
\newblock \emph{Internat. Trans. Oper. Res.}, 23\penalty0 (3):\penalty0
  371--392, 2016.

\bibitem[G{\"u}nl{\"u}k and Linderoth(2010)]{gunluk2010perspective}
Oktay G{\"u}nl{\"u}k and Jeff Linderoth.
\newblock Perspective reformulations of mixed integer nonlinear programs with
  indicator variables.
\newblock \emph{Math. Programming}, 124\penalty0 (1--2):\penalty0 183--205,
  2010.

\bibitem[{Gurobi Optimization, LLC}(2025)]{gurobi}
{Gurobi Optimization, LLC}.
\newblock Gurobi optimizer reference manual, 2025.
\newblock URL \url{https://www.gurobi.com}.

\bibitem[K{\'o}czy(2018)]{koczy2018partition}
L{\'a}szl{\'o}~{\'A}. K{\'o}czy.
\newblock \emph{Partition Function Form Games: Coalitional Games with
  Externalities}, volume~48 of \emph{Theory and Decision Library C}.
\newblock Springer, 2018.

\bibitem[Lee et~al.(2024)Lee, Hildebrand, Cai, and
  B{\"u}y{\"u}ktahtak{\i}n]{lee2024algorithms}
Hyunwoo Lee, Robert Hildebrand, Wenbo Cai, and {\.I}~Esra
  B{\"u}y{\"u}ktahtak{\i}n.
\newblock Random-restart best-response dynamics for large-scale integer
  programming games and their applications.
\newblock \emph{Preprint,}, 2024.
\newblock \url{https://arxiv.org/abs/2409.04078}.

\bibitem[Maschler et~al.(1979)Maschler, Peleg, and
  Shapley]{maschler1979geometric}
Michael Maschler, Bezalel Peleg, and Lloyd~S Shapley.
\newblock Geometric properties of the kernel, nucleolus, and related solution
  concepts.
\newblock \emph{Math. Oper. Res.}, 4\penalty0 (4):\penalty0 303--338, 1979.

\bibitem[Owen(1975)]{owen1975core}
Guillermo Owen.
\newblock On the core of linear production games.
\newblock \emph{Math. Programming}, 9\penalty0 (1):\penalty0 358--370, 1975.

\bibitem[{\"O}zener et~al.(2013){\"O}zener, Ergun, and
  Savelsbergh]{ozener2013allocating}
Okan~{\"O}rsan {\"O}zener, {\"O}zlem Ergun, and Martin Savelsbergh.
\newblock Allocating cost of service to customers in inventory routing.
\newblock \emph{Oper. Res.}, 61\penalty0 (1):\penalty0 112--125, 2013.

\bibitem[Peleg and Sudh{\"o}lter(2007)]{peleg2007introduction}
Bezalel Peleg and Peter Sudh{\"o}lter.
\newblock \emph{Introduction to the theory of cooperative games}.
\newblock Springer, 2007.

\bibitem[Rahwan et~al.(2009)Rahwan, Ramchurn, Jennings, and
  Giovannucci]{rahwan2009anytime}
Talal Rahwan, Sarvapali~D Ramchurn, Nicholas~R Jennings, and Andrea
  Giovannucci.
\newblock An anytime algorithm for optimal coalition structure generation.
\newblock \emph{J. Artificial Intelligence Res.}, 34:\penalty0 521--567, 2009.

\bibitem[Rahwan et~al.(2015)Rahwan, Michalak, Wooldridge, and
  Jennings]{rahwan2015coalition}
Talal Rahwan, Tomasz~P Michalak, Michael Wooldridge, and Nicholas~R Jennings.
\newblock Coalition structure generation: A survey.
\newblock \emph{Artificial Intelligence}, 229:\penalty0 139--174, 2015.

\bibitem[Sandholm et~al.(1999)Sandholm, Larson, Andersson, Shehory, and
  Tohm{\'e}]{sandholm1999coalition}
Tuomas Sandholm, Kate Larson, Martin Andersson, Onn Shehory, and Fernando
  Tohm{\'e}.
\newblock Coalition structure generation with worst case guarantees.
\newblock \emph{Artificial Intelligence}, 111\penalty0 (1-2):\penalty0
  209--238, 1999.

\bibitem[Schmeidler(1969)]{schmeidler1969nucleolus}
David Schmeidler.
\newblock The nucleolus of a characteristic function game.
\newblock \emph{SIAM J. Appl. Math.}, 17\penalty0 (6):\penalty0 1163--1170,
  1969.

\bibitem[Shapley(1953)]{shapley1953value}
Lloyd~S Shapley.
\newblock A value for $n$-person games.
\newblock In Harold~W Kuhn and Albert~W Tucker, editors, \emph{Contributions to
  the Theory of Games II}, pages 307--317. Princeton University Press, 1953.

\bibitem[Stamtsis and Erlich(2004)]{stamtsis2004use}
Georgios~C Stamtsis and Istv{\'a}n Erlich.
\newblock Use of cooperative game theory in power system fixed-cost allocation.
\newblock \emph{IEE Proc.-Generation, Transmission Distribution}, 151\penalty0
  (3):\penalty0 401--406, 2004.

\bibitem[Wang et~al.(2003)Wang, Fang, and Hipel]{wang2003water}
LZ~Wang, Liping Fang, and Keith~W Hipel.
\newblock Water resources allocation: a cooperative game theoretic approach.
\newblock \emph{J. Environ. Informatics}, 2\penalty0 (2):\penalty0 11--22,
  2003.

\bibitem[Wang et~al.(2024)Wang, Zeng, Cheng, Liu, Huang, Liu, He, Yao, Wang,
  and Li]{wang2024cooperative}
Qin Wang, Jincan Zeng, Beibei Cheng, Minwei Liu, Guori Huang, Xi~Liu, Gengsheng
  He, Shangheng Yao, Peng Wang, and Longxi Li.
\newblock A cooperative game approach for optimal design of shared energy
  storage system.
\newblock \emph{Sustainability}, 16\penalty0 (17):\penalty0 7255, 2024.

\bibitem[Yang(2025)]{yang2025weak}
Jian Yang.
\newblock The weak core, partition-based universal stability, and their risk
  associations through a partial order.
\newblock \emph{Naval Res. Logist.}, 2025.
\newblock Forthcoming.

\end{thebibliography}

\RUNAUTHOR{Lee et al.}
\ECSwitch
\OneAndAHalfSpacedXII

\ECHead{E-Companion}

\section{Notation}\label{app:notation}

\begin{table}[h!]
  \caption{Notation for the OCS formulation.}
  \label{tab:OCS_notation}
  \centering
  \renewcommand{\arraystretch}{0.95}
  \begin{tabular}{@{}ll@{}}
    \toprule
    \multicolumn{2}{@{}l}{\textit{Sets and indices}} \\
    \midrule
    \(N\)                  & set of players; index \(i \in N\) \\
    \(G := N\)             & set of potential group indices; at most \(|N|\)
                             groups form \\
    \(R\)                  & item set; index \(j \in R\) \\
    \(I_j \subseteq N\)    & players for whom item \(j\) is available \\
    \addlinespace
    \multicolumn{2}{@{}l}{\textit{Parameters}} \\
    \midrule
    \(p^i_j\)              & profit of item \(j\) for player \(i\) \\
    \(a_j\)                & weight of item \(j\), common across players \\
    \(b^i\)                & budget of player \(i\) \\
    \(u^i_j\)              & upper bound on usage of item \(j\) by player \(i\) \\
    \(B^i, d^i\)           & data of the private constraints
                             \(B^i x^i \le d^i\) \\
    \(\alpha \in (0,1]\)   & pooling parameter; the fraction of the
                             full-pooling bound a coalition may use on a common item \\
    \addlinespace
    \multicolumn{2}{@{}l}{\textit{Decision variables}} \\
    \midrule
    \(z_{ig} \in \{0,1\}\) & 1 if player \(i\) is assigned to group \(g\) \\
    \(x^{i,g}_j \ge 0\)    & usage of item \(j\) by player \(i\) in group \(g\) \\
    \(y^g_j \in \mathbb{Z}_+\) & aggregate usage of item \(j\) in group \(g\) \\
    \(w^g_j \in \{0,1\}\)  & 1 if item \(j\) is common in group \(g\), that is,
                             available to \(\ge 2\) assigned players \\
    \(\pi_g \ge 0\)        & total profit generated by group \(g\) \\
    \(v_i \ge 0\)          & payoff allocated to player \(i\) \\
    \(\phi_{ig} \ge 0\)    & linearization auxiliary,
                             \(\phi_{ig} = v_i z_{ig}\) (aggregated model) \\
    \(v_{ig} \ge 0\)       & payoff player \(i\) receives from group \(g\)
                             (disaggregated model) \\
    \bottomrule
  \end{tabular}
\end{table}

\section{Omitted Proofs and Examples}\label{app:proofs}

\begin{example}[Superadditive but not convex]
\label{ex:nonconvex}
Take $N = \{1,2,3\}$ with one item of weight $a = 4$ and profit $p^i = 4$,
available to all three players with $u^i = 1$ each, and budgets $b^i = 2$. Under full
pooling a coalition $c$ has budget $2|c|$ and bound $y^c \le |c|$, so
\[
  V(\{i\}) = 0, \qquad V(\{i,j\}) = 4, \qquad V(N) = 4,
\]
since two units would cost $8 > 6$. The game is superadditive, as
Proposition~\ref{prop:superadd_natural} requires. It is not convex: taking
$c = \{1,2\}$ and $d = \{1,3\}$ gives
\[
  V(c) + V(d) \;=\; 8 \;>\; 4 \;=\; V(c \cup d) + V(c \cap d).
\]
The second member of a coalition is worth $4$ and the third is worth nothing.
\end{example}

\begin{example}[A budget-only player must be paid beyond its own profit potential]
\label{ex:budget_only}
Consider three players with budgets $b^1 = b^2 = 1$ and
$b^3 = 10$, an item of profit $8$ and weight $10$ available only to player
1, an identical item available only to player 2, and a common item of
profit $2$ and weight $2$ available to players 1 and 2. Under full
pooling, $V(\{1\}) = V(\{2\}) = V(\{3\}) = 0$, $V(\{1,2\}) = 2$,
$V(\{1,3\}) = V(\{2,3\}) = 8$, and $V(N) = 10$. Summing the Core
inequalities for the two mixed pairs gives
$v_1 + v_2 + 2 v_3 \ge 16$, and with $v_1 + v_2 + v_3 = 10$ this forces
$v_3 \ge 6$; e.g., $(2, 2, 6)$ is a Core allocation. Since player 3 has no
items, $U_3 = 0$, so the cap $v_3 \le U_3$ would cut off \emph{every} Core
allocation of the grand coalition.
\end{example}

\begin{proof}{Proof of Proposition~\ref{prop:disagg_LP}.}
\emph{(Same mixed-integer feasible points.)}
Fix an integer assignment $z$ and let $g^*(i)$ be the unique group with
$z_{i,g^*(i)} = 1$. In the aggregated model, McCormick at integer $z$ gives
$\phi_{i,g^*(i)} = v_i$ and $\phi_{ig} = 0$ for $g \neq g^*(i)$, so
efficiency becomes $\pi_g = \sum_{i:\, g^*(i)=g} v_i$. In the disaggregated
model, the upper bound in~\eqref{eq:disagg} forces $v_{ig} = 0$ for
$g \neq g^*(i)$, and linking gives
$v_{i,g^*(i)} = v_i$, so efficiency yields the same
expression. Hence the two formulations have the same mixed-integer feasible
points.

\emph{($F_{\mathrm{disagg}} \subseteq F_{\mathrm{agg}}$.)}
We show every aggregated constraint is satisfied by points of
$F_{\mathrm{disagg}}$.
(i)~The upper bound $v_{ig} \le M z_{ig}$ is exactly the McCormick
constraint $\phi_{ig} \le M z_{ig}$.
(ii)~The efficiency $\pi_g = \sum_i v_{ig}$ is $\pi_g = \sum_i \phi_{ig}$,
which is the aggregated efficiency constraint.
(iii)~The remaining two McCormick constraints are implied by the linking
constraint $v_i = \sum_g v_{ig}$: first,
$\phi_{ig} = v_{ig} \le \sum_{g'} v_{ig'} = v_i$ since $v_{ig'} \ge 0$;
second,
$v_{ig} = v_i - \sum_{g' \neq g} v_{ig'} \ge v_i - \sum_{g' \neq g} M z_{ig'}
= v_i - M(1 - z_{ig})$,
using $v_{ig'} \le M z_{ig'}$ and the partition constraint
$\sum_{g'} z_{ig'} = 1$.
Since the linking constraint is \emph{not} present in $F_{\mathrm{agg}}$,
every point of $F_{\mathrm{disagg}}$ satisfies all constraints of
$F_{\mathrm{agg}}$, so $F_{\mathrm{disagg}} \subseteq F_{\mathrm{agg}}$.

(Strict inclusion):
For suitable instance data, consider $|N|=2$ with $M=100$, and take
$z_{11}=1$, $z_{21}=z_{22}=0.5$, $v_2=50$; this satisfies the partition and
symmetry-breaking constraints, and it suffices that the instance admits
fractional item usage generating group profits of $50$ in each group (two
items with sufficiently large profit and capacity suffice, since the
$x$-variables are continuous in the LP relaxation). In $F_{\mathrm{agg}}$:
$\phi_{21} \le \min(v_2,\, M z_{21}) = 50$,
$\phi_{21} \ge v_2 - M(1-z_{21}) = 0$, and similarly for $\phi_{22}$.
Setting $\phi_{21} = \phi_{22} = 50$ contributes $100 > 50 = v_2$ to the
group values: the LP relaxation double-counts player 2's payoff across
groups. In $F_{\mathrm{disagg}}$: $v_{21} + v_{22} = v_2 = 50$, so player
2's payoff is counted only once.

Finally, adding the linking constraint $v_i = \sum_g \phi_{ig}$ to
$F_{\mathrm{agg}}$ yields exactly the constraint set of
$F_{\mathrm{disagg}}$, which proves the last assertion. \Halmos
\end{proof}

\begin{proof}{Proof of Proposition~\ref{prop:stability_valid}.}
(i) Let \(k := |s| - \sum_{i\in s} z_{ig} \in \mathbb{Z}_{\ge 0}\). If
\(k = 0\), then \(s \subseteq c(g)\) and \eqref{eq:stability_bigM} is the
Core inequality \(\sum_{i\in s} v_i \ge V(s)\) for \(s\) within \(c(g)\),
which holds by SCS-feasibility. If \(k \ge 1\), then the right-hand side of
\eqref{eq:stability_bigM} equals \(V(s) - M_s k \le V(s) - M_s \le 0\),
while every SCS-encoding point satisfies \(v_i \ge V(\{i\}) \ge 0\) for
each player through singleton stability within the player's own coalition,
so \(\sum_{i\in s} v_i \ge 0\).
(ii) With \(s \subseteq c(g)\) we have \(k = 0\), so
\eqref{eq:stability_bigM} reads \(\sum_{i\in s} v_i \ge V(s)\), violated by
assumption. \Halmos
\end{proof}

\begin{proof}{Proof of Proposition~\ref{prop:lifted_valid}.}
Fix a mixed-integer point of $\mathcal{F}^{\mathrm{stab}}$ and a group $g$,
and let $k := \sum_{i\in s} z_{ig} \in \{0,\dots,|s|\}$. If $k = |s|$, then
$s \subseteq c(g)$ and the right-hand side of \eqref{eq:lifted_cut} equals
$\Delta(s) + \sum_{i\in s} V(\{i\}) = V(s)$, so \eqref{eq:lifted_cut} is
the Core inequality for $s$ within $c(g)$, which holds by SCS-feasibility.
If $k \le |s|-1$, then $k - |s| + 1 \le 0$ and $\Delta(s) \ge 0$ imply
\[
\Delta(s)\,(k - |s| + 1) + \sum_{i\in s} V(\{i\})\, z_{ig}
\;\le\; \sum_{i\in s} V(\{i\})\, z_{ig}
\;\le\; \sum_{i\in s} V(\{i\}),
\]
using $z_{ig} \le 1$ and $V(\{i\}) \ge 0$. Every player $i$ belongs to some
group $g'(i)$, and the singleton member of the stability
family~\eqref{eq:stability_bigM} for that group yields $v_i \ge V(\{i\})$.
Summing over $i \in s$ gives
$\sum_{i\in s} v_i \ge \sum_{i\in s} V(\{i\})$, which dominates the
right-hand side above. \Halmos
\end{proof}

\begin{proof}{Proof of Proposition~\ref{prop:lifted_dominance}.}
With $M_s = V(s)$ the right-hand side of \eqref{eq:stability_bigM} equals
$V(s)\bigl(\sum_{i\in s} z_{ig} - |s| + 1\bigr)$. Subtracting it from the
right-hand side of \eqref{eq:lifted_cut} and using
$V(s) = \Delta(s) + \sum_{i\in s}V(\{i\})$ gives
\[
\sum_{i\in s} V(\{i\}) \Bigl[\, z_{ig} - \Bigl(\textstyle\sum_{\ell \in s}
z_{\ell g} - |s| + 1\Bigr) \Bigr]
\;=\;
\sum_{i\in s} V(\{i\}) \Bigl[\, (|s|-1) - \textstyle\sum_{\ell \in s
\setminus \{i\}} z_{\ell g} \Bigr]
\;\ge\; 0,
\]
since each bracket is a sum of $|s|-1$ terms of the form
$1 - z_{\ell g} \ge 0$ and $V(\{i\}) \ge 0$. \Halmos
\end{proof}

\begin{proof}{Proof of Proposition~\ref{prop:heuristic_feasible}.}
\emph{Termination.} Each executed merge strictly decreases $|CS|$, so the loop
of Algorithm~\ref{alg:heuristic} runs at most $|N|-1$ times. Each Core check
terminates: for $|c| \le \tau$ it enumerates finitely many subcoalitions, and
for $|c| > \tau$ each round of row generation adds a subcoalition of $c$ not
generated before, of which there are finitely many. Both loops are additionally
capped by $T_1$, and Phase~2 by $T_2$.

\emph{The incumbent is an SCS at the end of initialization.} The singleton
partition with $\tilde v_i = V(\{i\})$ is an SCS by
Proposition~\ref{prop:ex_css}.

\emph{The invariant is maintained.} The incumbent is modified only in the last
step before the return, and only when a Core allocation has been obtained for
every coalition of $CS'$. Every $c \in \mathcal{P}$ other than a singleton
entered $\mathcal{P}$ only after \textsc{CoreNonempty}$(c)$ returned true, which
certifies that $(c,V_c)$ has a nonempty Core, and $CS' \subseteq \mathcal{P}$ is
a partition of $N$ because $\mathcal{P}$ contains the singletons, so
\eqref{eq:setpart} is feasible and its seeded solution is available even if
$T_2$ expires. The replacement therefore satisfies
$\sum_{i \in c}\tilde v_i = V(c)$ and
$\sum_{i \in s}\tilde v_i \ge V(s)$ for every $s \subseteq c$ and every
$c \in CS'$, which is the definition of an SCS. If the allocations are not all
obtained, the previous incumbent is returned unchanged.

\emph{Initialization is necessary.} An SCS requires
$\sum_{i \in s}\tilde v_i \ge V(s)$ for every subcoalition $s$ of every formed
coalition, and $s = \{i\}$ is such a subcoalition, so $V(\{i\})$ is needed for
every $i$ regardless of the partition returned.

No argument after initialization uses the size of $T_1$ or $T_2$: a smaller budget
yields a smaller pool, never an uncertified one, because an inconclusive check
is treated as a failure.

\emph{Cost.} Each executed merge reduces $|CS|$ by one, so the loop passes
through partitions of size $|N|, |N|-1, \dots, 2$ and considers at most
$\binom{k}{2}$ candidate pairs when $|CS| = k$. Summing,
$\sum_{k=2}^{|N|}\binom{k}{2} = \binom{|N|+1}{3}$, and each candidate receives at
most one Core check. A check on $|c| \le \tau$ enumerates the $2^{|c|}-2 \le
2^\tau - 2$ proper nonempty subcoalitions of $c$; above the threshold the
row-generation loop adds a distinct subcoalition each round and is bounded by
$T_1$. \Halmos
\end{proof}

\section{Extension to General MILP-Based CIPGs}\label{ss:general_milp}

Sections~\ref{s:cipg} and~\ref{s:formulation} were written for the resource
allocation game of Section~\ref{ss:player_coalition_IP}. With the model and its
cut families in place we can say exactly how much of that structure they use.

Two components use none of it. The stability inequalities
\eqref{eq:stability_bigM} and \eqref{eq:lifted_cut} reference the instance only
through the values $V(s)$ and $V(\{i\})$, so
Propositions~\ref{prop:stability_valid}, \ref{prop:lifted_valid} and
\ref{prop:lifted_dominance} and Theorem~\ref{thm:CSS_projection} hold verbatim
for any CIPG in the sense of Definition~\ref{def:cipg}. The formulation
comparison of Section~\ref{ss:disagg} involves only the assignment and payoff
variables, never the underlying player variables, so
Proposition~\ref{prop:disagg_LP} is unchanged.

The rest are written in the rows of that model and must be restated rather
than inherited; we go component by component. Throughout, consider CIPGs
in which each player solves a bounded MILP in the sense of
Definition~\ref{def:cipg}, with designated poolable (budget-type) rows and
private structural rows.

\begin{itemize}
\item \emph{OCS formulations.} Introduce group-indexed copies of each
player's variables and aggregate the right-hand sides of the poolable rows
across the members assigned to a group, exactly as in~\eqref{eq:OCS};
private rows are gated by the assignment variables as in~\eqref{eq:OCS}.
The McCormick and disaggregated linearizations, and hence
Proposition~\ref{prop:disagg_LP}, carry over unchanged, because they involve only the assignment and payoff
variables---never the underlying player variables.
\item \emph{Stability inequalities.} The basic cut~\eqref{eq:stability_bigM}
and the lifted cut~\eqref{eq:lifted_cut} reference the instance only
through the values $V(s)$ and $V(\{i\})$, so their validity
(Propositions~\ref{prop:stability_valid} and~\ref{prop:lifted_valid}) is
independent of how those values are generated.
\item \emph{Separation.} The single-level reformulation of the separation
MIP~\eqref{eq:sep_single_level} requires only that the subset-selection
variables enter the inner coalition problem through right-hand sides and
variable bounds, and that player variables are bounded---both hold for the
general model.
\item \emph{Heuristic and refinement.} The SCS-feasible heuristic and the
Shapley/nucleolus refinement access the instance solely through a $V(c)$
oracle and small LPs.
\end{itemize}

Consequently, instantiating the framework for another MILP class (e.g.,
facility-location or assignment-type games) changes only the inner
coalition-level MILPs, not the formulation, the cuts, or the
algorithms; the CKG of Section~\ref{s:comp_results} is one such
instantiation.

\begin{remark}[Relation to integer minimization games]
\label{rem:im_games}
Under full pooling the resource allocation game of
Section~\ref{ss:player_coalition_IP} fits the maximization mirror of the
(mixed-integer) \emph{integer minimization games} of \citet{caprara2010cost}:
the rows are fixed, and the subset enters only through right-hand sides and
variable bounds that are linear in its incidence vector, which is exactly the
condition under which the separation MIP \eqref{eq:sep_single_level} exists.
The proof of Proposition~\ref{prop:superadd_natural} mirrors that of their
Observation~1, the automatic subadditivity of the class. Restricted pooling
leaves the class: the common-or-individual designation makes the presence of
the $\alpha$-row depend on the subset itself, which is what the indicators
$w^g_j$ and $\kappa_j$ encode, and it is also what makes the coalition
structure a nontrivial decision. The two frameworks also answer an empty Core
differently: $\gamma$-budget-balanced cost shares relax efficiency for a
fixed grand coalition, while an SCS restores exact efficiency by refining the
partition; the two relaxations are orthogonal and could be combined.

One further transfer is useful in computation. The mirror of their
Observation~7 states that, within a formed coalition, the Core inequality for
a subset $s$ is redundant whenever $s$ splits into $s_1, s_2$ with
$V(s_1) + V(s_2) \ge V(s)$, since summing the inequalities for $s_1$ and
$s_2$ dominates it. The subset enumerations of Section~\ref{sec:separation}
and the Core checks of Algorithm~\ref{alg:heuristic} therefore need only the
\emph{indecomposable} subsets, and in the gated family
\eqref{eq:stability_bigM} the same reduction is valid at integer incumbents
in the presence of the singleton cuts: with the gate closed the summed cuts
for $s_1$ and $s_2$ dominate, and with it open the right-hand side is
nonpositive.
\end{remark}

\clearpage
\section{Complexity of the Stability Problems}\label{app:complexity}

A CIPG instance is the compact data of Section~\ref{ss:player_coalition_IP}; the
$2^{|N|}-1$ coalition values are not part of the input but are generated from
it, and that encoding is the source of the difficulty. When the values are
supplied explicitly, an optimal coalition structure follows by dynamic
programming over subsets in $O(3^{|N|})$ time
\citep{sandholm1999coalition,rahwan2015coalition}, polynomial in the size of
that exponential input; this is the baseline of Section~\ref{ss:res_cost}. We assume the data are rational and every
coalition program feasible and bounded.

\begin{proposition}[Evaluation and partition selection]
\label{prop:ocs-npc}
Deciding $V(c) \ge K$ is \textup{NP}-complete, and so is deciding whether some
coalition structure satisfies $\sum_{c \in CS(N)} V(c) \ge K$.
\end{proposition}

\begin{proof}{Proof.}
For $V(c) \ge K$, guess a plan for $c$ and verify its value; hardness is
\textsc{Knapsack} with $|c| = 1$. For the partition problem, guess a coalition
structure and one plan per block: a structure has at most $|N|$ blocks, and an
understated coalition value can only understate the objective, so no guessed
plan need be optimal. Hardness again follows with $|N| = 1$. \Halmos
\end{proof}

Membership is the informative half: one \emph{feasible} plan per block
certifies the bound, and no guess need be optimal. That property fails once stability is imposed, and it is why
Section~\ref{ss:OCS_formulation} can be a compact mixed-integer program while
Section~\ref{ss:stability_CSS} cannot.

\begin{proposition}[Separation]
\label{prop:sep-npc}
Deciding whether some $s \subseteq c$ satisfies
$\sum_{i \in s} \bar v_i < V(s)$ is \textup{NP}-complete.
\end{proposition}

\begin{proof}{Proof.}
Guess $s$ together with a plan for $s$ whose value exceeds
$\sum_{i \in s} \bar v_i$; a feasible plan suffices, since exhibiting one proves
$V(s) > \sum_{i \in s} \bar v_i$. Hardness is \textsc{Knapsack} with
$|c| = 1$. \Halmos
\end{proof}

Separation of the stability family is therefore not a subroutine a better
implementation would render polynomial, and its natural certificate---a subset
with a plan that beats the incumbent's payoff---is exactly the pair returned by
\eqref{eq:sep_single_level}.

\begin{proposition}[Verification and membership]
\label{prop:scs-check-dp}
Deciding whether a given pair $(CS(N), v)$ is an SCS lies in \textup{DP} and is
hard for both \textup{NP} and \textup{coNP}; it therefore lies in neither class
unless $\textup{NP} = \textup{coNP}$. Deciding whether a given $v$ admits some
SCS is likewise \textup{coNP}-hard.
\end{proposition}

\begin{proof}{Proof.}
\emph{Membership.} The pair is an SCS iff for every $c \in CS(N)$: (a)
$\sum_{i \in c} v_i = V(c)$, and (b) $\sum_{i \in s} v_i \ge V(s)$ for all
$s \subseteq c$. Condition (b), over at most $|N|$ blocks, is the complement of
the separation problem and lies in \textup{coNP} by
Proposition~\ref{prop:sep-npc}. Given (b), the case $s = c$ supplies
$\sum_{i \in c} v_i \ge V(c)$, so what remains of (a) is the residual assertion
$V(c) \ge \sum_{i \in c} v_i$, a polynomial conjunction of value queries and
hence in \textup{NP}. The problem is the intersection of a \textup{coNP} and an
\textup{NP} language, which is the class \textup{DP}.

\emph{Hardness for \textup{NP}.} Take $|N| = 1$ with no private constraints and
every item available, so \eqref{eq:coal_IP_general} is a $0$--$1$ knapsack, and
set $v_1 = K$. Then $(\{\{1\}\}, v)$ is an SCS iff $V(\{1\}) = K$; reduce from
\textsc{SubsetSum}, taking profits equal to weights and capacity equal to the
target.

\emph{Hardness for \textup{coNP}.} An instance must supply an efficient $v$,
which presupposes $V(c)$, so we build a gadget in which $V(c)$ is computable by
inspection while a proper subcoalition's value is not. Given a knapsack instance
with capacity $B$, integer profits and target $K$, take $N = \{1,2\}$ and
$c = \{1,2\}$ with every item available to player~$1$ only, so $I_j = \{1\}$ and
$R^c_{\mathrm{com}} = \emptyset$, and player~$2$ holding no items. Set
$b^1 = B$ and $b^2$ large enough that
$b^1 + b^2 \ge \sum_{j \in R} a_j u^1_j$. Then
$V(c) = P := \sum_{j \in R} p^1_j u^1_j$, $V(\{2\}) = 0$, and $V(\{1\})$ is the
knapsack optimum at capacity $B$. Assume $K \le P$ and put
$v = (K - \tfrac12,\ P - K + \tfrac12)$, so $v_1 + v_2 = P = V(c)$ and
efficiency holds by construction. Of the subsets of $c$: $s = \{2\}$ requires
$v_2 \ge 0$, which holds; $s = c$ holds with equality; and $s = \{1\}$ requires
$K - \tfrac12 \ge V(\{1\})$, that is $V(\{1\}) < K$, the profits being integers.
Hence $(\{c\}, v)$ is an SCS iff $V(\{1\}) < K$.

\emph{Hardness of membership.} In the same gadget the only coalition structures
are $\{\{1,2\}\}$ and $\{\{1\},\{2\}\}$; the second requires
$v_2 = V(\{2\}) = 0$ while $v_2 \ge \tfrac12$, so it is never an SCS. Therefore
$v$ admits an SCS iff $V(\{1\}) < K$. \Halmos
\end{proof}

Stability alone---condition (b)---is \textup{coNP}-complete by
Proposition~\ref{prop:sep-npc}. Efficiency alone is \emph{not} in \textup{NP}:
$\sum_{i \in c} v_i = V(c)$ is an exact-optimum assertion. Only the residual
assertion, once (b) has supplied one inequality, is an \textup{NP} question, and
it is the conjunction that leaves both classes. This is why
Section~\ref{sec:separation} verifies time-limit incumbents by re-solving $V(c)$
for every formed coalition rather than reading efficiency off the model:
\eqref{eq:OCS_efficiency} gives only
$\sum_{i \in c(g)} v_i = \pi_g \le V(c(g))$, and closing that to equality is the
\textup{NP} half.

\begin{corollary}[No compact description of $\mathcal{S}$]
\label{cor:no-compact}
Suppose there is a polynomial-time algorithm that, given an instance,
constructs a system in the payoff variables $v$ and auxiliary variables $w$,
of size polynomial in the input, with polynomial-time checkable constraints,
with every $v \in \mathcal{S}$ the projection of some $w$ of polynomial
encoding length, and whose projection onto $v$ is exactly $\mathcal{S}$. Then
$\textup{NP} = \textup{coNP}$.
\end{corollary}

\begin{proof}{Proof.}
Under the hypothesis, deciding $v \in \mathcal{S}$ lies in \textup{NP}:
construct the system, guess
$w$ and check the constraints. By the last part of
Proposition~\ref{prop:scs-check-dp} that decision problem is
\textup{coNP}-hard, so $\textup{coNP} \subseteq \textup{NP}$. \Halmos
\end{proof}

Theorem~\ref{thm:CSS_projection} describes $\mathcal{S}$ exactly, but only
through a family of inequalities that cannot be written down in advance.
Corollary~\ref{cor:no-compact} says this is not an artifact of our encoding.
Because it permits auxiliary variables and bounds the total size, it forbids
every uniformly constructible polynomial-size exact formulation, not only
those in the payoff variables:
exactness cannot be bought by trading inequalities for variables. What it does
not exclude is $\mathcal{F}^{\mathrm{stab}}$ itself, which uses polynomially
many variables and exponentially many inequalities.
What creates the gap is stability rather than integrality; the same obstruction
appears for worth functions computable in polynomial time, where core membership
over coalition structures is already \textup{coNP}-complete.

\begin{proposition}[Optimization over stable structures]
\label{prop:oscs-sigma2}
Deciding whether there is an SCS $(CS(N), v)$ with $\sum_{i \in N} v_i \ge K$
lies in $\Sigma_2^p$ and is \textup{NP}-hard.
\end{proposition}

\begin{proof}{Proof.}
\emph{Membership.} For each $c \in CS(N)$ the restriction of $v$ to $c$ lies in
the Core of $(c, V_c)$, the polyhedron
$\{ w : \sum_{i \in c} w_i = V(c),\ \sum_{i \in s} w_i \ge V(s)\ \forall\,
\emptyset \ne s \subsetneq c \}$. Singleton stability gives
$w_i \ge V(\{i\}) \ge 0$ and efficiency gives $w_i \le V(c)$, so the polytope is
bounded and nonempty and has a vertex, which solves a nonsingular
$|c| \times |c|$ subsystem with $0$--$1$ coefficients and coalition values as
right-hand sides. By Cramer's rule and the Hadamard bound the vertex is
polynomially encodable. Replacing $v$ on each block by such a vertex preserves
the structure and the total payoff, so the witness is polynomially balanced, and
verifying it means one instance of Proposition~\ref{prop:scs-check-dp}, which
lies in $\textup{DP} \subseteq \textup{P}^{\textup{NP}}$. Hence the problem lies
in $\exists \cdot \textup{P}^{\textup{NP}} = \Sigma_2^p$.

\emph{Hardness.} Take $|N| = 1$; the only SCS payoff is $v_1 = V(\{1\})$.
\Halmos
\end{proof}

We do not settle the exact complexity of the OSCS problem and regard
$\Sigma_2^p$-completeness as the natural conjecture;
Corollary~\ref{cor:no-compact} does not depend on it. The obstruction is that
the $\Delta_2^p$-hardness constructions for core problems over coalition
structures assume a polynomial-time worth function, and in that regime the
second level is not reached.

\begin{remark}[Where the difficulty lies on our instances]
\label{rem:not-the-knapsack}
Proposition~\ref{prop:ocs-npc} is not a statement about the instances of
Section~\ref{s:comp_results}. With $a_j \sim U\{5,10\}$ the coalition programs
are easy: Table~\ref{tab:bruteforce} reports $262{,}143$ values in $2{,}126.9$
seconds at $n = 18$, $\alpha = 1.0$, about $8$ milliseconds each, and
$3{,}628.4$ seconds at $\alpha = 0.5$, about $14$ milliseconds each. The gain in
Section~\ref{ss:res_cost} comes from evaluating $766$ values rather than
$10^9$, not from evaluating any one faster.
\end{remark}

\begin{remark}[Why fission stability, and not the coalition structure core]
\label{rem:fusion-warning}
Our notion allows only subsets of a \emph{formed} coalition to block. The
coalition structure core \citep[Def.~3.31]{koczy2018partition} instead requires
$\sum_{i \in s} v_i \ge V(s)$ for every $s \subseteq N$. Under that definition
the question this paper asks does not arise. Let $(CS(N), v)$ satisfy it.
Efficiency on each block gives
$\sum_{i \in N} v_i = \sum_{c \in CS(N)} V(c)$, while applying the requirement
blockwise to an arbitrary partition $\mathcal{Q}$ gives
$\sum_{i \in N} v_i = \sum_{T \in \mathcal{Q}} \sum_{i \in T} v_i \ge
\sum_{T \in \mathcal{Q}} V(T)$. Taking the maximum over $\mathcal{Q}$,
\begin{equation}
  \sum_{c \in CS(N)} V(c) \;=\; \sum_{i \in N} v_i
  \;\ge\; \max_{\mathcal{Q}} \sum_{T \in \mathcal{Q}} V(T),
  \label{eq:csc_forces_ocs}
\end{equation}
so $CS(N)$ is already an optimal coalition structure. Under the coalition
structure core every stable structure is welfare-maximal by definition, and the
gap between an OCS and an OSCS---the quantity Section~\ref{ss:res_solutions}
measures---is identically zero.

The stronger notion is also frequently empty here. Call $V$ \emph{cohesive} if
$\sum_{c \in CS(N)} V(c) \le V(N)$ for every partition; superadditivity implies
cohesion. Under cohesion \eqref{eq:csc_forces_ocs} and the requirement at
$s = N$ give $\sum_{i \in N} v_i = V(N)$, so $v$ lies in the Core of the grand
coalition. Under full pooling Proposition~\ref{prop:superadd_natural} makes $V$
superadditive, and at $n = 8$, $\alpha = 1.0$ our certified runs report an
optimal coalition structure of value $5{,}609$ against a best stable structure
of $5{,}606$: no stable structure attains $V(N)$, so the grand coalition's Core
is empty and with it the coalition structure core.

Fission stability is what makes an SCS always exist
(Proposition~\ref{prop:ex_css}) and the price of stability a computable
quantity rather than a definitional zero. The cost is that our structures are
not protected against a coalition spanning two blocks, which is the fusion direction
we leave open.
\end{remark}

\section{A Global Bound on the Objective}\label{app:objbound}

The root-bound margin of Section~\ref{ss:res_formulation} compares
$\mathrm{AGG}_\alpha$ and $\mathrm{DIS}_\alpha$ as stated. A single inequality
narrows it. Since $\sum_{i \in N} v_i = \sum_{g \in G} \pi_g$ by
\eqref{eq:OCS_partition}--\eqref{eq:OCS_efficiency}, since
$\pi_g \le V_\alpha(c(g)) \le V_{\alpha=1}(c(g))$, and since $V_{\alpha=1}$ is
superadditive by Proposition~\ref{prop:superadd_natural},
\begin{equation}
  \sum_{i \in N} v_i \;\le\; V_{\alpha=1}(N)
  \label{eq:obj_bound}
\end{equation}
is valid for both formulations at every $\alpha$, at the cost of one
coalition-value solve. At $n = 4$ it takes the aggregated root bound from
$4.5$ times the integer optimum to $1.02$ at $\alpha = 0.5$ and to the optimum
itself at $\alpha = 1.0$, which is tighter than the disaggregated
root bound.

We do not use it, for two reasons. At $\alpha = 1.0$ the bound is not
informative: Proposition~\ref{prop:superadd_natural} already identifies the
grand coalition as an OCS, so computing $V_{\alpha=1}(N)$ solves the problem
that \eqref{eq:obj_bound} is meant to help solve. And bounding the objective
directly leaves the relaxation tight at every node of the tree, which removes
the differences in bound that branching uses to choose variables; node counts
rose whenever we added it. The informative case is $\alpha < 1$ and the OSCS
problem, where the optimum lies strictly below $V_{\alpha=1}(N)$, and there
\eqref{eq:obj_bound} deserves a proper computational study.

\section{Additional Computational Results}\label{app:addl}

Table~\ref{tab:ablation_lifted} reports the lifted-versus-basic cut
ablation and Table~\ref{tab:SOCS_formulation} the AGG-versus-DIS
formulation effect for the OSCS methods; both are discussed in
Section~\ref{s:numerical}.
Table~\ref{tab:enum_lazy_DL} completes the ten configurations of
Experiment~1 by reporting the cutplane and lazy runs under the disaggregated
formulation.
Table~\ref{tab:heur_diag} reports per-instance diagnostics of the
SCS-feasible heuristic (phase times, Core checks, pool sizes, and the
seeded partition), complementing the Heur and HGap columns of
Table~\ref{tab:SOCS_methods}.
Table~\ref{tab:refinement_summary} summarizes the payoff refinement of
Section~\ref{ss:coalition_refinement} by coalition size.
Table~\ref{tab:bruteforce} reports the comparison with the exhaustive
baseline of Section~\ref{s:numerical}.
\paragraph{Exhaustive baseline: detailed comparison.}
In Table~\ref{tab:bruteforce}, ``Evaluations (ours)'' counts the distinct
coalition values computed by heuristic plus solver (one shared cache),
``Time, ours'' is the wall-clock time of lazy\_MIP\_ws (AGG) for the
entire OSCS computation, and the BF columns report the baseline's two
stages ($V$-enumeration measured to $n = 18$, OCS to $n = 16$, OSCS
search to $n = 14$).
At small $n$ under full pooling the heuristic's Core checks alone touch
every subset, so the evaluation economy of Table~\ref{tab:bruteforce} is
an asymptotic phenomenon, emerging from $n \approx 10$ onward; brute force
scales as $\Theta(2^n)$ in value evaluations alone, already requiring over
2{,}000 seconds at $n = 18$ and an extrapolated $10^7$ seconds at
$n = 30$, whereas the lazy framework queries $V$ only where stability of a
candidate partition is actually at stake. For exact certification, the two
routes are in fact comparable through $n \approx 14$--$16$---brute force
is even competitive at $n = 16$, $\alpha = 1.0$, where its extrapolated
total of roughly $18$ minutes undercuts our time limit. The separation is
in scaling: the OSCS-search stage passes from hours at $n = 20$ to days at
$n = 22$ and to more than $25$ years at $n = 30$. The deeper difference is
that brute force is all-or-nothing: interrupted before completion, it
yields neither a usable incumbent nor a bound, whereas every incumbent
accepted by the lazy framework carries a certified optimality gap, and every
incumbent lazy\_MIP\_ws returned was verified SCS-feasible, so the pipeline
degrades gracefully as $n$ grows.
\paragraph{Lifted versus basic cuts: details.}
On the five double-timeout cells of the ablation, neither variant achieves
systematically smaller final gaps (each is smaller on some cells; all gaps
remain below $0.5\%$): since neither family tightens the root relaxation,
any advantage must arise from path-dependent branch-and-bound effects, and
none is systematic on these instances. The fallback counts are 0 of
342{,}585 added cuts at $\alpha = 1.0$---structural, since the full-pooling
bound makes $V$ superadditive and the synergy condition
$\Delta(s) \ge 0$ of Proposition~\ref{prop:lifted_valid} holds for every
subset---versus 25{,}954 of 41{,}346 ($62.8\%$) at $\alpha = 0.5$, both
aggregated over all lazy, lazy\_MIP, and lazy\_MIP\_ws runs of the two
experiments, where restricted bounds create subadditive subsets.
\paragraph{Formulation effect: details.}
In the $\dagger$ cells of Table~\ref{tab:SOCS_formulation}, the solver
returns singleton partitions with near-zero objectives (303--375 versus
the expected 12{,}900--20{,}900); the first such collapse already appears
at $n = 20$, $\alpha = 0.5$, and the one remaining configuration with
$n \ge 24$ times out normally with a sub-1\% gap. Despite DIS's tighter
relaxation and smaller constraint system
(Proposition~\ref{prop:disagg_LP}), the solver generates poor incumbents on
these large OSCS instances without an initial solution---an observed search
phenomenon whose cause the present evidence does not establish---and
lazy\_MIP\_ws with DIS achieves gaps comparable to AGG (both under 1\%)
precisely because the warm start removes that burden.
\paragraph{Coalition structure patterns: details.}
The near-zero cost of stability has a structural explanation. At
$\alpha = 0.5$, the formed coalitions are predominantly pairs, and a pair
is worth forming precisely when
$V(\{i,j\}) \ge V(\{i\}) + V(\{j\})$---in which case a stable two-way
split always exists, so every OCS pair is automatically stabilizable. At
$\alpha = 1.0$, the pooled coalitions in our instances are large enough
that splitting them costs little or nothing (at $n = 10$, both the
grand-coalition OCS and the two-coalition OSCS attain $7{,}564$). This
reflects our test instances rather than CIPGs in general: with
$V(\{i\}) = 0$, each pair worth $10$, and $V(N) = 12$, the three pair
inequalities force a total payoff of at least $15 > 12$, so the grand
coalition's Core is empty and the cost of stability is nearly $17\%$.
\paragraph{Payoff refinement: details.}
The refinement enumerates all subset values $V(s)$ of each formed
coalition with the coalition-level IP (under the same
$\alpha$-restricted policy) and solves the nucleolus by the sequential-LP
scheme; formed coalitions are small, so at most $2^{|c|}-1$ small IPs are
required per coalition. Since large coalitions arise mainly under
$\alpha = 1.0$, the Shapley Core-membership contrast across pooling
regimes is sharp (100/102 at $\alpha = 0.5$ versus 18/44 at
$\alpha = 1.0$)---an empirical reflection of the non-convexity discussed in
Section~\ref{ss:V_structural}.
\bgroup
\def\arraystretch{1.05}\SingleSpacedXI
\begin{table}[tbp]
\caption{Coalition-value evaluations: proposed pipeline (lazy\_MIP\_ws, AGG)
versus the exhaustive baseline. Measured $V$-enumeration times differ by
$\alpha$---at $n = 18$, $3{,}628$ seconds at $\alpha = 0.5$ against $2{,}127$ at
$\alpha = 1.0$---and the brute-force columns report the $\alpha = 1.0$ series,
which understates the baseline's cost at $\alpha = 0.5$.}
\label{tab:bruteforce}
\centering
\scalebox{0.85}{
\begin{tabular}{rrrrrrrr}
\toprule
& & \multicolumn{2}{c}{\textbf{Evaluations (ours)}} & \multicolumn{2}{c}{\textbf{Time, ours (s)}} & \multicolumn{2}{c}{\textbf{BF time (s)}} \\
\cmidrule(lr){3-4} \cmidrule(lr){5-6} \cmidrule(lr){7-8}
$n$ & $2^n - 1$ & $\alpha = 0.5$ & $\alpha = 1.0$ & $\alpha = 0.5$ & $\alpha = 1.0$ & $V$-enum. & OSCS search \\
\midrule
2 & 3 & 3 & 3 & 0.09 & 0.01 & 0.0 & 0.0 \\
4 & 15 & 12 & 15 & 0.2 & 0.1 & 0.1 & 0.0 \\
6 & 63 & 28 & 63 & 0.6 & 0.3 & 0.3 & 0.1 \\
8 & 255 & 51 & 255 & 2.0 & 2.1 & 1.3 & 0.2 \\
10 & 1,023 & 81 & 321 & 4.0 & 3.1 & 6.4 & 1.3 \\
12 & 4,095 & 123 & 1,009 & 17 & 13 & 26.5 & 10.1 \\
14 & 16,383 & 184 & 6,151 & 72 & 172 & 112.8 & 76.5 \\
16 & 65,535 & 213 & 6,076 & 74 & 1800 & 509.5 & $\approx$580 \\
18 & 262,143 & 271 & 6,025 & TL (0.40\%) & TL (0.03\%) & 2,126.9 & $\approx$4,392 \\
20 & 1,048,575 & 357 & 14,732 & TL (0.12\%) & TL (0.02\%) & $\approx$8,878 & $\approx$33,224 \\
22 & 4,194,303 & 418 & 15,840 & TL (0.72\%) & TL (0.01\%) & $\approx$37,061 & $\approx$251,597 \\
24 & 16,777,215 & 478 & 13,567 & TL (0.64\%) & TL (0.09\%) & $\approx$154,706 & $\approx$1,905,298 \\
26 & 67,108,863 & 573 & 17,448 & TL (0.87\%) & TL (0.09\%) & $\approx$645,792 & $\approx$14,428,445 \\
28 & 268,435,455 & 666 & 11,727 & TL (0.60\%) & TL (0.08\%) & $\approx$2,695,743 & $\approx$109,263,786 \\
30 & 1,073,741,823 & 766 & 11,840 & TL (0.94\%) & TL (0.05\%) & $\approx$11,252,893 & $\approx$827,433,238 \\
\bottomrule
\end{tabular}
}
\par\vspace{4pt}
\parbox{0.92\linewidth}{\footnotesize \emph{Note.} Brute-force (BF)
times are measured up to the stage caps ($V$-enumeration to $n = 18$, OSCS
search to $n = 14$) and geometrically extrapolated beyond (marked
$\approx$); TL entries in ``Time, ours'' report the remaining MIP gap. For
timeout rows, ``Evaluations (ours)'' records search progress at termination
and need not be monotone in $n$.
Column definitions are given in the text.}
\end{table}
\egroup

\bgroup
\def\arraystretch{1.15}\SingleSpacedXI
\begin{table}[tbp]
\caption{Ablation: lifted stability cuts~\eqref{eq:lifted_cut} versus basic
cuts~\eqref{eq:stability_bigM} within lazy\_MIP (AGG, no warm start).}
\label{tab:ablation_lifted}
\centering
\scalebox{0.85}{
\begin{tabular}{rcrrrrrrrrr}
\toprule
& & \multicolumn{4}{c}{\textbf{lifted}} & \multicolumn{4}{c}{\textbf{basic}} & \\
\cmidrule(lr){3-6} \cmidrule(lr){7-10}
$n$ & $\alpha$ & Time & Gap & Cuts & CB & Time & Gap & Cuts & CB & Obj \\
\midrule
2 & 0.5 & 0.01 & --- & 1 & 8 & 0.02 & --- & 1 & 8 & 1340 \\
2 & 1.0 & 0.01 & --- & 2 & 13 & 0.02 & --- & 2 & 13 & 1365 \\
4 & 0.5 & 0.2 & --- & 7 & 17 & 0.2 & --- & 7 & 17 & 3037 \\
4 & 1.0 & 0.1 & --- & 14 & 15 & 0.1 & --- & 14 & 14 & 3105 \\
6 & 0.5 & 0.5 & --- & 35 & 18 & 0.5 & --- & 35 & 18 & 3622 \\
6 & 1.0 & 0.4 & --- & 54 & 18 & 0.3 & --- & 56 & 17 & 3696 \\
8 & 0.5 & 2.3 & --- & 158 & 20 & 7.7 & --- & 151 & 18 & 5487 \\
8 & 1.0 & 1.8 & --- & 225 & 23 & 10 & --- & 296 & 49 & 5606 \\
10 & 0.5 & 12 & --- & 137 & 26 & 10 & --- & 156 & 30 & 7483 \\
10 & 1.0 & 9.4 & --- & 257 & 45 & 10 & --- & 300 & 80 & 7564 \\
12 & 0.5 & 19 & --- & 91 & 33 & 19 & --- & 118 & 25 & 8244 \\
12 & 1.0 & 75 & --- & 1348 & 105 & 82 & --- & 1003 & 48 & 8325 \\
14 & 0.5 & 147 & --- & 616 & 46 & 109 & --- & 326 & 39 & 8855 \\
14 & 1.0 & 133 & --- & 2316 & 133 & 61 & --- & 1611 & 64 & 8953 \\
16 & 0.5 & 162 & --- & 354 & 40 & 118 & --- & 223 & 35 & 10250 \\
16 & 1.0 & TL & 0.02 & 4533 & 228 & TL & 0.04 & 2488 & 136 & 10341 \\
18 & 0.5 & TL & 0.39 & 188 & 62 & TL & 0.47 & 73 & 59 & 11425 \\
18 & 1.0 & TL & 0.03 & 4809 & 181 & TL & 0.03 & 2630 & 81 & 11559 \\
20 & 0.5 & TL & 0.28 & 115 & 95 & TL & 0.27 & 107 & 83 & 12905 \\
20 & 1.0 & TL & 0.09 & 8692 & 234 & TL & 0.03 & 6131 & 181 & 12970 \\
\bottomrule
\end{tabular}
}
\par\vspace{4pt}
\parbox{0.92\linewidth}{\footnotesize \emph{Note.} All Experiment~1
instances. Time (s), MIP gap (\%) at timeout (``---'' means solved to
optimality), number of added cuts, and callback invocations (CB). Obj is
the common incumbent (identical optima wherever both arms certify
optimality).}
\end{table}
\egroup

\bgroup
\def\arraystretch{1.15}\SingleSpacedXI
\begin{table}[tbp]
\caption{Payoff refinement summary over the $146$ formed OSCS coalitions of
size $\ge 2$.}
\label{tab:refinement_summary}
\centering
\scalebox{0.85}{
\begin{tabular}{crrrr}
\toprule
$|c|$ & \#Coalitions & Shapley $\in$ Core & Max.\ viol.\ (\% of $V(c)$) & Thin Core \\
\midrule
2 & 98 & 98 & --- & 0 \\
3 & 12 & 10 & 0.145 & 0 \\
4 & 10 & 7 & 0.057 & 0 \\
5 & 4 & 1 & 0.130 & 0 \\
6 & 5 & 2 & 0.049 & 0 \\
7 & 8 & 0 & 0.232 & 2 \\
8 & 4 & 0 & 0.201 & 2 \\
9 & 5 & 0 & 0.078 & 3 \\
\midrule
Total & 146 & 118 & 0.232 & 7 \\
\bottomrule
\end{tabular}
}
\par\vspace{4pt}
\parbox{0.92\linewidth}{\footnotesize \emph{Note.} All instances
($n = 2,\dots,30$, both $\alpha$ values). Shapley value and nucleolus of
each restricted game are computed by enumerating all subset values $V(s)$
with the coalition-level IP. ``Thin Core'' counts coalitions with
least-core $\varepsilon > -0.01$. The nucleolus is Core-feasible in all
$146$ cases.}
\end{table}
\egroup

\bgroup
\renewcommand{\arraystretch}{0.95}\SingleSpacedXI
\setlength{\tabcolsep}{4pt}
\begin{table}[tbp]
\caption{OSCS formulation effect: AGG vs.\ DIS within lazy\_MIP and lazy\_MIP\_ws.}
\label{tab:SOCS_formulation}
\centering
\footnotesize
\begin{tabular}{rcrrrrrrrrrrr}
\toprule
& & \multicolumn{5}{c}{\textbf{lazy\_MIP}} & \multicolumn{5}{c}{\textbf{lazy\_MIP\_ws}} & \\
\cmidrule(lr){3-7} \cmidrule(lr){8-12}
& & \multicolumn{2}{c}{AGG} & \multicolumn{2}{c}{DIS} & & \multicolumn{2}{c}{AGG} & \multicolumn{2}{c}{DIS} & & \\
\cmidrule(lr){3-4} \cmidrule(lr){5-6} \cmidrule(lr){8-9} \cmidrule(lr){10-11}
$n$ & $\alpha$ & Time & Gap & Time & Gap & Better & Time & Gap & Time & Gap & Better & Obj \\
\midrule
2 & 0.5 & 0.01 & --- & 0.01 & --- & DIS & 0.09 & --- & 0.01 & --- & DIS & 1340 \\
2 & 1.0 & 0.01 & --- & 0.01 & --- & DIS & 0.01 & --- & 0.01 & --- & DIS & 1365 \\
4 & 0.5 & 0.2 & --- & 0.3 & --- & AGG & 0.2 & --- & 0.2 & --- & AGG & 3037 \\
4 & 1.0 & 0.1 & --- & 0.2 & --- & AGG & 0.1 & --- & 0.1 & --- & DIS & 3105 \\
6 & 0.5 & 0.5 & --- & 0.6 & --- & AGG & 0.6 & --- & 0.5 & --- & DIS & 3622 \\
6 & 1.0 & 0.4 & --- & 0.3 & --- & DIS & 0.3 & --- & 0.2 & --- & DIS & 3696 \\
8 & 0.5 & 2.3 & --- & 6.5 & --- & AGG & 2.0 & --- & 2.9 & --- & AGG & 5487 \\
8 & 1.0 & 1.8 & --- & 3.1 & --- & AGG & 2.1 & --- & 2.2 & --- & AGG & 5606 \\
10 & 0.5 & 12 & --- & 3.0 & --- & DIS & 4.0 & --- & 2.9 & --- & DIS & 7483 \\
10 & 1.0 & 9.4 & --- & 15 & --- & AGG & 3.1 & --- & 1.8 & --- & DIS & 7564 \\
12 & 0.5 & 19 & --- & 74 & --- & AGG & 17 & --- & 5.6 & --- & DIS & 8244 \\
12 & 1.0 & 75 & --- & 217 & --- & AGG & 13 & --- & 169 & --- & AGG & 8325 \\
14 & 0.5 & 147 & --- & 99 & --- & DIS & 72 & --- & 66 & --- & DIS & 8855 \\
14 & 1.0 & 133 & --- & 42 & --- & DIS & 172 & --- & 182 & --- & AGG & 8953 \\
16 & 0.5 & 162 & --- & 259 & --- & AGG & 74 & --- & TL & 0.33 & AGG & 10250 \\
16 & 1.0 & TL & 0.02 & TL & 0.06 & AGG & 1800 & --- & TL & 0.04 & AGG & 10341 \\
18 & 0.5 & TL & 0.39 & TL & 0.66 & AGG & TL & 0.40 & TL & 0.43 & AGG & 11425 \\
18 & 1.0 & TL & 0.03 & TL & 0.05 & AGG & TL & 0.03 & TL & 0.12 & AGG & 11559 \\
20 & 0.5 & TL & 0.28 & $\dagger$ & $\dagger$ & AGG & TL & 0.12 & TL & 0.33 & AGG & 12905 \\
20 & 1.0 & TL & 0.09 & TL & 0.02 & DIS & TL & 0.02 & TL & 0.13 & AGG & 12970 \\
\midrule
22 & 0.5 & TL & 0.91 & TL & 1.0 & AGG & TL & 0.72 & TL & 0.72 & --- & 15125 \\
22 & 1.0 & TL & 0.10 & TL & 0.28 & AGG & TL & 0.01 & TL & 0.01 & --- & 15327 \\
24 & 0.5 & TL & 0.83 & $\dagger$ & $\dagger$ & AGG & TL & 0.64 & TL & 0.74 & AGG & 16278 \\
24 & 1.0 & TL & 0.05 & TL & 0.08 & AGG & TL & 0.09 & TL & 0.11 & AGG & 16437 \\
26 & 0.5 & TL & 1.1 & $\dagger$ & $\dagger$ & AGG & TL & 0.87 & TL & 0.91 & AGG & 17816 \\
26 & 1.0 & TL & 0.17 & $\dagger$ & $\dagger$ & AGG & TL & 0.09 & TL & 0.11 & AGG & 18007 \\
28 & 0.5 & TL & 0.89 & $\dagger$ & $\dagger$ & AGG & TL & 0.60 & TL & 0.86 & AGG & 18735 \\
28 & 1.0 & TL & 0.16 & $\dagger$ & $\dagger$ & AGG & TL & 0.08 & TL & 0.08 & --- & 18916 \\
30 & 0.5 & TL & 1.1 & $\dagger$ & $\dagger$ & AGG & TL & 0.94 & TL & 0.96 & AGG & 20653 \\
30 & 1.0 & TL & 0.11 & $\dagger$ & $\dagger$ & AGG & TL & 0.05 & TL & 0.05 & --- & 20854 \\
\bottomrule
\end{tabular}
\par\vspace{4pt}
\parbox{0.92\linewidth}{\footnotesize \emph{Note.} Wall-clock time (s)
and MIP gap (\%) at timeout; ``---'' means solved to optimality. Obj is the
lazy\_MIP (AGG) incumbent. ``Better'' follows the convention of
Table~\ref{tab:OCS_formulation}. $\dagger$: solver degradation (see text).}
\end{table}
\egroup

\bgroup
\def\arraystretch{1.15}\SingleSpacedXI
\begin{table}[tbp]
\caption{The cutplane and lazy configurations under the disaggregated formulation
(DIS), $n \le 20$.}
\label{tab:enum_lazy_DL}
\centering
\scalebox{0.9}{
\begin{tabular}{rcrrrrrr}
\toprule
& & \multicolumn{3}{c}{\textbf{cutplane (DIS)}} & \multicolumn{3}{c}{\textbf{lazy (DIS)}} \\
\cmidrule(lr){3-5} \cmidrule(lr){6-8}
$n$ & $\alpha$ & Time & Gap & Cuts & Time & Gap & Cuts \\
\midrule
2 & 0.5 & 0.01 & --- & 0 & 0.01 & --- & 1 \\
2 & 1.0 & 0.01 & --- & 1 & 0.02 & --- & 2 \\
4 & 0.5 & 0.3 & --- & 1 & 0.3 & --- & 11 \\
4 & 1.0 & 0.3 & --- & 14 & 0.2 & --- & 21 \\
6 & 0.5 & 0.8 & --- & 2 & 0.5 & --- & 33 \\
6 & 1.0 & 0.6 & --- & 50 & 0.3 & --- & 71 \\
8 & 0.5 & 3.1 & --- & 4 & 6.4 & --- & 154 \\
8 & 1.0 & 3.7 & --- & 203 & 3.1 & --- & 271 \\
10 & 0.5 & 8.5 & --- & 4 & 3.1 & --- & 147 \\
10 & 1.0 & 12 & --- & 850 & 22 & --- & 446 \\
12 & 0.5 & 17 & --- & 6 & 73 & --- & 298 \\
12 & 1.0 & 408 & --- & 2586 & 218 & --- & 987 \\
14 & 0.5 & 231 & --- & 23 & 101 & --- & 147 \\
14 & 1.0 & 117 & --- & 12602 & 42 & --- & 721 \\
16 & 0.5 & 502 & --- & 15 & 265 & --- & 1029 \\
16 & 1.0 & TL & n/a & 0 & TL & 0.04 & 2962 \\
18 & 0.5 & TL & n/a & 0 & TL & 0.66 & 566 \\
18 & 1.0 & TL & n/a & 0 & TL & 0.03 & 5952 \\
20 & 0.5 & TL & n/a & 0 & TL & 0.63 & 2150 \\
20 & 1.0 & TL & n/a & 0 & TL & 0.13 & 6302 \\
\bottomrule
\end{tabular}
}
\par\vspace{4pt}
\parbox{0.92\linewidth}{\footnotesize \emph{Note.} Completes the
configurations announced in Section~\ref{s:comp_results}; the AGG
counterparts appear in Table~\ref{tab:SOCS_methods}. Wall-clock time (s),
MIP gap (\%) at timeout (``---'' means solved to optimality; ``n/a'' means
no single MIP bound is available), and number of stability cuts.}
\end{table}
\egroup

\bgroup
\def\arraystretch{1.15}\SingleSpacedXI
\begin{table}[tbp]
\caption{Diagnostics of the SCS-feasible heuristic (lazy\_MIP\_ws, AGG).}
\label{tab:heur_diag}
\centering
\scalebox{0.9}{
\begin{tabular}{rcrrrrrrr}
\toprule
$n$ & $\alpha$ & Phase 1 (s) & Phase 2 (s) & Core checks & Pool $|\mathcal{P}|$ & Heur.\ Obj & $|CS|$ & avg.\ $|c|$ \\
\midrule
2 & 0.5 & 0.05 & 0.00 & 0 & 2 & 1340 & 2 & 1.0 \\
2 & 1.0 & 0.00 & 0.00 & 1 & 3 & 1365 & 1 & 2.0 \\
4 & 0.5 & 0.02 & 0.00 & 2 & 6 & 3037 & 3 & 1.3 \\
4 & 1.0 & 0.03 & 0.00 & 9 & 13 & 3105 & 1 & 4.0 \\
6 & 0.5 & 0.04 & 0.00 & 9 & 15 & 3622 & 4 & 1.5 \\
6 & 1.0 & 0.13 & 0.00 & 24 & 30 & 3696 & 1 & 6.0 \\
8 & 0.5 & 0.08 & 0.00 & 11 & 19 & 5487 & 4 & 2.0 \\
8 & 1.0 & 0.63 & 0.00 & 50 & 52 & 5600 & 2 & 4.0 \\
10 & 0.5 & 0.18 & 0.00 & 15 & 25 & 7483 & 6 & 1.7 \\
10 & 1.0 & 0.65 & 0.00 & 74 & 81 & 7564 & 2 & 5.0 \\
12 & 0.5 & 0.17 & 0.00 & 35 & 47 & 8244 & 6 & 2.0 \\
12 & 1.0 & 2.13 & 0.00 & 112 & 118 & 8325 & 2 & 6.0 \\
14 & 0.5 & 0.23 & 0.01 & 41 & 55 & 8848 & 8 & 1.8 \\
14 & 1.0 & 3.61 & 0.01 & 147 & 150 & 8940 & 4 & 3.5 \\
16 & 0.5 & 0.50 & 0.01 & 59 & 75 & 10250 & 9 & 1.8 \\
16 & 1.0 & 24.1 & 0.01 & 199 & 198 & 10340 & 2 & 8.0 \\
18 & 0.5 & 0.68 & 0.01 & 75 & 92 & 11420 & 10 & 1.8 \\
18 & 1.0 & 28.4 & 0.01 & 285 & 267 & 11549 & 3 & 6.0 \\
20 & 0.5 & 1.38 & 0.01 & 115 & 129 & 12896 & 11 & 1.8 \\
20 & 1.0 & 17.3 & 0.01 & 329 & 304 & 12963 & 5 & 4.0 \\
\midrule
22 & 0.5 & 1.82 & 0.01 & 125 & 144 & 15151 & 11 & 2.0 \\
22 & 1.0 & 98.5 & 0.01 & 469 & 413 & 15339 & 3 & 7.3 \\
24 & 0.5 & 2.23 & 0.02 & 98 & 121 & 16304 & 13 & 1.8 \\
24 & 1.0 & 10.3 & 0.02 & 464 & 434 & 16428 & 5 & 4.8 \\
26 & 0.5 & 2.94 & 0.02 & 207 & 217 & 17862 & 13 & 2.0 \\
26 & 1.0 & 106 & 0.06 & 725 & 564 & 18014 & 5 & 5.2 \\
28 & 0.5 & 3.57 & 0.01 & 160 & 185 & 18776 & 14 & 2.0 \\
28 & 1.0 & 65.5 & 0.02 & 725 & 614 & 18930 & 6 & 4.7 \\
30 & 0.5 & 4.46 & 0.02 & 210 & 237 & 20676 & 17 & 1.8 \\
30 & 1.0 & 60.1 & 0.02 & 774 & 708 & 20866 & 6 & 5.0 \\
\bottomrule
\end{tabular}
}
\par\vspace{4pt}
\parbox{0.92\linewidth}{\footnotesize \emph{Note.} Phase 1 is the
Core-verified bottom-up merging (Section~\ref{sec:heuristics}) and Phase 2
the set-partitioning IP over the collected pool $\mathcal{P}$ of
Core-feasible coalitions; ``Core checks'' counts Core-nonemptiness LPs
solved in Phase 1. Heur.\ Obj is the objective of the returned SCS-feasible
partition ($|CS|$ coalitions of average size avg.\ $|c|$), which seeds
lazy\_MIP\_ws. Phase 2 is negligible throughout because the pool sizes
remain in the hundreds; total heuristic time is dominated by the Core
checks on large pooled coalitions at $\alpha = 1.0$.}
\end{table}
\egroup

\clearpage
\section{Coalition Structure and Payoff Details}\label{app:details}

\subsection{Per-coalition solution concepts: details}
\label{app:refinement_details}

For each coalition \(c \in CS(N)\), the restricted game \((c,V_c)\) admits
the usual cooperative solution concepts:
(i) the Shapley value is well-defined for \((c,V_c)\);
(ii) the nucleolus of \((c,V_c)\) exists and lies in its Core whenever the
Core is nonempty, as is the case for every formed coalition of an SCS; and
(iii) if \((c,V_c)\) is convex (supermodular), then the Shapley value of
\((c,V_c)\) also belongs to its Core.
Non-uniqueness of the OSCS payoffs is harmless: replacing the returned
allocation \(v^{(c)}\) by any other Core allocation \(w^{(c)}\) of
\((c,V_c)\) preserves efficiency \(\sum_{i\in c} w^{(c)}_i = V(c)\) and all
internal Core inequalities \(\sum_{i\in s} w^{(c)}_i \ge V(s)\), so the SCS
constraints remain satisfied. The least-core value is
\[
\varepsilon(c) \;:=\; \min_{\varepsilon,\,\psi} \Bigl\{ \varepsilon \;:\;
\sum_{i\in c} \psi_i = V(c),\;\;
\sum_{i\in s} \psi_i \;\ge\; V(s) - \varepsilon
\ \ \forall\,\emptyset \neq s \subsetneq c \Bigr\},
\]
where the quantification is over proper nonempty subsets; the Core of
\((c,V_c)\) is nonempty if and only if \(\varepsilon(c) \le 0\), and
allocations attaining \(\varepsilon(c) < 0\) lie in the relative interior
of the Core.

Table~\ref{tab:coalition_structures} reports the optimal coalition
structures found by OCS and OSCS for all small instances ($n \le 12$),
illustrating the structural patterns discussed in
Section~\ref{s:numerical}.
Table~\ref{tab:refinement_detail} reports the corresponding per-coalition
payoff refinement: the Shapley value and the nucleolus of each formed OSCS
coalition of size at least two.
\bgroup
\def\arraystretch{1.25}\SingleSpacedXI
\begin{table}[tbp]
\caption{Coalition structures for small instances ($n \le 12$).}
\label{tab:coalition_structures}
\centering
\scalebox{0.95}{
\begin{tabular}{rccclccl}
\toprule
$n$ & $\alpha$ & \multicolumn{3}{c}{\textbf{OCS (AGG)}} & \multicolumn{3}{c}{\textbf{OSCS (lazy\_MIP\_ws, AGG)}} \\
\cmidrule(lr){3-5} \cmidrule(lr){6-8}
& & Obj & $|CS|$ & Coalition Structure & Obj & $|CS|$ & Coalition Structure \\
\midrule
2 & 0.5 & 1340 & 2 & $\scriptstyle \{0\}, \{1\}$ & 1340 & 2 & $\scriptstyle \{0\}, \{1\}$ \\
2 & 1.0 & 1365 & 1 & $\scriptstyle \{0,1\}$ & 1365 & 1 & $\scriptstyle \{0,1\}$ \\
\addlinespace[2pt]
4 & 0.5 & 3037 & 3 & $\scriptstyle \{0,3\}, \{1\}, \{2\}$ & 3037 & 3 & $\scriptstyle \{0,3\}, \{1\}, \{2\}$ \\
4 & 1.0 & 3105 & 1 & $\scriptstyle \{0,1,2,3\}$ & 3105 & 1 & $\scriptstyle \{0,1,2,3\}$ \\
\addlinespace[2pt]
6 & 0.5 & 3622 & 4 & $\scriptstyle \{0,1\}, \{2\}, \{3\}, \{4,5\}$ & 3622 & 4 & $\scriptstyle \{0,1\}, \{2\}, \{3\}, \{4,5\}$ \\
6 & 1.0 & 3696 & 1 & $\scriptstyle \{0,1,2,3,4,5\}$ & 3696 & 1 & $\scriptstyle \{0,1,2,3,4,5\}$ \\
\addlinespace[2pt]
8 & 0.5 & 5487 & 4 & $\scriptstyle \{0,5\}, \{1,3\}, \{2,6\}, \{4,7\}$ & 5487 & 4 & $\scriptstyle \{0,5\}, \{1,3\}, \{2,6\}, \{4,7\}$ \\
8 & 1.0 & 5609 & 1 & $\scriptstyle \{0,1,2,3,4,5,6,7\}$ & 5606$^*$ & 2 & $\scriptstyle \{0,2,4,7\}, \{1,3,5,6\}$ \\
\addlinespace[2pt]
10 & 0.5 & 7483 & 6 & $\scriptstyle \{0\}, \{1,3\}, \{2,9\}, \{4,7\}, \{5,6\}, \{8\}$ & 7483 & 6 & $\scriptstyle \{0\}, \{1,3\}, \{2,9\}, \{4,7\}, \{5,6\}, \{8\}$ \\
10 & 1.0 & 7564 & 1 & $\scriptstyle \{0,1,2,3,4,5,6,7,8,9\}$ & 7564 & 2 & $\scriptstyle \{0,2,4,5,6,7,8,9\}, \{1,3\}$ \\
\addlinespace[2pt]
12 & 0.5 & 8244 & 6 & $\scriptstyle \{0,7\}, \{1,6\}, \{2,9\}, \{3,5\}, \{4,8\}, \{10,11\}$ & 8244 & 6 & $\scriptstyle \{0,7\}, \{1,6\}, \{2,9\}, \{3,5\}, \{4,8\}, \{10,11\}$ \\
12 & 1.0 & 8325 & 1 & $\scriptstyle \{0,1,2,3,4,5,6,7,8,9,10,11\}$ & 8325 & 2 & $\scriptstyle \{0,1,3,4,5,6,8,10,11\}, \{2,7,9\}$ \\
\bottomrule
\multicolumn{8}{l}{\footnotesize $^*$Stability constraint reduces objective value.} \\
\end{tabular}
}
\par\vspace{4pt}
\parbox{0.92\linewidth}{\footnotesize \emph{Note.} OCS: optimal coalition
structure (AGG, no stability). OSCS: optimal stable coalition structure
(lazy\_MIP\_ws, AGG).}
\end{table}
\egroup

\bgroup
\def\arraystretch{1.2}\SingleSpacedXI
\begin{table}[tbp]
\caption{Per-coalition payoff refinement for small instances ($n \le 12$).}
\label{tab:refinement_detail}
\centering
\scalebox{0.9}{
\begin{tabular}{ccrlllcr}
\toprule
$n$ & $\alpha$ & $V(c)$ & Coalition & Shapley $\phi$ & Nucleolus & $\phi \in$ Core & $\varepsilon$ \\
\midrule
2 & 1.0 & 1365 & $\scriptstyle\{0,1\}$ & $\scriptstyle(582,784)$ & $\scriptstyle(582,784)$ & \checkmark & -12.50 \\
\addlinespace[2pt]
4 & 0.5 & 1596 & $\scriptstyle\{0,3\}$ & $\scriptstyle(704,892)$ & $\scriptstyle(704,892)$ & \checkmark & -17.00 \\
4 & 1.0 & 3105 & $\scriptstyle\{0,1,2,3\}$ & $\scriptstyle(709,745,740,911)$ & $\scriptstyle(709,742,739,916)$ & \checkmark & -5.33 \\
\addlinespace[2pt]
6 & 0.5 & 1364 & $\scriptstyle\{0,1\}$ & $\scriptstyle(798,566)$ & $\scriptstyle(798,566)$ & \checkmark & -2.00 \\
6 & 0.5 & 1077 & $\scriptstyle\{4,5\}$ & $\scriptstyle(692,385)$ & $\scriptstyle(692,385)$ & \checkmark & -10.00 \\
6 & 1.0 & 3696 & $\scriptstyle\{0,1,2,3,4,5\}$ & $\scriptstyle(809,585,627,594,699,383)$ & $\scriptstyle(806,588,626,598,699,380)$ & $\times$ & -0.67 \\
\addlinespace[2pt]
8 & 0.5 & 1507 & $\scriptstyle\{0,5\}$ & $\scriptstyle(743,764)$ & $\scriptstyle(743,764)$ & \checkmark & -6.00 \\
8 & 0.5 & 1333 & $\scriptstyle\{1,3\}$ & $\scriptstyle(614,719)$ & $\scriptstyle(614,719)$ & \checkmark & -2.00 \\
8 & 0.5 & 1198 & $\scriptstyle\{2,6\}$ & $\scriptstyle(567,631)$ & $\scriptstyle(567,631)$ & \checkmark & -11.00 \\
8 & 0.5 & 1449 & $\scriptstyle\{4,7\}$ & $\scriptstyle(690,759)$ & $\scriptstyle(690,759)$ & \checkmark & -18.00 \\
8 & 1.0 & 2790 & $\scriptstyle\{0,2,4,7\}$ & $\scriptstyle(761,566,693,771)$ & $\scriptstyle(763,560,691,776)$ & $\times$ & -4.00 \\
8 & 1.0 & 2816 & $\scriptstyle\{1,3,5,6\}$ & $\scriptstyle(639,751,786,640)$ & $\scriptstyle(637,756,789,633)$ & \checkmark & -9.20 \\
\addlinespace[2pt]
10 & 0.5 & 1578 & $\scriptstyle\{1,3\}$ & $\scriptstyle(846,732)$ & $\scriptstyle(846,732)$ & \checkmark & -6.50 \\
10 & 0.5 & 1355 & $\scriptstyle\{2,9\}$ & $\scriptstyle(602,752)$ & $\scriptstyle(602,752)$ & \checkmark & -4.50 \\
10 & 0.5 & 1450 & $\scriptstyle\{4,7\}$ & $\scriptstyle(696,754)$ & $\scriptstyle(696,754)$ & \checkmark & -19.50 \\
10 & 0.5 & 1487 & $\scriptstyle\{5,6\}$ & $\scriptstyle(817,670)$ & $\scriptstyle(817,670)$ & \checkmark & -20.00 \\
10 & 1.0 & 1602 & $\scriptstyle\{1,3\}$ & $\scriptstyle(858,744)$ & $\scriptstyle(858,744)$ & \checkmark & -18.50 \\
10 & 1.0 & 5962 & $\scriptstyle\{0,2,4,5,6,7,8,9\}$ & $\scriptstyle(892,606,691,819,679,764,745,766)$ & $\scriptstyle(892,606,691,820,680,766,743,763)$ & $\times$ & 0.00 \\
\addlinespace[2pt]
12 & 0.5 & 1312 & $\scriptstyle\{0,7\}$ & $\scriptstyle(636,676)$ & $\scriptstyle(636,676)$ & \checkmark & -5.50 \\
12 & 0.5 & 1331 & $\scriptstyle\{1,6\}$ & $\scriptstyle(796,535)$ & $\scriptstyle(796,535)$ & \checkmark & -18.00 \\
12 & 0.5 & 1370 & $\scriptstyle\{2,9\}$ & $\scriptstyle(659,711)$ & $\scriptstyle(659,711)$ & \checkmark & -16.00 \\
12 & 0.5 & 1284 & $\scriptstyle\{3,5\}$ & $\scriptstyle(672,612)$ & $\scriptstyle(672,612)$ & \checkmark & -11.00 \\
12 & 0.5 & 1447 & $\scriptstyle\{4,8\}$ & $\scriptstyle(662,784)$ & $\scriptstyle(662,784)$ & \checkmark & -25.50 \\
12 & 0.5 & 1500 & $\scriptstyle\{10,11\}$ & $\scriptstyle(706,794)$ & $\scriptstyle(706,794)$ & \checkmark & -4.50 \\
12 & 1.0 & 2052 & $\scriptstyle\{2,7,9\}$ & $\scriptstyle(663,680,708)$ & $\scriptstyle(668,676,707)$ & \checkmark & -5.67 \\
12 & 1.0 & 6273 & $\scriptstyle\{0,1,3,4,5,6,8,10,11\}$ & $\scriptstyle(644,809,673,653,624,547,805,711,806)$ & $\scriptstyle(640,811,672,653,624,550,807,710,806)$ & $\times$ & -0.67 \\
\bottomrule
\end{tabular}
}
\par\vspace{4pt}
\parbox{0.92\linewidth}{\footnotesize \emph{Note.} OSCS coalitions of size
$\ge 2$. Payoff vectors are ordered by player index within each coalition
and rounded componentwise to integers, so a displayed vector may not sum
exactly to $V(c)$; efficiency, Core membership, and least-core values are
computed from the unrounded allocations. $\varepsilon$ is the least-core value
($\varepsilon \le 0$ iff the Core is nonempty); $\phi \in$ Core indicates
whether the Shapley value satisfies all Core inequalities.}
\end{table}
\egroup






  



\end{document}